\documentclass[11pt]{article}
\usepackage{a4wide}
\usepackage{graphicx}
\usepackage{epstopdf}
\usepackage{color}
\usepackage{float}
\usepackage{subcaption}
\usepackage[utf8]{inputenc}
\usepackage[english]{babel}
\usepackage[pdftex,dvipsnames,table]{xcolor}
\usepackage{hyperref}
\usepackage{amssymb,amsmath, amsthm}
\usepackage{cleveref}
\usepackage{autonum}
\usepackage{enumitem}
\usepackage[labelfont=bf,font={small},margin=1cm]{caption}
\usepackage{mathtools}
\usepackage{xargs}
\usepackage[numbers]{natbib}
\usepackage{tikz}
\usepackage{geometry}
\usepackage{multirow}

\newcommand{\Acal}[0]{\ensuremath{{\mathcal A}}}

\newcommand{\Hcal}[0]{\ensuremath{{\mathcal H}}}
\newcommand{\Ical}[0]{\ensuremath{{\mathcal I}}}
\newcommand{\Jcal}[0]{\ensuremath{{\mathcal J}}}

\newcommand{\Mcal}[0]{\ensuremath{{\mathcal M}}}

\newcommand{\Scal}[0]{\ensuremath{{\mathcal S}}}

\newcommand{\Xcal}[0]{\ensuremath{{\mathcal X}}}

\newcommand{\eN}[0]{\ensuremath{ \mathbb N}}
\newcommand{\Zed}[0]{\ensuremath{ \mathbb Z}}

\newcommand{\Pee}[0]{\ensuremath{{\mathbb P}}}

\newcommand{\Ee}[0]{\ensuremath{{\mathbb E}}}

\newcommand{\isd}[0]{\hspace{.2ex} \raisebox{-.1ex}{$=$} \hspace{-1.5ex} 
\raisebox{1ex}{{$\scriptstyle d$}} \hspace{.8ex} }

 \newcommand{\eps}{\varepsilon}
\newcommand{\dtv}{d_{\on{TV}}}

\DeclareMathOperator{\Bi}{Bi}

\DeclareMathOperator{\TOTO}{\an{TOTO}}
\DeclareMathOperator{\TOOB}{\an{TOOB}}

\definecolor{orange}{RGB}{255,127,0}
\definecolor{pink}{RGB}{255,150,150}

\newcommand{\FS}[1]{{{\BR{{}#1}}}}
\newcommand{\TV}[1]{{{\BL{#1}}}}

\DeclareMathOperator{\inv}{inv}
\DeclareMathOperator{\Mallows}{Mallows}
\DeclareMathOperator{\rk}{rk}
\DeclareMathOperator{\Geo}{Geo}
\DeclareMathOperator{\TGeo}{TruncGeo}
\DeclareMathOperator{\Unif}{Unif}

\newcommand{\piempty}[0]{\ensuremath{\pi_{\text{empty}}}}
\newcommand{\sempty}[0]{\ensuremath{s_{\text{empty}}}}

\renewcommand{\phi}{\varphi}
\renewcommand{\succ}{\an{succ}}

\newtheorem{thm}{Theorem}
\numberwithin{thm}{section} 
\numberwithin{figure}{section}
\newtheorem{lem}[thm]{Lemma}
\newtheorem{cor}[thm]{Corollary}
\newtheorem{prop}[thm]{Proposition}

\newtheorem*{claim*}{Claim}
\newtheorem*{con*}{Conjecture}

\theoremstyle{definition}

\newtheorem{rem}[thm]{Remark}

\newtheorem*{exa*}{Example}

\newcommand{\floor}[1]{\left\lfloor #1 \right\rfloor}
\newcommand{\ceil}[1]{\left\lceil #1 \right\rceil}
\newcommand{\on}[1]{\operatorname{#1}}
\newcommand{\an}[1]{\mathsf{#1}}
\renewcommand{\mid}{:}
\newcommandx{\bm}[1]{ \begin{bmatrix} #1 \end{bmatrix}  }
\newcommand{\be}{\coloneqq}

\newcommand{\ind}[1]{\mathbf{1}_{#1}}
\renewcommand{\vec}[1]{\underline{#1}}

\renewcommand{\epsilon}{\varepsilon} 

\let\emptyset\varnothing

   \def\bE{{\mathbb E}}

   \def\bN{{\mathbb N}}   
\def\bP{{\mathbb P}}      \def\bR{{\mathbb R}}

   \def\bZ{{\mathbb Z}}

   \def\cE{{\mathcal E}}   
      \def\cI{{\mathcal I}}
      \def\cL{{\mathcal L}}
\def\cM{{\mathcal M}}   \def\cN{{\mathcal N}}   
      
\def\cS{{\mathcal S}}

\def\fA{{\mathfrak A}}   \def\fB{{\mathfrak B}}

\newenvironment{proofof}[1]{\vspace{1ex}\noindent{\bf Proof of #1:}}{\hspace*{\fill}$\blacksquare$\vspace{1ex}}

\title{Logical limit laws for Mallows random permutations}
\author{Tobias M\"uller\thanks{Bernoulli Institute for Mathematics, CS and AI, Groningen University, {\tt tobias.muller@rug.nl}. 
Supported in part by the Netherlands Organisation for Scientific Research (NWO) under project nos 612.001.409 and 639.032.529.} 
\and Fiona Skerman\thanks{Department of Mathematics, Uppsala University, {\tt fiona.skerman@math.uu.se}.
Supported by AI4Research at Uppsala University and by the Wallenberg AI, Autonomous Systems and Software Program (WASP).} \and Teun W.~Verstraaten\thanks{Bernoulli Institute for Mathematics, CS and AI, Groningen University, 
{\tt t.w.verstraaten@rug.nl}.}}
\date{\today}

\renewcommand{\FS}[1]{#1}
\renewcommand{\TV}[1]{#1}

\begin{document}
\maketitle

\begin{abstract}
    A random permutation $\Pi_n$ of $\{1,\dots,n\}$ follows the $\Mallows(n,q)$ distribution with 
    parameter $q>0$ if $\bP \left( \Pi_n = \pi \right)$ is proportional to $q^{\inv(\pi)}$ for all $\pi$. 
    Here $\inv(\pi) \be |\{ i<j : \pi(i)> \pi(j) \}|$ denotes the number of inversions of $\pi$. 
    We consider properties of permutations that can be expressed by the sentences of two different logical languages.
    Namely, \emph{the theory of one bijection} ($\TOOB$), which describes 
    permutations via a single binary relation, and \emph{the theory of two orders} ($\TOTO$), 
    where we describe permutations by two total orders. 
    We say that the {\em convergence law} holds with respect to 
    one of these languages if, for every sentence $\phi$ in the language, the probability $\bP (\Pi_n\text{ satisfies } \phi)$
    converges to a limit as $n\to\infty$. If moreover that limit is $\in\{0,1\}$ for all sentences, then the {\em zero--one law} holds.
    
    We will show that with respect to $\TOOB$ the $\Mallows(n,q)$ distribution satisfies the zero--one law when $0<q<1$ is 
    fixed, and for fixed $q>1$ the convergence law fails.  (In the case when $q=1$ Compton~\cite{Compton89II} 
    has shown the convergence law holds but not the 
    zero--one law.)
    
    We will prove that with respect to $\TOTO$ the $\Mallows(n,q)$ distribution satisfies the convergence law 
    but not the zero--one law for any fixed $q\neq 1$, and that if $q=q(n)$ satisfies 
    $1- 1/\log^*n < q < 1 + 1/\log^*n$ then $\Mallows(n,q)$ fails the convergence law. 
    Here $\log^*$ denotes the discrete inverse of the tower function.
\end{abstract}

\section{Introduction}

Throughout the paper, we denote by $[n] := \{1,\dots,n\}$ the first $n$ positive integers and 
by  $S_n$ the set of permutations on $[n]$. 
A pair $i,j\in [n]$ is an inversion 
of the permutation $\pi \in S_n$ if $i<j$ and $\pi(i) > \pi(j)$. 
We denote by $\inv (\pi)$ the number of inversions of a permutation $\pi$. 

For $n \in \eN$ and $q>0$, the Mallows distribution $\Mallows(n,q)$ 
samples a random element $\Pi_n$ of $S_n$ such that for all $\pi\in S_n$ we have

\begin{equation}\label{eq:Mallowsdef} 
    \Pee( \Pi_n = \pi ) = \frac{q^{\inv(\pi)}}{\sum_{\sigma\in S_n} q^{\inv(\sigma)}}. 
\end{equation}

\noindent
In particular, if $q=1$ then the Mallows distribution is simply the uniform distribution on $S_n$.

The Mallows distribution was first introduced by C.L. Mallows~\cite{Mallows} in 1957 
in the context of statistical ranking theory. 
It has since been studied in connection with a diverse range of topics, including 
mixing times of Markov chains~\cite{Benjamini2005, Diaconis2000},
finitely dependent colorings of the integers~\cite{HolroydHutchcroftLevy2020}, 
stable matchings~\cite{AngelEtAl}, random binary search trees~\cite{Louigi}, learning theory~\cite{BravermanMossel,Tang19}, 
 $q$-analogs of exchangeability~\cite{GnedinOlshanski2010,Gnedin}, determinantal point processes~\cite{BorodinDiaconisFulman2010},
statistical physics~\cite{Starr2009,StarrWalters2018}, genomics~\cite{FangEtAl2021} and
random graphs with tunable properties~\cite{enright2021tangled}.

Properties of the Mallows distribution that have been investigated include pattern 
avoidance~\cite{CraneDesalvo2017, CraneDesalvoElizalde2018, Pinsky2021}, the number of descents~\cite{HeDescents}, 
the longest monotone subsequence~\cite{basu2016limit,BhatnagarPeled2015, MuellerStarr2013}, the longest common 
subsequence of two Mallows permutations~\cite{Jin2019} and the cycle structure~\cite{Peled,MullerVerstraaten}.

In the present paper we will study the Mallows distribution from the perspective of first order logic.
Given a sequence of random permutations $(\Pi_n)_{n\geq 1}$, we say that the \emph{convergence law} holds
with respect to some fixed logical language describing permutations
if the limit $\lim_{n\to\infty} \bP \left( \Pi_n\models \phi \right) $ exists, for every sentence $\phi$ in the language. 
Here and in the rest of the paper the notation $\pi \models \phi$ denotes that the sentence $\phi$ holds for the permutation 
$\pi$.
If this limit is either $0$ or $1$ for all such $\phi$ then we say that the \emph{zero--one law} holds. 
%
Following \cite{Albert}, we will consider two different logical languages for permutations: 
the Theory of One Bijection ($\TOOB$) and the Theory of Two Orderings ($\TOTO$). 
Here we give an informal overview of them. More precise definitions follow in Section \ref{sec:logic}. 

In $\TOOB$, we can express a property of a permutation 
using variables representing the elements of the domain of the permutation, the usual quantifiers $\exists, \forall$, 
the usual logical connectives $\vee,\wedge,\neg,\dots$, parantheses and two binary relations $=,R$. 
Here $=$ has the usual meaning ($x=y$ denotes that the variables $x,y$ represent the same 
element of the domain of the permutation) and $R(x,y)$ holds if and only if $\pi(x)=y$.
We are for instance able to query in $\TOOB$ if a permutation has a fixed point by

$$ \exists x : R(x,x). $$

As is shown \cite{Albert} (Proposition 3 and the comment following it), we cannot express by a $\TOOB$
sentence whether or not a permutation contains the pattern $231$. 
(A permutation $\pi$ contains the pattern $231$ if there exist $i<j<k$ such that $\pi(k)<\pi(i)<\pi(j)$.)

The logical language $\TOTO$ is constructed 
similarly to $\TOOB$. 
Instead of the relation $R$ there now are two relations $<_1, <_2$.
The relation $<_1$ represents the usual linear order on the domain $[n]$ of $\pi \in S_n$ and
the relation $<_2$ represents the usual linear order of the images $\pi(1),\dots,\pi(n)$. 
That is, $x<_1y$ if and only if $x<y$, while $x <_2 y$ if and only if $\pi(x) < \pi(y)$. 
In $\TOTO$ we can for instance express that a permutation contains the pattern $231$, via

$$ \exists x, y, z : (x <_1 y) \wedge (y <_1 z) \wedge (z <_2 x) \wedge (x <_2 y).
$$ 

\noindent
See Figure~\ref{fig:231} for an illustration.  
We can also express that $\pi(1)<\pi(2)$ by: 

     $$ \begin{array}{l} 
     \exists x, y : (x <_1 y) \wedge  (x <_2 y) \wedge \left( \forall z : (z<_1y) \rightarrow (z=x)\right). 
     \end{array} $$

\noindent
See Section 3 of \cite{Albert} for generalizations of pattern containment that can be expressed in $\TOTO$. 
On the other hand, in $\TOTO$ we cannot express whether or not a permutation has a fixed point, as shown 
in Corollary~27 of \cite{Albert}.

\begin{figure}
    \begin{subfigure}[t]{0.32\textwidth}
        \captionsetup{width=.9\linewidth}
        \includegraphics[scale=0.5]{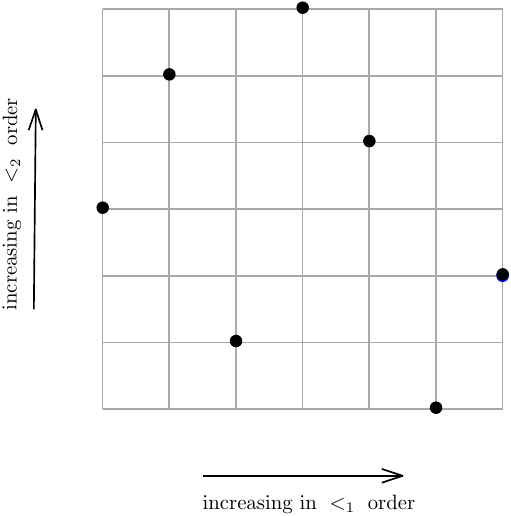}
        \caption{the initial three elements induce the pattern 231. } 
    \end{subfigure}
    \begin{subfigure}[t]{0.32\textwidth}
        \captionsetup{width=.9\linewidth}
        \includegraphics[scale=0.5]{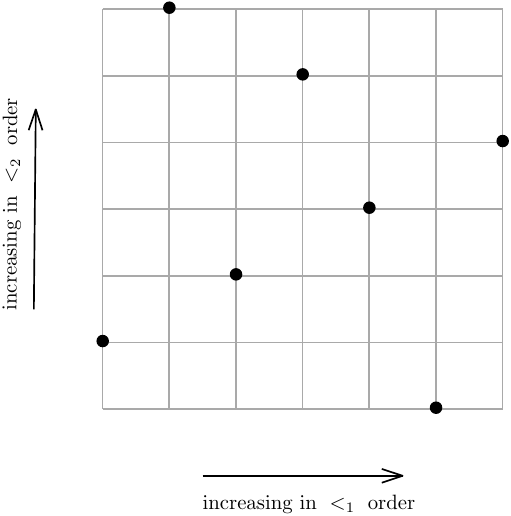}
        \caption{contains 231 but not on the initial three elements.} 
    \end{subfigure}
    \begin{subfigure}[t]{0.32\textwidth}
        \captionsetup{width=.9\linewidth}
        \includegraphics[scale=0.5]{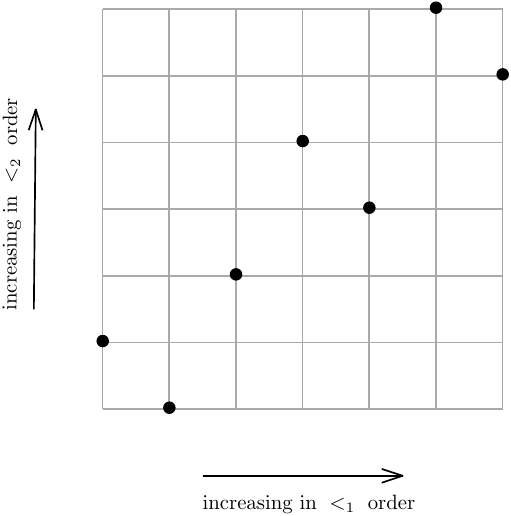}
        \caption{does not contain 231.} 
    \end{subfigure}
    \caption{   \label{fig:231}}
\end{figure}

 It is not hard to see that for a uniform random permutation the probability 
 that $\pi(1)<\pi(2)$ equals $1/2$. 
 So in particular, for uniform permutations $\TOTO$  does not satisfy a zero-one law. See also 
 the remark following Question 1 in~\cite{Albert}.
 What is more, in $\TOTO$ uniform permutations do not even satisfy the convergence law as was 
 first shown by Foy and Woods~\cite{Foy}. 
Let us also mention two very recent results on logical limit laws for random permutations. 
In \cite{albert2022convergence} it is shown the uniform distribution on the set of permutations in $S_n$ that avoid 231 
admits a convergence law in $\TOTO$, and in~\cite{braunfeld2021logical} it is shown the same result holds 
for the uniform distribution on the class of layered permutations.

Both $\TOOB$ and $\TOTO$ fall under the umbrella of {\em first order} logical languages, as
they only allow quantification over the elements of the domain. In contrast, second order logic
also allows us to quantify over relations. 
Second order logic is much more powerful than first order logic.
It is for instance possible to express in second order logic the property that the 
domain has an even number of elements. In particular, the convergence law will fail for 
second order logic, for trivial reasons.


The study of first order properties of random permutations is part of a larger theory of 
first order properties of random discrete 
structures. See for instance the monograph~\cite{Spencerstrange}. 
Examples of structures for which the first order properties have been studied include the Erd\H{o}s-R\'enyi random 
graph (see e.g.~\cite{PurpleBook},\cite{Spencerstrange}), Galton-Watson 
trees~\cite{podder2017first}, bounded degree graphs~\cite{koponen2012random}, random graphs 
from minor-closed classes~\cite{heinig2018logical}, random perfect graphs~\cite{muller2018first}, random 
regular graphs~\cite{haber2010logic} and the very recent result on bounded degree uniform 
attachment graphs \cite{malyshkin2022logical}. 

Let us also mention some work on {\em random orders} that is closely related to the topic of the present paper. 
The logic of random orders was introduced in \cite{winkler1985random} and has a single binary 
relation $<$. 
To sample a ``$k$-dimensional'' random order on $[n]$ pick $k$ random 
orderings $<_1, \ldots, <_k$ on $[n]$ and  set $x<y$ if $x <_i y$ in each of the orders $i=1, \ldots, k$. 

Non-convergence was proven for two-dimensional random orders in~\cite{spencer1990nonconvergence} 
by constructing a first--order sentence $\phi$ for which no limit probability exists.
Replacing every occurrence of $<$ in this sentence by $(x <_1 y) \land (x <_2 y)$
yields an alternative proof of non-convergence also in $\TOTO$ for uniformly random permutations. 
Many properties of random orders are known \cite{albert1989random, bollobas1988longest}, see \cite{brightwell1993models} 
and \cite{brightwell2016mathematics} for surveys of different models of random partial 
orders and also their relation to theories of spacetime~\cite{brightwell2016mathematics}. 

%
An appeal of $\TOTO$ is that it allows one to express pattern containment. 
Permutation patterns arise naturally in statistics. Let $(X,Y)$ be drawn from a continuous distribution 
on $\mathbb{R}^2$ and suppose we have $n$ random samples $(X_i, Y_i)$. Many 
statistical tests rely only on the relative ordering in the two dimensions, i.e.\ on the permutation induced by 
the $n$ points. 
For example the Kendall rank correlation coefficient is 
$\tau = 1-2\inv(\pi)/\binom{n}{2}$ and indeed there is a test for independence of $X$ and $Y$ using only counts of 
permutations of length 4~\cite{kral2013quasirandom, yanagimoto1970measures}. See \cite{even2021counting} for a combinatorial 
account of the use of permutation patterns in statistics. 

\subsection{Main results}

The main results of this paper are collected in the following two theorems:

\begin{thm}\label{thm:logical_limits}
    In $\TOOB$ the following holds:
    \begin{enumerate}
        \item\label{item:logi_convergence}{\bf[Compton,\cite{Compton89II}]} For $q=1$ the $\Mallows(n,q)$ distribution satisfies the convergence law but not the zero--one law;
        \item  \label{item:logi_zero_one} If $q<1$ is fixed then $\Mallows(n,q)$ satisfies the zero--one law;
        \item   \label{item:logi_non_convergence}  If $q>1$ is fixed then $\Mallows(n,q)$ does not satisfy the convergence law.
    \end{enumerate}
\end{thm}

\noindent
(Part~\ref{item:logi_convergence} of Theorem~\ref{thm:logical_limits} was already shown by Compton~\cite{Compton89II}
in 1989, who in fact proved a much stronger statement.)

The function $\log^* n$ equals the number of times we need to iterate the base two logarithm to reach a number 
below one (starting from $n \in\eN$).

\begin{thm}\label{thm:toto}
    In $\TOTO$ the following holds:
    \begin{enumerate}
        \item\label{part:conv_toto} If $q \neq 1$ is fixed, then the convergence law holds for $\Mallows(n,q)$ but not the zero--one law;
        \item\label{part:toto_non_convergence} If $q=q(n)$ satisfies 
        $1- \frac{1}{\log^*n} < q < 1 + \frac{1}{\log^*n}$, then the $\Mallows(n,q)$ distribution does not satisfy the convergence law. 
    \end{enumerate}
\end{thm}

Part~\ref{part:toto_non_convergence} of Theorem~\ref{thm:toto} extends the previously mentioned result 
of Foy and Woods~\cite{Foy}.

The $\log^* n$ term is a very slowly growing function. 
As will be clear from the proof of Theorem~\ref{thm:toto}~Part~\ref{part:toto_non_convergence}, it is certainly 
not best possible and can be replaced by even more slowly growing functions, such as $\log^{**} n$ defined in 
Section~\ref{subsec:definitions}, with little effort.

We will in fact show that there exists a single $\TOTO$ formula $\phi$ such that for all 
sequences $q=q(n)$ satisfying $1-\frac{1}{\log^*n} <q < 1 + \frac{1}{\log^*n}$, the quantity 
$\bP \left( \Pi_n \models \phi \right) $ does not have a limit as $n\to\infty$.

\subsection{Organization of the paper}

To aid the reader, in the next section we give an informal, intuitive synopsis of the proofs of our main results.
In Section~\ref{sec:prelim} we collect definitions and results from the literature that we will 
rely on in our proofs. We also provide some short proofs of relatively elementary observations that we need
but are not readily available in the literature. 
In Section~\ref{sec:toob} we give the full proof of Theorem~\ref{thm:logical_limits}. 
In Section~\ref{sec:proof_conv_toto} we prove Part~\ref{part:conv_toto} of Theorem~\ref{thm:toto}.
As a preparation for the full proof of Part~\ref{part:toto_non_convergence} of Theorem~\ref{thm:toto}, 
we first give a proof in the special case when $q=1$, in Section~\ref{sec:totoqis1}.
In Section~\ref{sec:totoallq}, we then extend the result 
to the full range of $q$ stated in Part~\ref{part:toto_non_convergence} of Theorem~\ref{thm:toto}.

Finally, in Section~\ref{sec:discuss}, we offer some additional discussion and some suggestions for further work.

\section{Overview of proof methods}

\vspace{1ex}

\noindent\textbf{Proof Sketch of Theorem~\ref{thm:logical_limits}:} 
The proof of Theorem~\ref{thm:logical_limits} will rely on the following observations: 
Firstly, given any sentence $\phi \in \TOOB$ and $\pi \in S_n$ the vector \FS{of cycle counts} 
$(C_1(\pi_n),\ldots, C_n(\pi_n))$ completely determines whether or not $\pi_n\models \phi$ 
\FS{(where $C_i(\pi_n)$ denotes the number of $i$--cycles in $\pi_n$)}. 
Secondly, it follows from the Hanf-Locality Theorem for bounded degree structures (stated as Theorem~\ref{thm:hanf_general} below) 
that for any fixed $\phi \in \TOOB$ there exists an $h \in \bN$ such that for any $t \in \bN$ and $n,m\geq h$, the 
sentence $\phi$ cannot distinguish between the disjoint union of $n$ cycles of length $t$ and the disjoint 
union of $m$ cycles of length $t$. 
Moreover, this $h$ can be selected such that $\phi$ additionally cannot distinguish between two cycles 
both of length at least $h$.

For $\Pi_n\sim\Mallows(n,q)$ with $0<q<1$, we use a result given in \cite{MullerVerstraaten} by Jimmy He together with the first and last author of the current paper. They show that there are positive constants $m_1,m_2,\ldots$ depending on $q$ such that $\frac{1}{\sqrt{n}} \left(C_1(\Pi_n)-m_1 n, \ldots , C_\ell (\Pi_n) - m_{\ell} n \right)$ converges in distribution to a multivariate normal with zero mean. 
This implies in particular that $\Pi_n$ will contain more than $h$ cycles of each of the lengths $1,\ldots, h$ with probability tending to 1, giving the 
zero--one law for $0<q<1$.

For fixed $q>1$, another result (stated as Theorem~\ref{thm:odd_cycles} below) from \cite{MullerVerstraaten} implies 
the existence of a value $x=x(q)$ such that for $\Pi_n\sim\Mallows(n,q)$ the quantity 
$\bP (  C_1(\Pi_n)\geq x  ) $ does not have a limit as $n\to \infty$. 
This property can be queried by a $\TOOB$ sentence, establishing that $\Mallows(n,q)$ does not satisfy the convergence law 
w.r.t. $\TOOB$ for fixed $q >1$.

\vspace{1ex}

\noindent\textbf{Proof Sketch of Theorem~\ref{thm:toto}~Part~\ref{part:conv_toto}:} 
The proof of Theorem~\ref{thm:toto}~Part~\ref{part:conv_toto} is inspired by the approach taken by Lynch 
in \cite{Lynch} to show a convergence law for random strings having certain letter distributions. 
We consider some $\TOTO$ sentence $\phi$ having quantifier depth $d$ (quantifier depth will be defined in the next section). 
Writing $\pi \equiv_d \sigma$ if $\pi$ and $\sigma$ agree on all $\TOTO$ sentences of depth at most $d$, we 
note that $\equiv_d$ is an equivalence relation with only finitely many equivalence 
classes (Theorem \ref{lem:finite_classes}). 
We now rely on a sampling procedure for generating a sequence $\Pi_1, \Pi_2,\ldots$ of $\Mallows(n,q)$ distributed permutations
from a sequence $Z_1,Z_2,\ldots$ of independent $\Geo(1-q)$ random variables
(The details of the construction will be given in Section \ref{subsec:constr_trunc}.) 
Taking a dynamic viewpoint, we can imagine exposing the values $Z_1,Z_2,\ldots$ one by one and following the 
sequence of permutations that arises. 
Foregoing complicating details, we can define in this manner a Markov chain on a countably infinite state 
space that follows the equivalence class under $\equiv_d$ of this sequence of permutations. 
We show that we can the state space have finitely many recurrent classes, 
each of them irreducible, aperiodic and positive recurrent, and that the chain will hit such a 
class a.s. 
Standard results on the convergence of Markov chains will give the convergence of the quantity $\bP (  \Pi_n \models \phi    )  $ as $n\to\infty$.

\vspace{1ex}

\noindent\textbf{Proof Sketch of Theorem~\ref{thm:toto}~Part~\ref{part:toto_non_convergence}:} 
The proof of Theorem~\ref{thm:toto}~Part~\ref{part:toto_non_convergence} proceeds in several steps. 
We start by determining a sentence $\phi\in\TOTO$ such that for $\Pi_n\sim\Mallows(n,1)$ we have 
\begin{equation}\label{eq:cases_prb}
    \bP( \Pi_n \models \phi ) = 
    \begin{cases}
        1 - O( n^{-100} ) & \text{if } \log^{**}\log^{**}n \text{ is even}, \\
        O( n^{-100} ) &  \text{if } \log^{**}\log^{**}n \text{ is odd}.
    \end{cases},  
\end{equation}
for $n$ in a certain set that contains both an infinite number of odd and 
an infinite number of even values. 
We do not mention the additional technical conditions on the values of $n$ in this sketch. 

We stress that for the $q=1$ case both Foy and Woods~\cite{Foy} 
and Spencer~\cite{spencer1990nonconvergence} have constructed first order sentences 
whose probability of holding does not converge with $n$. 
We have however not found a way to use their proofs
``off the shelf'' as an ingredient to prove our Theorem~\ref{thm:toto}~Part~\ref{part:toto_non_convergence}.

We will show~\eqref{eq:cases_prb} by associating pairs of intervals in $\Pi_n\sim \Mallows(n,1)$ 
to directed graph structures having $N=\Theta(\log \log n)$ vertices. See Figure~\ref{fig:cherry} 
for an example permutation and pair of intervals corresponding to a directed cherry. The \FS{directed graph corresponding to the permutation} is random, and the number of pairs of intervals is sufficiently high so that with high probability we will be likely to find any directed graph on $N$ vertices. We then adapt the method developed by Spencer and Shelah in \cite{ShelahSpencer} of using graphs to model arithmetic on sets to find a sentence about these graphs that oscillates between being true and false depending on $N$, which will provide the dependency of $\bP \left( \Pi_n \models\phi \right)$ on the parity of $\log ^{**}\log^{**} n$.

We then show that if $q=1\pm O(n^{-4})$ then the total variation distance between $\Mallows(n,q)$ and $\Mallows(n,1)$ is 
$O(n^{-2})$. This allows us to extend the previous result from $q=1$ to $q=1-O(n^{-4})$, replacing the $O(n^{-100})$ terms 
in~\eqref{eq:cases_prb} by $O(n^{-2})$. We extend the result several more times but now we will have to work a little harder 
each time. 

Next, we consider the range $q = 1\pm O(1/n)$. We are able to define (in $\TOTO$)
a random variable $Z_n$ such that $n^{.1} \leq Z_n \leq n^{.2}$ with an appropriately high 
probability. Then, we restrict attention to the permutation induced by $\Pi_n$ (using the rank function 
defined in the next section) on 
$\{1,\dots,Z_n\}$. 
Rewriting $q$ in terms of $Z_n$, we have $q = 1 \pm o( |Z_n|^{-4} )$, which will allow us to exploit the previous 
case.

We subsequently extend again, now to the (asymmetric) range $1-1/\log^* n < q \leq 1 + O(1/n)$.
Here we 
consider the permutation induced on $\{1,\dots,\Pi^{-1}_n(1)-1\}$. 
It will turn out that the quantity $N := \Pi^{-1}_n(1)-1$ is of the order 
$\Theta\left( \min( n, 1/|1-q| \right)$, in an appropriate probabilistic sense. Rewriting $q$ in terms of $N$ tells us 
$q = 1 \pm O(1/N)$, which will allow us to exploit the previous case.

Finally, we will extend to the full range. 
Here the key will be that the ($\TOTO$-expressible) property that $\Pi_n(1)<\Pi_n(n)$
is able to distinguish, in an appropriate probabilistic sense, between the 
cases $q < 1-\Omega(1/n)$ and $q > 1 + \Omega(1/n)$.

\section{Notation and preliminaries\label{sec:prelim}}

We use the notation $\bN =\{1,2,\ldots\}$ and $\bZ_{\geq 0} = \{0,1,\ldots \}$.

\subsection{Requisites from probability theory\label{subsec:req_prob}}

We collect some results from probability theory that we will use in the sequel. 

\begin{thm}[Chernoff's inequality]\label{thm:chernoff}
    Let $X_1,\ldots, X_n$ be independent such that $0\leq X_i \leq 1$ for $i\in [n]$. Define $S_n \be X_1+\ldots + X_n$ and $\mu \be \bE S_n$. Let $0< \epsilon < 1$, then
    \begin{equation}
        \bP \left( S_n \geq (1  + \epsilon) \mu \right)  \leq \exp \left\{ - \frac{\mu \epsilon^2}{3} \right\},
    \end{equation}
         \begin{equation}
        \bP \left( S_n \leq (1  - \epsilon) \mu \right)  \leq \exp \left\{ - \frac{\mu \epsilon^2}{2} \right\}.
    \end{equation}
\end{thm}

%

We denote by $\Bi(n,p)$ the binomial distribution. We will need the following crude bound:
\begin{lem}\label{lem::binom_max}
    Let $X\sim \Bi(n, k/n)$ for $0\leq k \leq n$. Then $\bP \left( X = k \right) \geq 1/(n+1)$. 
\end{lem}

\begin{proof}
    If $k=n$ then the result is clear, so assume that $k\neq n$. Setting $p=k/n$ we have for all $1\leq j \leq  n$ that
    \begin{align}
        \frac{\bP \left(X= j \right)}{ \bP \left(X= j-1\right) } &= \frac{ \binom{n}{j} (1-p)^{n-j}p^{j}}{ \binom{n}{j-1} (1-p)^{n-j+1}p^{j-1} }=\frac{ (n-j+1) p}{j (1-p)}.
    \end{align}
    The right--hand side above is less than one if and only if 
    \begin{align}
        j-jp > (n-j+1)p \quad\iff\quad j > \frac{(n+1)k}{n} \quad\iff \quad j \geq k + 1,
    \end{align}
    the last equivalence holding due to $j$ being an integer and $k\neq n$. Thus $\bP \left( X = j\right)$ is maximized at $j=k$. As there are only $n+1$ possible outcomes for $X$ we therefore must have $\bP \left(X=k\right) \geq 1/(n+1)$.
\end{proof}

Given two discrete probability distributions $\mu_1$ and $\mu_2$ on a countable set $\Omega$, their total variation distance is defined as

\begin{equation}\label{eq:def_dtv}
   \dtv(\mu_1,\mu_2) = \max_{ A \subseteq \Omega  }| \mu_1(A) - \mu_2(A) |.  
\end{equation}

\noindent
This can be expressed alternatively as 

\begin{equation}\label{eq:dtvalt} 
\dtv(\mu_1,\mu_2) = \frac12 \sum_{x\in \Omega} |  \mu_1(x) - \mu_2(x) |
= \sum_{x : \mu_1(x)>\mu_2(x) } \mu_1(x)-\mu_2(x). 
\end{equation}

\noindent
(See for instance Proposition 4.2 in \cite{Peres} for a proof.) As is common, we will interchangeably use the notation $\dtv(X,Y) := \dtv(\mu,\nu)$ if $X\sim\mu$ and $Y\sim\nu$.

A {\em coupling} of two probability measures $\mu, \nu$ is a joint probability measure for a pair of random variables $(X,Y)$ satisfying $X\isd\mu, Y\isd\nu$. We will also speak of a coupling of $X,Y$ as being a probability space for $(X',Y')$ with $X'\isd X, Y'\isd Y$. 

\begin{lem}\label{lem:coupling}
    Let $\mu$ and $\nu$ be two probability distributions on the same countable set $\Omega$. Then
    \begin{equation}
    \dtv(\mu, \nu) = \min \{ \bP \left(  X\neq Y\right) \mid (X,Y) \text{ is a coupling of $\mu$ and $\nu$}  \}.
    \end{equation}
\end{lem}

\noindent
(See for instance~\cite{Peres}, Proposition 4.7 and Remark 4.8 for a proof.)

For $n\in \bN$ and $p\in (0,1)$ we define the $\TGeo(n,p)$ distribution where $Y\sim \TGeo(n,p)$ means 
\begin{equation}\label{eq:mass_geometric}
   \bP( Y = k ) = \frac{p(1-p)^{k-1}}{1-(1-p)^n},\qquad k \in \{1,\ldots, n\}.
\end{equation}
Observe that, setting $X \sim\Geo(p)$, the probability mass in \eqref{eq:mass_geometric} 
is exactly equal to $\bP \left( X = k\, |\, X \leq n \right) $. 
Therefore we shall refer to the $\TGeo(n,p)$ as the truncated geometric distribution. 

\subsubsection{Markov chains\label{sec:markov}}

We give a brief overview of some of the concepts related to Markov chains which will be used in 
Section \ref{sec:proof_conv_toto}. The definitions and results that follow are widely known and 
can be found in most books on Markov chains or more general stochastic processes, see for 
instance \cite{Peres}, \cite{Grimmett} or \cite{Markov}.

Let $X_n$ be a Markov chain with state space $\cS$. Given $i,j\in \cS$, if 
$\bP \left( X_n = j\, |\, X_0 = i \right) > 0$ for some $n\in \bZ_{\geq 0}$ we say that 
\emph{$j$ is reachable from $i$}. 
We denote this by $i\to j$. If $i\to j$ and $j\to i$ then we say that \emph{the states $i$ and $j$ communicate} 
and we write $i\leftrightarrow j$. The relation $\leftrightarrow$ is an equivalence relation, the equivalence classes under $\leftrightarrow$ will be called \emph{communicating classes}. For every state $i\in \cS$ we have $i\leftrightarrow i$. In the case that the entire state space is a single communicating class we call the Markov chain \emph{irreducible}. The evolution of a Markov chain depends on the initial distribution over the states. 
We use the notations $\bP_i (\cdot)$ and $\bE_i (\cdot)$ for probabilities and expectations
with respect to the Markov chain started in state $X_0=i$. If $X_0$ is distributed according to some distribution 
$\mu_0$ over $\cS$ then we write $\bP_{\mu_0}$ and $\bE _{\mu_0}$. 

For any $i \in \cS$ we define the {\em hitting time} 
$\tau_i := \inf\{ n\geq 0 : X_n=i \}$ 
to be the first time that the Markov chain is in state $i$, and we 
define the {\em first visit time}
$\tau^+_i := \inf\{ n\geq 1 : X_n= i \}$ to be the first time after zero that the Markov chain is in state $i$.
If $\bP_i ( \tau^+ _i < \infty ) = 1$ then we say that the state $i$ is \emph{recurrent}, otherwise we say that 
it is \emph{transient}. 
If additionally $\bE _i \tau^+_i <\infty$ then we say that $i$ is \emph{positive recurrent}. 
If $\cS$ is finite then recurrence implies positive recurrence. If $i$ is transient, then all $j$ in the communicating 
class containing $i$ are transient, and similarly if $i$ is (positive) recurrent. 
Thus we will talk about positive recurrent classes, etc. 
If $X_n$ is in a recurrent class for some $n$, then $X_{n+k}$ is 
contained in this class for all $k \geq 0$. 
The {\em period} of a state is the greatest common divisor of the set $\{n\geq 1 : \Pee_i(X_n= i)> 0\}$.
A state is {\em aperiodic} if the period equals one. A communicating class or chain is called aperiodic if all of its states 
are.
Proofs for the following two results are given in Theorem~4 and Lemma~5, respectively, in Section 6.3 of \cite{Grimmett}.

\begin{thm}\label{thm:partition_state_space}
    The state space $\cS$ of a Markov chain can be partitioned uniquely as
    \begin{equation}
        \cS = T \cup C_1 \cup C_2 \cup \ldots
    \end{equation}
    where $T$ is the set of transient states, and the $C_i$ are recurrent communicating classes. 
\end{thm}

\begin{lem}\label{lem:at_least_one_recurrent}
    If $\cS$ is finite, then at least one state is recurrent and all recurrent states are positive recurrent.
\end{lem}

\begin{lem}\label{lem:hit_recurrent}
    If $\cS$ is finite then with probability $1$ the chain $X_n$ will be in a recurrent state for some $n\geq 0$. 
\end{lem}

\begin{proof}
    Let $\cS = T \cup C_1 \cup C_2 \cup \ldots$ be as in Theorem~\ref{thm:partition_state_space}. 
    For each of the finitely many $i\in T$ there is a positive probability that the chain moves from $i$ to a 
    state in $\cS \setminus T$ in at most $|T|$ steps. Then $\bP \left( X_n \in T \right) \leq K c^n$ for 
    some constants $K>0$ and $0<c<1$. 
    By the Borel--Cantelli Lemma, $X_n$ is in $T$ for only finitely many $n$ with probability one, so 
    that it must visit a recurrent state. 
\end{proof}

The following theorem is the Markov chain convergence theorem, see e.g.~\cite{DoucMarkov} Theorem 7.6.4 for a proof (note that in \cite{DoucMarkov}, what we call irreducibility is called strong irreducibility).

\begin{thm}[Markov chain convergence theorem]\label{thm:markov_convergence}
    Consider an irreducible, aperiodic and positive recurrent Markov chain 
    on a countable state space $\cS$.
    There exists a unique invariant probability measure $\pi$ over $\cS$ such that for every probability measure $\mu_0$ over 
    $\cS$:
    \begin{equation}
        \dtv(X_n,\pi) \xrightarrow[n\to\infty]{} 0
    \end{equation}
    where $X_0 \isd \mu_0$ and $(X_n)_{n\geq 0}$ denotes the evolving sequence of states of the Markov chain. 
\end{thm}

We will need the following result in Section~\ref{sec:proof_conv_toto}. 
It is alluded to in many texts on Markov chains when discussing the convergence behavior
of a Markov chains that are not irreducible. We have not found an appropriate reference however. 
So we present a short proof. 

\begin{lem}\label{lem:convergence_of_class}
    Let $X_n$ be a Markov chain on a countable state space $\cS$, and let $\cM \subseteq \cS$ be an aperiodic and positively 
    recurrent communicating class.
    Then the limit 
    \begin{equation}
        \lim_{n\to \infty}  \bP \left( X_n \in A  \right) 
    \end{equation}
    exists for all $A\subseteq \cM$. 
\end{lem}

\begin{proof}
    We define $\tau_{\cM} := \inf\{ n\geq 0 : X_n \in \Mcal \}$.
    We have 
    \begin{equation}
        \bP \left( X_n \in A \right) = \sum\limits_{j=0}^\infty \bP \left( X_n\in A\, | \, \tau_{\cM} = j \right)  
        \bP \left( \tau_{\cM}=j \right) .
    \end{equation}
    (We use that $X_n \in A$ can only happen if the chain enters $\cM$ at some point.)
    Conditioning on $\tau_{\cM} = j$ induces a probability distribution $\nu_{j}$ over the states in $\cM$ 
    where for $B\subseteq \cM$ we have $\nu_{j} ( B ) = \bP \left( X_j \in B\, |\, \tau_{\cM}=j  \right)$. 
    Consider the Markov chain we get by discarding all states outside of $\Mcal$ and 
    leaving the transition probabilities inside $\Mcal$ as is. This is a valid chain, since 
    states in $\Mcal$ can only reach states inside $\Mcal$, as $\Mcal$ is recurrent.
    This ``reduced chain'' is clearly aperiodic, irreducible and positive recurrent. 
    So by Theorem~\ref{thm:markov_convergence} there exists 
    a unique distribution $\pi$ over $\cM$ such that for all $j\in \bZ_{\geq 0}$
    \begin{equation}\label{eq:ergkjnerjkgr}
        \lim_{n\to \infty}\bP \left( X_n\in A\,|\, \tau_{\cM}=j \right) 
        = \lim_{n\to \infty}\bP_{\nu_{j}} \left( X_{n-j}\in A \right) = \pi(A). 
    \end{equation}
    Now 
    
    $$ \left|\Pee\left( X_n \in A \right) - \pi(A)\cdot\Pee( \tau_{\cM} < \infty ) \right| 
    \leq \sum_{j=0}^\infty \left| \Pee\left( X_n \in A \, |\, \tau_{\cM} = j \right) - \pi(A) \right| \cdot
    \Pee( \tau_{\cM} = j ) \xrightarrow[n\to\infty]{} 0, $$
    
    \noindent
    where the limit follows by dominated convergence (each term of the summand tends to zero separately, 
    and the sum is bounded by $\sum_j \Pee( \tau_{\cM} = j ) = \Pee( \tau_{\cM} < \infty ) \leq 1$).
\end{proof}

\subsection{Mallows permutations\label{subsec:useful_lit}\label{subsec:constr_trunc}}

It is a standard result in enumerative combinatorics (see Corollary 1.3.13 in~\cite{StanleyVol1}) that for $q\neq 1$
the denominator in the definition of the Mallows distribution~\eqref{eq:Mallowsdef} satisfies 

\begin{equation}\label{eq:Znq}
    \sum_{\sigma\in S_n} q^{\inv(\sigma)} = \prod_{i=1}^n \frac{1-q^i}{1-q}.  
\end{equation}

We define for $n\geq 1$ the permutation $r_n : i \mapsto n- i + 1$. For every $\pi \in S_n$ we have
\begin{align}
    \inv (r_n\circ \pi)& = \inv (\pi \circ r_n) = \binom{n}{2}-\inv(\pi),\\
    \inv(\pi^{-1}) &= \label{eq:inversions_of_inverse}\inv(\pi).
\end{align}
From this it can be seen that, if $\Pi_n\sim \Mallows(n,q)$, then 

$$ \Pi^{-1}_n\isd \Pi_n 
\quad \text{ and } \quad 
r_n\circ \Pi_n \isd  \Pi_n\circ r_n \isd \Mallows(n,1/q). $$ 

For $X = \{x_1,\dots,x_n\}$ a set of distinct numbers and $x \in X$ we define the {\em rank} $\rk(x,X)$ of $x$ in $X$ as the unique $i$ such that $x$ is the $i$-th smallest element of $X$. For $x = (x_1,\dots,x_n)$ a sequence of distinct numbers, let us write 

\begin{equation}\label{eq:rank_definition}
    \rk(x_1,\dots,x_n) := \left(\rk(x_1,\{x_1,\dots,x_n\}), \dots, \rk(x_n,\{x_1,\dots,x_n\}) \right).
\end{equation}
With some abuse of notation we will sometimes write $\pi = \rk(x_1,\ldots, x_n)$ to mean that $\pi\in S_n$ is the permutation satisfying $\pi(i) = \rk(x_i, \{x_1,\ldots, x_n\})$ for all $i\in [n]$.

\begin{lem}[\cite{BhatnagarPeled2015}, Corollary 2.7]\label{lem:prefix_mallows}
If $\Pi_n \sim \Mallows(n,q)$ and $1 \leq i < j \leq n$ then
    \begin{equation}
        \rk ( \Pi_n(i),\Pi_n(i+1),\ldots,\Pi_n(j-1),\Pi_n(j) )\isd \Mallows(j-i+1,q).
    \end{equation}
\end{lem}

In \cite{Gnedin} Gnedin and Olshansky introduced a random bijection $\Sigma :\bZ \to \bZ$ which 
in a suitable sense 
is an extension of the finite $\Mallows(n,q)$ model with $0<q<1$. 
We denote this distribution by $\Mallows( \bZ, q)$ where $0<q<1$. 
We will not need the details of the construction and refer the reader to the original paper for a detailed description. 

In \cite{MullerVerstraaten}, Jimmy He together with the first and last author of the current paper studied the limit behavior of the cycles counts of $\Mallows(n,q)$ distributed random permutations for fixed $q\neq 1$. Given a permutation $\pi \in S_n$ we define $C_i(\pi)$ to be the number of $i$--cycles of $\pi$ for $i\geq 1$. We will use the following results:

\begin{thm}[\cite{MullerVerstraaten}, Theorem 1.1]\label{thm:normal}
Fix $0<q<1$ and let $\Pi_n \sim \Mallows(n,q)$. There exist positive constants $m_1,m_2,\ldots $ and an infinite matrix $P\in \bR^{\bN \times \bN}$ such that for all $\ell\geq 1$ we have

\begin{equation}
\frac{1}{\sqrt{n}} \left(C_1(\Pi_n)-m_1 n, \ldots , C_\ell (\Pi_n) - m_{\ell} n \right) \quad 
\xrightarrow[n\to\infty]{d}\quad \cN_\ell(\vec{0}, P_\ell),
\end{equation}
where $\cN_\ell(\cdot,\cdot)$ denotes the $\ell$--dimensional multivariate normal distribution and $P_\ell$ is the submatrix of $P$ on the indices $[\ell] \times [\ell]$.
\end{thm}

We define the bijections $r,\rho :\bZ \to \bZ$ by $r(i) = -i$ and $\rho(i) = 1-i$. 

\begin{thm}[\cite{MullerVerstraaten}, Theorems 1.3 and 1.7]\label{thm:odd_cycles}
Let $q>1$ and $\Pi_{n}\sim \on{Mallows}(n, q)$ and $\Sigma \sim\on{Mallows}(\bZ, 1/q)$. We have 
$$ (C_1(\Pi_{2n+1}), C_3(\Pi_{2n+1}),\ldots )
 \xrightarrow[n\to\infty]{d}    \quad (C_1(r \circ \Sigma), C_3(r \circ \Sigma),\ldots ) 
 $$ 
and 
$$ (C_1(\Pi_{2n}), C_3(\Pi_{2n}),\ldots )   \xrightarrow[n\to\infty]{d}    
\quad (C_1(\rho \circ \Sigma), C_3(\rho \circ \Sigma),\ldots ). 
$$ 
Moreover, the two limit distributions above are distinct for all $q>1$.
In particular, there exists a value $x=x(q)$ such that $\Pee( C_1(\rho\circ\Sigma)\geq x ) \neq \Pee( C_1(r\circ\Sigma)\geq x )$.
\end{thm}

The following result is due to Bhatnagar and Peled.
\begin{thm}[\cite{BhatnagarPeled2015}, Theorem 1.1]\label{thm:bhat_peled}
    For all $0<q<1$, and integer $1\leq i\leq n$, if $\Pi_n\sim \Mallows(n,q)$ then 
    \begin{equation}
        \bE |\Pi_n(i)-i| \leq  \min \left( \frac{2q}{1-q}, n-1 \right).
    \end{equation}
\end{thm}

%
%

The following method of sampling a $\Mallows(n,q)$ distributed permutation for $0<q<1$ goes back to the work of 
Mallows~\cite{Mallows}.
Let $Z_1,\dots,Z_n$ be independent with $Z_i \sim \allowbreak {\TGeo(n-i+1,1-q)}$. We now set 
\begin{align}\label{eq:MallowsIter}    
    \Pi_n(1) &:= Z_1,\\    
    \Pi_n(i)&:= \text{the $Z_i$-th smallest number in $[n]\setminus\{ \Pi_n(1),\ldots, \Pi_n(i-1)\}$},\qquad\text{ for $1<i\leq n$}.
\end{align}
\noindent
Then $\Pi_n$ follows a $\Mallows(n,q)$ distribution. If in this construction we replace the $Z_i$ with 
independent uniform random variables on the the set $[n-i+1]$, then the above procedure samples a uniformly distributed permutation from $S_n$. 
%

A natural adaptation of the sampling procedure generates a random bijection $\Pi : \bN \to \bN$, as follows.
We let $Z_1, Z_2, \dots$ be i.i.d.~$\Geo(1-q)$ and set: 
\begin{align}\label{eq:MallowsItereN}
   \Pi(1) &:= Z_1, \\
   \Pi(i) &:= \text{the $Z_i$-th smallest number in $\eN\setminus\{\Pi(1),\dots,\Pi(i-1)\}$},\qquad\text{for }i>1.  
\end{align}
The distribution on permutations of $\eN$ generated by this procedure will denoted by $\Mallows(\eN,q)$.
It was first introduced by Gnedin and Olshanki in \cite{GnedinOlshanski2010}.

To avoid confusing the reader we stress that, while the $\Mallows(\Zed,q)$ distribution that features 
in Theorem~\ref{thm:odd_cycles} above was also introduced by Gnedin and Olshanski, 
it is very much distinct from the $\Mallows(\eN,q)$ distribution.
(In particular, one is a random permutation of $\eN$ and the other a permutation of $\Zed$.) 
At least as far as we are aware, there is no simple iterative procedure for 
generating $\Sigma\sim\Mallows(\Zed,q)$ available in the literature.

\begin{rem}\label{rem:get_any_pi} For any finite sequence $i_1,\dots,i_n$ of distinct positive integers 
we can uniquely associate a 
sequence $z_1,\ldots, z_n$ such that 
if $\Pi \sim \Mallows(\bN, q)$ is constructed from an i.i.d. sequence $Z_1,Z_2,\ldots$ of $\Geo(1-q)$ random variables as 
described, then $\{ \Pi(1)=i_1,\dots,\Pi(n)=i_n\} = \{  Z_1=z_1,\ldots, Z_n = z_n  \}$. 
\end{rem}

The following result was obtained by Basu and Bhatnagar (\cite{basu2016limit}, Lemma W) and 
independently by Crane and DeSalvo (\cite{CraneDesalvo2017}, Lemma W):

\begin{lem}\label{lem:basu_Pin}
    For $q<1$, if $\Pi\sim \Mallows(\bN,q)$ and $\Pi_n \be \rk (\Pi(1),\ldots, \Pi(n))$ then $\Pi_n\isd \allowbreak \Mallows(n,q)$.
\end{lem}

This gives us a natural coupling of $\Pi \sim \Mallows(\eN,q)$ and a
sequence $\Pi_1,\Pi_2,\dots$ with $\Pi_n \sim \Mallows(n,q)$. 

Given $\Pi \sim \Mallows(\eN,q)$ we can define a sequence of {\em regeneration times }$T_0 <T_1 < T_2<\ldots$ as follows:

\begin{align}
	T_0 &:=0,\\
	T_i &:=  \inf\{ j > T_{i-1} \text{ s.t. }  \Pi( [j] )=[j] \} \qquad (i=1,2,\dots),.
\end{align}

We have 

\begin{lem}\label{lem:Rplus}[\cite{MullerVerstraaten}] $\Ee T_1^k < \infty$ for every $k \in \eN$.
\end{lem}

\noindent
Although this last result is not spelled out explicitly in~\cite{MullerVerstraaten} -- or anywhere else in the literature as far 
as we are aware -- it readily follows by applying Lemma 2.11 in~\cite{MullerVerstraaten} to the $(\infty,q)$-arc chain
as defined there. See also the proof of Lemma 4.3 in~\cite{MullerVerstraaten}. 
Basu and Bhatnagar~\cite{basu2016limit} had previously shown that $T_1$ has finite second moment. 

We also define the {\em interarrival times}

$$ X_i := T_i - T_{i-1}, $$ 

\noindent 
Looking at the sampling procedure generating $\Pi$, it is not difficult to see that 
conditional on the event $T_1 = t$, the bijection $i \mapsto \Pi(i+t)-t$ is distributed like $\Pi$.
It follows that the interarrival times $X_1, X_2, \dots$ are i.i.d.
Moreover, writing $\Xcal_i := \{T_{i-1}+1,\dots,T_i\}$ we see that 
$\Pi$ maps $\Xcal_i$ bijectively onto $\Xcal_i$, and in fact the permutations $\tilde{\Pi}_1 : [X_1]\to[X_1],
\tilde{\Pi}_2: [X_2]\to[X_2], \dots$ 
given by 

$$ \tilde{\Pi}_i(j) := \Pi(T_{i-1}+j) - T_{i-1} \quad \text{ for $j=1,\dots,X_i$, } $$

\noindent
are i.i.d.~as well. 

For permutations $\pi \in S_n$ and $\sigma\in S_m$ we define $\pi \oplus \sigma$ to be the 
permutation $\tau \in S_{n+m}$ satisfying

\begin{align}\label{eq:def_oplus}
    \tau(i) = \begin{cases}    \pi(i) &\text{ if }i\leq n,\\ \sigma(i) + n&\text{ if }i>n.  \end{cases}
\end{align}

With this definition, we can write 

$$ \Pi = \tilde{\Pi}_1 \oplus \tilde{\Pi}_2 \oplus \dots. $$

\noindent
For $n \in \eN$ let us write:

$$ N(n) := \max\{ i : T_i \leq n \}. $$ 

\noindent 
We have the identity

\begin{equation}\label{eq:Pinoplus} 
    \Pi_n = \begin{cases}
            \tilde{\Pi}_1 \oplus \dots \oplus \tilde{\Pi}_{N(n)} & \text{ if $T_{N(n)} = n$, } \\
            \tilde{\Pi}_1 \oplus \dots \oplus \tilde{\Pi}_{N(n)} \oplus 
            \rk( \tilde{\Pi}_{N(n)+1}(1),\dots,\tilde{\Pi}_{N(n)+1}(n-T_{N(n)}) ) & \text{ otherwise. }
           \end{cases}
\end{equation}

\subsection{First--order logic\label{sec:logic}}

Here we briefly cover some of the concepts from logic and model theory that we will need in our proofs. 
For a more complete descriptions, see for instance Chapter 1 of \cite{Immerman}.

Given a set of relation symbols $R_1,\ldots, R_k$ with associated arities $a_1,\ldots, a_k$, we define $\cL_{R_1,\ldots, R_k}$ as the first--order language consisting of all 
first--order formulas built from the usual Boolean connectives together 
with $R_1,\ldots, R_k$. As an example, for a single binary relation $R$, we may consider the formula $\phi_0(x)\in \cL_{R}$ defined as
\begin{equation}\label{eq:skhrbejrhg}
    \phi_0(x) \be R(x,x)\wedge \neg \exists y : (\neg(x=y) \wedge R(y,y)).
\end{equation}
A \emph{free variable} of a formula $\phi$ is a variable that does not have a quantifier. 
A formula with no free variables is called a {\em sentence}. 
Given $R_1,\ldots, R_k$ with arities $a_1,\ldots, a_k$, a $(R_1,\ldots, R_k)$--structure $\fA$ is a 
tuple $(A,R_1^\fA,\ldots, R_k^\fA)$ where $A$ is a set, called the 
domain of $\fA$, and each $R_j^{\fA}$ is a relation of arity $a_j$ over $A$. 
We will routinely write simply $R_j$ for $R_j^\fA$ and talk about structures instead of $(R_1,\ldots, R_k)$--structures. 
If a structure $\fA$ satisfies a sentence $\phi$ then we write $\fA \models \phi$. Similarly, if $i_1,\ldots,i_k $ 
are elements in the domain of $\fA$ and $\phi(x_1,\ldots, x_k)$ is a formula 
whose free variables are $x_1,\ldots, x_k$ then we write $(\fA, i_1,\ldots, i_k) \models \phi$ if $\phi$ is 
satisfied by $\fA$ under the assignment $x_j \mapsto i_j$. 
We have for instance that $(\fA, i)\models \phi_0$ precisely when the element $x$ is the unique 
element satisfying $R(x,x)$ in $\fA$. 
An {\em atomic formula} is a formula without quantifiers and Boolean connectives.

Two $(R_1,\ldots, R_k)$--structures $\fA, \fB$ \TV{with domain $A$ and $B$ respectively} are called \emph{isomorphic} if there is a bijection $f: A\to B$ such that
\begin{equation}
    (i_1,\ldots, i_k) \in R_j^{\fA} \iff  (f(i_1),\ldots, f(i_k)) \in R_j^{\fB}\qquad\text{ for all $i_1,\ldots, i_k\in A$ and all $R_j$.}
\end{equation}
The following is a well--known result on isomorphic structures, see e.g. Theorem~W of~\cite{Logic} for a proof.

\begin{lem}
    If $\fA$ and $\fB$ are isomorphic, then for all first--order sentences $\phi$ we have $\fA \models \phi$ if and only if $\fB \models \phi$. 
\end{lem}

We recursively define the quantifier depth of a first--order sentence $\phi$, denoted $D(\phi)$, as follows: 
\begin{itemize}
    \item For any atomic formula $\phi$ we set $D(\phi)=0$;
    \item For any formula $\phi$ we set $D(\neg \phi) = D(\phi)$;
    \item For any two formulas $\phi, \psi$ we set $D(\phi \wedge \psi)=D(\phi \vee \psi) = \max\{ D(\phi), D(\psi) \}$;
    \item For a formula $\psi(x)$ with a free variable $x$ we set $D(\exists x :\psi(x) ) = D(\forall x :\psi(x) ) = 
    D(\psi(x)) + 1$.
\end{itemize}

We say that two structures $\fA$ and $\fB$ are \emph{$d$--equivalent} (notation : $\fA \equiv_d \fB$) 
if $\fA$ and $\fB$ agree on all sentences $\phi$ with $D(\phi)\leq d$. 
The relation $\equiv_d$ is an equivalence relation. 
The following lemma will be crucial in the proof of Theorem~\ref{thm:toto}~Part~\ref{part:conv_toto}. 
A proof can be found for instance following Corollary W in \cite{Libkin}.

\begin{lem}\label{lem:finite_classes}
    The equivalence relation $\equiv_d$ has only finitely many equivalence classes.
\end{lem}

For a structure $\fA$ over domain $A$ we define the \textit{Gaifman Graph of $\fA$}, denoted $G_\fA$, as the 
graph $(A, E^\fA)$ where $E^\fA$ is the collection of all pairs $\{a,b\} \subseteq A$ such that 
$a$ and $b$ occur together in an element of at least one relation $R^\fA$. 
For an element $x\in A$ we define its $r$--neighborhood $N(x,r)$ as
\begin{equation}
    N(x,r) = \{  y  \in A \mid d_{G_{\fA}}(x,y) \leq r \}.
\end{equation}
We say that a structure $\fA$ has \emph{degree $d$} if the maximum degree of $G_{\fA}$ is $d$. 

For a structure $\fA$ and $X\subseteq A$ we define $\fA \upharpoonright X$ to be the restriction of $\fA$ to $X$.
That is $\fA \upharpoonright X$ has domain $X$ and for 
every relation symbol $R_i$, $R_i^{ \fA \upharpoonright X } = R_i^\fA \cap X^{a_i}$, where $a_i$ is the 
arity of $R_i$.

For $r\geq 0$ and a structure $\fA$, the \emph{$r$--type} of an element $x\in A$ is the isomorphism class of $\fA \upharpoonright N(x,r)$. Two structures $\fA$ and $\fB$ are said to be \emph{$(r,s)$--equivalent} if for every $r$--type, $\fA$ and $\fB$ either have the same number of elements of this type, or they both have more than $s$ elements of this type. 
We will use the following powerful result in the proof of Theorem~\ref{thm:logical_limits}. 
A proof can for instance be found in \cite{Immerman} (Theorem W).

\begin{thm}[Bounded--Degree Hanf Theorem]\label{thm:hanf_general}
    Let $k$ and $d$ be fixed. Then there is an integer $s$ such that for all structures $\fA$ and $\fB$ of degree at most 
    $k$, if $\fA$ and $\fB$ are $(2^d, s)$ equivalent, then $\fA \equiv_d \fB$. 
\end{thm}

We also briefly mention \emph{second--order logic}. In second--order logic, formulas can now also take arguments that are relations, and we can quantify over relations. An example of such a formula is 
\begin{equation}
    \xi(R) \be \forall x \,(  \neg R(x,x) \wedge (\exists! y : R(x,y) \wedge R(y,x) )  ),
\end{equation}
where we use the shorthand $\exists ! x $ to mean that there exists a unique such $x$, i.e., $ {\exists! x :\psi(x)} \be \exists x : (\psi(x)\wedge \forall y \ (\psi(y) \rightarrow y=x))$. The formula $\xi(R)$ is satisfied by a binary relation $R$ if and only if $R$ encodes a perfect matching. Moreover, in second--order logic we can quantify over relations. So we may write the sentence $\phi \be \exists R : \xi(R)$. A \TV{finite} structure satisfies $\phi$ if and only if its domain has even cardinality.

\subsubsection{$\an{TOOB}$ and $\an{TOTO}$\label{subsec:TOOB_TOTO}}


Given a single binary relation $R$, we define $\TOOB \be \cL_{R}$, the first--order language obtained from the single binary relation $R$. A permutation $\pi$ can be encoded as a structure $([n], R^\pi)$ by stipulating that $R^\pi(i,j)$ precisely when $\pi(i)=j$.

We define $\TOTO = \cL_{<_1,<_2}$ where $<_1$ and $<_2$ are two binary relations. A permutation $\pi$ is then encoded as $([n], <_1^\pi, <_2^\pi)$ where $i <_1^\pi j$ if and only $i<j$ in the usual ordering of $[n]$, and $i <_2^\pi j$ if and only if $\pi(i) < \pi(j)$. So $<_1$ and $<_2$ (recall that we often leave out the superscripts) are both total orders on the domain $[n]$.

We remark that an $R$--structure or a $(<_1,<_2)$--structure does not necessarily correspond to a permutation. 
In an $R$-structure $\fA$ the relation $R^\fA$ might not correspond to a bijection and similarly 
in a $(<_1,<_2)$-structure one or both of 
the relations $<_1^\fA, <_2^\fA$ might not be a total order on the domain. 
One way to fix this is to identify sets of axioms that structures in $\TOOB$, respectively $\TOTO$, should satisfy.
In our setting we will be sampling from a distribution over the set of permutations and regarding them as 
either $R$--structures or $(<_1,<_2)$--structures, so there is no need for such axioms. 

We define the $\an{TOTO}$ formulas
\begin{align}
    \succ_1(x,y) \be (x<_1 y) \wedge \neg \exists w ( (x<_1 w) \wedge (w <_1 y)),\\
    \succ_2(x,y) \be (x<_2 y) \wedge \neg \exists w ( (x<_2 w) \wedge (w <_2 y)).
\end{align}
These formulas are such that $(\pi,i,j)\models \succ_1(x,y)$ if and only if $j$ is the successor of $i$ under $<_1$, and $(\pi,i,j)\models \succ_2(x,y)$ if and only if $j$ is the successor of $i$ under $<_2$. We also recursively define $\succ_1^{(1)} \be \succ_1$ and for $k\geq 1$
\begin{align}
    \succ_1^{(k+1)}(x,y) \be \exists w (   \succ_1(w, y) \wedge \succ_1^{(k)}(x,w)  ),
\end{align}
and similarly $\succ_2^{(k)}$. These functions are such that for $k\geq 1$
\begin{align}
    (\pi,i,j)\models \succ_1^{(k)} (x,y)\qquad& \iff \qquad i +k = j,\\
    (\pi,i,j)\models \succ_2^{(k)}(x,y) \qquad& \iff \qquad \pi(i) +k = \pi(j).
\end{align}
%
%
%

Recall the definition of $ \rk(\vec{x})$ for a sequence $\vec{x}=(x_1,\ldots, x_n)$ of distinct numbers as 
given in \eqref{eq:rank_definition}.
The following lemma is essentially a special case of Theorem 4.2.1 in \cite{ShorterModel}.
For completeness we provide a short proof.

\begin{lem}\label{lem:relativization}
    Let $\phi(\vec{x}) \in \TOTO$ be a formula with $k\geq 0$ free variables, and 
    let $\psi(y) \in \TOTO$ be a formula with one free variable. There exists a formula $\varphi^\psi(\vec{x}) \in \TOTO$ such 
    that for all $n$, all $\pi \in S_n$ and all 
    $\vec{i} = (i_1,\dots,i_k) in [n]^k$we have 
    \begin{equation}
        (\pi,\vec{i}) \models \phi^{\psi} (\vec{x}) \quad\text{ if and only if }\quad \pi \models 
        \psi(i_1)\wedge\dots\wedge\psi(i_k) \text{ and } (\pi^*,\vec{i}) \models \phi(\vec{x}),
    \end{equation}
    where $\pi^* = \rk(\pi(a_1),\ldots, \pi(a_m))$ with $a_1<\dots<a_m$ and 
    $\{a_1,\dots,a_m\} = \{ a \in [n] : 
    (\pi,a) \models \psi(y) \}$. 
\end{lem}

\begin{proof}
    The proof is by induction on the structure of the formula $\varphi$, starting with the case when 
    it is an atomic formula. 
    Every atomic formula is of the form $R(x_1, x_2)$, for $R \in \{=,<_1,<_2\}$. 
    In this case we define $\phi^{\psi}(\vec{x})$ simply as $\psi(x_1)\wedge\psi(x_2)\wedge R(x_1,x_2)$.
    This clearly does the trick for atomic formulas.
    
    For formulas constructed from smaller formulas using one of the logical connectives, we 
    define $(\neg \phi)^{\psi}$ as $\neg (\phi^{\psi})$ and 
    $(\phi \vee \psi)^{\psi}$ as $\phi^{\psi}\vee \psi^{\psi}$ and similarly for $\wedge,\rightarrow,\leftrightarrow$. 
    And, if $\varphi$ starts with a quantifier then we set 
    
    $$ \left( \exists x : \varphi(x) \right)^{\psi} := \exists x : \psi(x) \wedge \varphi^\psi(x), 
    \quad 
    \left( \forall x : \varphi(x) \right)^{\psi} := \forall x : \psi(x) \rightarrow \varphi^\psi(x). $$
   
   As the reader can readily check the obtained formula has the behaviour described by the lemma statement.
\end{proof}

The formula $\phi^{\psi}$ provided by Lemma~\ref{lem:relativization}
is called the \emph{relativization of $\phi$ to $\psi$}.

\begin{lem}\label{lem:reverse_formula}
    Let $\phi$ be a $\TOTO$ formula with $k\geq 0$ free variables. There exists a $\TOTO$ formula $\phi^{\on{reverse}}$ such that for all $n\geq 1$, all $\pi \in S_n$ and all $\vec{i}\in [n]^k$ we have
    \begin{equation}
        (\pi,\vec{i}) \models \phi \quad \text{ if and only if }\quad (r_n \circ \pi,\vec{i}) \models \phi^{\on{reverse}}.
    \end{equation}
\end{lem}

\begin{proof}
    As in the proof Lemma~\ref{lem:relativization}, we use induction on the structure of formulas. 
    For atomic formulas we set    
    
    $$
        (x = y)^{\on{reverse}} := (x = y),\quad\quad
        (x <_1 y)^{\on{reverse}} := (x <_1 y),\quad\quad
        (x <_2 y)^{\on{reverse}} := (y <_2 x).
    $$
    
    Now, $\pi(i) < \pi(j)$ if and only if $r_n \circ \pi (i) > r_n\circ \pi(j)$, so that
    \begin{equation}
        (\pi ,i,j) \models x <_2 y \quad\iff\quad (r_n \circ \pi,i,j) \models y <_2 x.
    \end{equation}
    The other two relations are straightforward to check and the result now easily follows by induction 
    on the structure of the formula. 
\end{proof}

\subsection{Two rapidly growing functions\label{subsec:definitions}}

The \emph{tower function} is  defined as 

\begin{equation}
    T(n) = 2^{2^{\cdot ^{\cdot ^{2}}}}    
\end{equation}

\noindent
where the tower of $2$'s has height $n$. 
It may also be defined recursively by $T(0)=1$ and $T(i) = 2^{T(i-1)}$ for $i\geq 1$. 


Similarly, we define the \emph{wowzer function} $W(n)$ by
\begin{equation}
    W(0) = 1,\qquad\text{ and }\qquad W(n) = T(W(n-1))\quad\text{ for $n\geq 1$}.
\end{equation}
The function $W(\cdot)$ escalates rapidly. 
We have $W(1) = T(1) = 2$, $W(2) = T(2)=4$, $W(3) = T(4) = 65536$, and $W(4) = T(65536)$ is a tower of $2$'s of height $65536$. 

The previous two functions are part of a larger sequence of functions called the hyperoperation 
sequence, see e.g.~\cite{Goodstein}. The tower function is sometimes known under the name tetration and the wowzer function under the name pentation. Knuth also describes these functions in \cite{Knuth} as part of a larger hierarchy in terms of arrow notation.

The \emph{log--star} function $\log^* n$ is the ``discrete inverse'' of the Tower function. 
It can be defined by 
\begin{equation}
     \log^* n := \min\{ k \in \bZ_{\geq 0} : T(k) \geq n \}. 
\end{equation}
We have $\log^*(T(n)) = n$ for all $n\geq 0$. 
Also note that, phrased differently, $\log^* n$ is the number of times we need to iterate the base two logarithm, starting from
$n$ to reach a number less than $1$.

We define the $\log^{**}$ function as 
\begin{equation}\label{eq:log_star_star_def}
     \log^{**} n := \min\{ k \in \bZ_{\geq 0} : W(k) \geq n \}.
\end{equation}

The $\log^{**}n$ function grows incredibly slowly. 
Although the wowzer function is part of a larger hierarchy of functions as 
mentioned, we have not found the function $\log^{**}$ used anywhere in the literature. 
We emphasize that the notation $\log^{**}n$ is our own and may not be standard.

Although the $\log^{**}$ function is not strictly increasing, we do have the following simple monotonicity principle. 

\noindent
(Here and in the rest of the paper, for $f$ a function and $i$ an integer, 

$$f^{(i)} = \underbrace{f \circ \dots \circ f}_{i\times}, $$

\noindent
denotes the $i$-fold iteration of $f$.)

\begin{lem}\label{lem:ineq_w_log_star}
    For all $i\in \bZ_{\geq 0}$ and integers $x,y$, if $W^{(i)} (x) < y$ then $x < (\log^{**})^{(i)} (y)$.
\end{lem}

\begin{proof}
    We use induction. The base case $i=0$ holds trivially. So suppose that $W^{(i+1)}(x) = W(W^{(i)}(x)) < y$. From \eqref{eq:log_star_star_def} we immediately obtain that $ W^{(i)}(x) < \log^{**}y $, and the result follows by induction. 
\end{proof}


\begin{lem}\label{lem:composed_large_functions}
    For all $m > 2$ we have
    \begin{equation}
        T^{(3)}\circ W^{(2)}(m)  < W^{(2)}(m+ 1).
    \end{equation}
\end{lem}

\begin{proof}
    By the definition of $W$ we have $T^{(3)}(W(W(m))) = W(W(m)+3)$, so it is enough to check that $W(W(m)+3) < W(W(m+1))$. This reduces to checking that 
    \begin{equation}\label{eq:akjhekfhwb}
        W(m) + 3 < W(m+1) = T(W(m)).
    \end{equation}
    But $T(x) > x+3$ for all $x>2$. As $W(m) > 2$ for $m>2$, \eqref{eq:akjhekfhwb} holds.
\end{proof}

\section{Proof of Theorem~\ref{thm:logical_limits}\label{sec:toob}}

Here prove Theorem~\ref{thm:logical_limits}. 
We are in particular considering $\TOOB$ and all mention of 
$(r,k)$-equivalence classes and all uses of $\equiv_d$ in this section 
are thus wrt.~$\TOOB$.
%
The bounded--degree Hanf theorem (Theorem~\ref{thm:hanf_general}) gives the following basic corollary.
For completeness we spell out a proof.

\begin{cor}\label{cor:nietteun}
 For every $d\geq 1$ there is a $h=h(d)$ such that the following holds.
 If $\pi,\sigma$ are two permutations such that 
 $C_i(\pi), C_i(\sigma) \geq h$ for $1\leq i \leq h$ then 
 $\pi\equiv_d\sigma$.
\end{cor}

\begin{proof}
The Gaifman graph (wrt.~$\TOOB$) of a permutation has maximum degree \TV{at most} $2$.
Let $s$ be as provided by the bounded--degree Hanf theorem (with $k=2$), and set 
$h := 1000\cdot\max(2^d,s)$.
We point out that all points on a cycle of length $\geq h$ in any permutation have the same $2^d$--type. 
Both $\pi$ and $\sigma$ have at least $s$ such points.
For each $1\leq i<h$ both $\pi$ and $\sigma$ have at least $s$ points on $i$-cycles.
(And for each $i$ all points on $i$-cycles have the same $2^d$--type.)
 
The result thus follows from the bounded--degree Hanf theorem.
\end{proof}

We now have the necessary tools to prove Theorem~\ref{thm:logical_limits}.
Part~\ref{item:logi_convergence} has already been shown by Compton~\cite{Compton89II}, so we will 
only supply proofs for the other two parts.

\begin{proof}[Proof of Theorem~\ref{thm:logical_limits} Part~\ref{item:logi_zero_one}]
Fix an arbitrary sentence $\varphi \in \TOTO$, let $d$ be its quantifier depth and 
let $h=h(d)$ be as supplied by Corollary~\ref{cor:nietteun}.
Fix a permutation $\pi$ satisfying $C_1(\pi) = h$ for $i=1,\dots,h$, and let 
$\Pi_n \sim \Mallows(n,q)$ as usual.
By Theorem~\ref{thm:normal} 
we have 

$$ \Pee( C_1(\Pi_n), \dots, C_h(\Pi_n) \geq h ) \xrightarrow[n\to\infty]{} 1. $$

\noindent
By the previous corollary we thus have that 
    \begin{equation}
        \lim_{n\to \infty} \bP \left( \Pi_n\models \phi \right) = 1_{\{\pi \models \phi\}}.
    \end{equation}
In particular the limiting probability is $\in \{0,1\}$, which is what needed to be shown.  
\end{proof}

\begin{proof}[Proof of Theorem~\ref{thm:logical_limits} Part~\ref{item:logi_non_convergence}]
    Theorem~\ref{thm:odd_cycles} tells that for $q>1$ there exists a value $x$ such that, with 
    $\Pi_n\sim\Mallows(n,q)$ and 
    $\Sigma \sim \Mallows(\bZ, 1/q)$, we have 

   \begin{equation}
       \lim_{n\to\infty}
         \bP \left( C_1(\Pi_{2n+1})\geq x \right) = \bP \left( C_1(r \circ \Sigma)\geq x \right) 
        \neq
        \bP \left(C_1(\rho \circ \Sigma)\geq x \right) = 
        \lim_{n\to\infty} \bP \left( C_1(\Pi_{2n})\geq x \right).
    \end{equation}
   
   \noindent
   The event that $C_1(\Pi_{n})\geq x$ 
   can clearly be queried by a sentence in $\TOOB$.
\end{proof}

\section{Proof of Theorem~\ref{thm:toto}~Part~\ref{part:conv_toto}\label{sec:proof_conv_toto}}

The approach we will take to prove Theorem~\ref{thm:toto}~Part~\ref{part:conv_toto} is inspired by the 
approach taken by Lynch in \cite{Lynch} where a convergence law for random strings over the alphabet $\{0,1\}$ 
is proven for various distributions over the letters. 
Lynch defines a Markov chain on a finite state space that follows the equivalence class of such a random string 
as its length increases, and then uses standard convergence results for Markov chains to conclude the argument. 
The Markov chain that we will define is more complicated to analyze than the one used in \cite{Lynch}, mainly because 
it is defined on a countably infinite state space. 

Throughout this section we will consider permutations in the first--order language $\TOTO$, that is, using the 
two total orders $<_1$ and $<_2$ as described in Section~\ref{subsec:TOOB_TOTO}. 
So $\pi \equiv_d \sigma$ now means that $\pi$ and $\sigma$ agree on all $\TOTO$ sentences of 
depth at most $d$. 
As per Lemma~\ref{lem:finite_classes}, there are finitely many equivalence classes for the equivalence 
relation $\equiv_d$. Given some fixed $d$, we denote these classes by $\cE_1,\ldots, \cE_\ell$, where 
$\ell = \ell(d)$. For a permutation $\pi$ we let $ [\pi]_d $ be the equivalence class containing $\pi$. 
It is convenient to also allow the ``permutation of length zero'', denoted $\piempty$. 
In terms of $\TOTO$ this is simply the $(<_1,<_2)$-structure with domain $S=\emptyset$.
Recall the definition of $\oplus$ from Section~\ref{subsec:constr_trunc}. It is convenient, and makes sense, to 
set $\piempty \oplus \sigma = \sigma \oplus \piempty = \sigma$ for every permutation $\sigma$.

We denote by $\on{id}_n \in S_n$ the identity on $[n]$. Albert et al.~\cite{Albert} have shown the following 
useful facts:

\begin{prop}[\cite{Albert}, Proposition 26]\label{prop:only_fixed}
    Let $n,m$ and $d$ be positive integers with $n,m \geq 2^d-1$. Then we have $\on{id}_m \equiv_d \on{id}_n$.
\end{prop}

\begin{prop}[\cite{Albert}, Proposition 28]\label{lem:concat_equiv}
    Suppose the permutations $\pi_1,\dots,\pi_n, \sigma_1,\dots,\sigma_n$ satisfy $\sigma_i \equiv_d \pi_i$
    for $i=1,\dots,n$. Then $\sigma_1\oplus\ldots\oplus\sigma_n \equiv_d \pi_1\oplus\ldots\oplus\pi_n$. 
\end{prop}

Until further notice we let $0<q<1$ be fixed. 
We couple $\Pi \sim \Mallows(\eN,q)$ and $\Pi_1, \Pi_2, \dots$ with $\Pi_i \sim \Mallows(n,q)$ via the 
coupling described in Section~\ref{subsec:constr_trunc}. That is, $\Pi_n= \rk( \Pi(1),\dots,\Pi(n) )$.
As usual,we denote by $Z_1,Z_2,\dots$ the sequence of i.i.d.~$\Geo(1-q)$--variables that generates $\Pi$.
We let $T_1,T_2,\dots$ and $X_1,X_2,\dots$ and $\tilde{\Pi}_1,\tilde{\Pi}_2,\dots$ and $N(n)$ be as 
defined in Section~\ref{subsec:constr_trunc}.

Let us say a vector $(y_1,\dots,y_k) \in \eN^k$ is {\em allowable} if $y_1,\dots,y_k$ are distinct and
$\{y_1,\dots,y_j\} \neq [j]$ for all $1\leq j \leq k$. Let us denote by $\Acal$ the set of all 
allowable sequences, of any length. Here we also consider the sequence of length zero allowable. 
We will denote this ``empty sequence'' by $\sempty$.

We define a Markov chain $(M_n)_n$ on state space

$$ \Scal := \{ \cE_1,\dots, \cE_\ell \} \times \Acal. $$

Given that $M_n = ([\pi]_d, (y_1,\dots,y_k) )$ for some permutation $\pi$ and 
$(y_1,\dots,y_k) \in \Acal$, we determine $M_{n+1}$ by drawing $Z\sim \Geo(1-q)$ independently of
the previous history, letting $Y$ be the $Z$-th smallest element of $\eN \setminus \{y_1,\dots,y_k\}$ and 
setting

$$ M_{n+1} = \begin{cases}
              \left([\pi]_d, (y_1,\dots,y_k,Y) \right) & \text{ if $(y_1,\dots,y_k,Y) \in \Acal$, } \\
              \left([\pi \oplus (y_1,\dots,y_k,Y)]_d, \sempty\right) & \text{ otherwise. }
             \end{cases}
$$             

\noindent
(Here we also allow $k=0$ in which case $(y_1,\dots,y_k)$ should be interpreted as the empty sequence $\sempty$.) 

Having another look at~\eqref{eq:Pinoplus} and Lemma~\ref{lem:concat_equiv}, we see that $M_n$ describes the evolution of 

\begin{equation}\label{eq:MnPin} 
\left( [\Pi_{T_{N(n)}}]_d, ( \Pi(T_{N(n)}+1)-T_{N(n)}, \dots, \Pi(n)-T_{N(n)} ) \right), 
\end{equation}

\noindent
where the second argument is interpreted as the empty sequence if $T_{N(n)} = n$ and where as 
the starting state we take $M_0 = ([\piempty]_d, \sempty)$. 

Recall that for states $i, j$ of some Markov chain, the notation $i \to j$ means that $j$ is reachable from $i$.

\begin{lem}\label{lem:to_empty_sequence}
    For every state $(\cE,s) \in \Scal$ there exists an $\equiv_d$--equivalence class $\cE'$ such that 
    \begin{equation}
        (\cE,s) \to (\cE',\sempty). 
    \end{equation}
\end{lem}

\begin{proof}
    Suppose $M_n = (\cE,s)$ with $s=(y_1,\dots,y_k)$ and $k>0$. 
    In every subsequent step, with probability $1-q$ we have $Z=1$, which means that 
    $Y$ is the first element of $\eN\setminus \{y_1,\dots,y_k\}$.
    Therefore, we can reach a state of the form $(\cE',\sempty)$ in 
    $\max(y_1,\dots,y_k)-k$ steps.
    
    In case when $M_n = (\cE,s)$ with $s=\sempty$ then we can reach the state $(\cE',\sempty)$ in one step
    where $\cE' = [\pi \oplus \on{id}_1]_d$ with $\pi \in \cE$ arbitrary.
\end{proof}

The next lemma contains all the information about the chain $M_n$ we will need. 

\begin{lem}\label{lem:Markov_Master} The following hold for the Markov chain $(M_n)_n$:
    \begin{enumerate}
        \item Every recurrent communicating class of $M_n$ contains an element of the form $(\cE,\sempty)$;
        \item\label{itm:MM_finitely_many} There are only finitely many recurrent communicating classes of $M_n$;
        \item\label{itm:MM_positive_recurrent} Every recurrent communicating class of $M_n$ is positive recurrent;
        \item\label{itm:MM_aperiodic} Every recurrent communicating class of $M_n$ is aperiodic;
        \item\label{itm:MM_reach_recurrent} With probability one, the chain $M_n$ is in a positive recurrent state for some $n\geq 1$. 
    \end{enumerate}
\end{lem}

\begin{proof}
    For the first two claims, suppose that $(\cE,s)$ is contained in some recurrent communicating class. 
    Then Lemma~\ref{lem:to_empty_sequence} implies that this communicating class contains an element of the 
    form $(\cE',\sempty)$. There are only finitely many such elements, proving the first two claims of the lemma. 

    For the third claim, consider a recurrent class of $M_n$. 
    It contains an element $\alpha$ of the form $\alpha = (\cE,\sempty)$. We define the ``reduced chain'' $(L_i)_i$ by 
    
    $$ L_i := M_{T_i}, $$
    
    \noindent
    so that $L_i$ takes values in the finite set $\{\cE_1,\dots,\cE_\ell\}\times\{\sempty\}$. It is easily 
    seen that $L$ is also a Markov chain.
    The state $\alpha$ is recurrent for $L$ and since $L$ takes only finitely many states, 
    the state $\alpha$ is in fact positively recurrent for $L$.
    
    We denote by $\tau_\alpha^+$ the first visit time of $M_n$ to $\alpha$, and by 
    $\tilde{\tau}_\alpha^+$ the first visit time of $L_i$ to $\alpha$.
    That $\alpha$ is positively recurrent for $L$ means that 
    $\Ee_{\alpha} \tilde{\tau}_\alpha^+ < \infty$. We need to 
    deduce that $\Ee_\alpha\tilde{\tau}_\alpha^+ < \infty$ as well. 
    
    Let $\delta > 0$ be small but fixed, to be determined in the course of the proof. 
    We have
    
    $$ \Pee_\alpha( \tau_\alpha^+ > n ) \leq \Pee_\alpha( \tilde{\tau}_\alpha^+ > \delta n ) 
    + \Pee_\alpha( T_{\floor{\delta n}} > n ). $$
    
    Note that the second probability in the RHS does not depend on the starting state. 
    For convenience we'll write $m:=\floor{\delta n}$.
    Recall that 
    $T_m = X_1+\dots+X_m$ where $X_1,\dots, X_m$ are i.i.d.~and distributed like $T_1$.
    We see that 
    
    $$ \begin{array}{rcl} \Pee( T_m > n ) & = & \Pee( T_m - \Ee T_m > n - \Ee T_m ) \\
    & = & \Pee( T_m - m\Ee T_1 > n - m\Ee T_1 ) \\ 
    & \leq & \Pee(  T_m - m\Ee T_1 > n/2 ) \\
    & \leq & \Pee\left( \left( T_m - m\Ee T_1 \right)^4 > n^4/16 \right) \\
    & \leq & \frac{16 \Ee \left( T_m - m\Ee T_1 \right)^4}{n^4}.
    \end{array} $$
    
    \noindent
    where the third line holds since $\Ee T_1 < \infty$ by Lemma~\ref{lem:Rplus} and
    assuming the constant $\delta$ is chosen sufficiently small. The last line uses Markov's inequality.
    Now note that 
    
    $$ \begin{array}{rcl} \Ee \left( T_m - m\Ee T_1 \right)^4
    & = & \Ee \left( (X_1-\Ee T_1) + \dots + (X_m-\Ee T_1) \right)^4  \\[2ex]  
    & = & \displaystyle \sum_{i=1}^m \sum_{j=1}^m \sum_{k=1}^m \sum_{\ell=1}^m 
    \Ee\left( (X_i-\Ee T_1)\cdot(X_j-\Ee T_1)\cdot(X_k-\Ee T_1)\cdot(X_\ell-\Ee T_1) \right) \\[2ex]
    & = & m \Ee (X_1-\Ee T_1)^4 + 3 m(m-1) \left( \Ee (X_1-\Ee T_1)^2 \right)^2 \\
    & = & O( n^2 ), \end{array} $$
    
    \noindent 
    again using that all moments of $T_1$ are finite (and that the $X_i$ are i.i.d.~distributed like $T_1$).
    It follows that 
    
    $$ \Pee( T_m > n ) = O( n^{-2} ). $$
    
    This gives
    
    $$ \begin{array}{rcl} \Ee_\alpha \tau_\alpha^+ & = & \sum_n \Pee_\alpha(\tau_\alpha^+ > n ) \\
    & \leq & \sum_{n} \Pee_\alpha( \tilde{\tau}_\alpha^+ > \delta n ) + \sum_n O(n^{-2} ) \\
    & \leq & \ceil{1/\delta} \cdot \sum_k \Pee_\alpha( \tilde{\tau}_\alpha^+ > k  ) + O(1) \\ 
    & = & \ceil{1/\delta} \cdot \Ee_\alpha\tilde{\tau}_\alpha^+ + O(1) \\
    & < & \infty, \end{array} $$
    
    \noindent
    using a standard formula for the expectation of non-negative integer valued random variables.
    We've established that $\alpha$ is positively recurrent. Therefore, so are all states in 
    the same communicating class.

    For the fourth claim, consider a recurrent class containing $([\pi]_d,\sempty)$ for some permutation $\pi$. 
    For all $n,t\geq 1$ we have
    \begin{equation}
        \bP \left( M_{n+t} = (  [\pi \oplus \on{id}_t]_d ,\sempty ) \,|\, M_n = ([\pi]_d,\sempty)\right) \geq 
        (1-q)^t >0.
    \end{equation}
    In particular the recurrent class contains both $([\pi \oplus \on{id}_{2^d-1}]_d,\sempty)$ and 
    $( [\pi \oplus \on{id}_{2^d}]_d,\sempty)$, where $M_n$ may transition from the former to the latter in one step 
    with positive probability. By Propositions~\ref{prop:only_fixed} and~\ref{lem:concat_equiv} we have $[\pi \oplus \on{id}_{2^d-1}]_d 
    = [\pi \oplus \on{id}_{2^d}]_d$ so that 
    the class is indeed aperiodic.

    For the fifth claim, note the probability that $M_n$ is never equal to an element $(\cE,\sempty)$ of 
    a recurrent class is exactly the probability that $L_i$ is never equal to a recurrent state $(\cE,\sempty)$. 
    This probability is zero by Lemma~\ref{lem:hit_recurrent}.
\end{proof}

\begin{proof}[Proof of Theorem~\ref{thm:toto}~Part~\ref{part:conv_toto}]
    Fix some $\phi \in {\TOTO}$ with quantifier depth $d$. 
Let the recurrent classes of $M_n$ be denoted $\cM_1,\ldots, \cM_k$. 
(There are finitely many of them by Lemma~\ref{lem:Markov_Master} Part~\ref{itm:MM_finitely_many}.) Define 
\begin{align}
    \tau_{\cM_i} &= \inf \{ n \geq 0 \mid M_n\in \cM_i  \}.
\end{align}
    By~\eqref{eq:Pinoplus} and Proposition~\ref{lem:concat_equiv}, the equivalence class under $\equiv_d$ of $\Pi_n$ can be 
    recovered from $M_n$. 
    In particular there exist $\cM^\phi_1\subseteq \cM_1, \ldots, \cM^\phi_k\subseteq \cM_k$ such that, for all $n$,
    $\Pi_n \models \varphi$ if $M_n \in \cM_i^\varphi$ and 
    $\Pi_n \models \neg\varphi$ if $M_n \in \cM_i \setminus \cM_i^\varphi$.
    (Here and in the remainder of the proof we implicitly take $M_n$ as given via~\eqref{eq:MnPin}.)
    This gives 
   
   $$
    \sum_{i=1}^k \bP \left( M_n \in \cM_i^\varphi  \right) \leq 
       \bP \left( \Pi_n\models \phi \right) 
       \leq   \sum_{i=1}^k \bP \left( M_n \in \cM_i^\varphi  \right) 
       + \Pee( \tau_{\cM_1}, \dots, \tau_{\cM_k} > n ).
    $$

    \noindent
    By Lemma~~\ref{lem:convergence_of_class} each of the terms in the LHS has a limit and by 
    Lemma~\ref{lem:Markov_Master} Part~\ref{itm:MM_reach_recurrent} we have 
    
    $$ \lim_{n\to\infty} \Pee( \tau_{\cM_1}, \dots, \tau_{\cM_k} > n ) = 0. $$
    
    \noindent
    This establishes that $\lim_n \Pee( \Pi_n \models \varphi )$ exists.
    
    To see that this limit is not restricted to being $0$ or $1$, let $\phi$ be the $\TOTO$ 
    sentence \TV{$\exists x (\neg \exists y (y <_1 x \vee y<_2 x))$} expressing that $\Pi_n(1) = 1$. Then (still considering the case $q<1$ fixed), we have
    \begin{equation}
        \bP \left( \Pi_n \models \phi \right)  = \bP \left( \TGeo(n,1-q) = 1 \right)\  \xrightarrow[n\to\infty]{}\   1-q \ \notin\  \{0,1\}.
    \end{equation}

    It remains to handle the (fixed) $q>1$ case. In this case we have for all $\TOTO$ sentences $\phi$ that
    \begin{equation}
        \lim_{n\to\infty} \bP \left( \Pi_n\models \phi \right) = \lim_{n\to\infty} \bP \left( r_n\circ \Pi_n \models \phi^{\on{reverse}} \right) ,
    \end{equation}
    where $\phi^{\on{reverse}}$ is as in Lemma~\ref{lem:reverse_formula}. As $r_n\circ \Pi_n\isd \Mallows(n,1/q)$ we can use the 
    result for $q<1$ to conclude that this limit exists and that this limit is not in $\{0,1\}$ for the $\TOTO$ sentence expressing that $\Pi_n(1) = n$.
\end{proof}

\section{Non--convergence in \texorpdfstring{$\TOTO$}{TOTO} for uniformly random permutations\label{sec:totoqis1}}

In this section we state and prove a proposition that will be used later on in the proof of Theorem~\ref{thm:toto}~Part~\ref{part:toto_non_convergence}.

Recall the definition of $W(\cdot)$ and $\log^{**} (\cdot)$ given in Section \ref{subsec:definitions}. We will prove the following result:

\begin{prop}\label{prop:oscilate}
    There exists a $\phi \in \TOTO$ such that for $\Pi_n \sim \Mallows(n,1)$ and $n$ satisfying $W^{(2)}(\log^{**}\log^{**}n-1)<{\log\log n}$, we have
    \begin{equation}
        \bP( \Pi_n \models \phi ) = 
        \begin{cases}
            1 - O( n^{-100} ) & \text{if } \log^{**}\log^{**}n \text{ is even}, \\
            O( n^{-100} ) &  \text{if } \log^{**}\log^{**}n \text{ is odd}.
        \end{cases}   
    \end{equation}
\end{prop}

\noindent
(We remark that by Lemma~\ref{lem:composed_large_functions} the condition 
$W^{(2)}(\log^{**}\log^{**}n-1)<{\log\log n}$ is for instance satisfied -- with plenty of room to spare -- if 
$n$ is of the form $n = W^{(2)}(m)$ with $m>3$.)

Proposition~\ref{prop:oscilate} is an explicit version of a result given in \cite{Foy} by Foy and Woods who showed that 
there is a $\TOTO$ sentence $\psi$ such that $\bP \left( \Mallows(n,1)\models \psi \right) $ does not have a limit. 
Our plan of attack will follow in broad lines the proof given by Shelah and Spencer in \cite{ShelahSpencer} to show 
non--convergence in first--order logic 
for graphs in the Erd\H{o}s--R\'{e}nyi random graph $G(n,p)$ with $p$ near $n^{-1/7}$. 
See also Chapter 8 of the excellent monograph \cite{Spencerstrange} by Spencer for a similar argument 
for $p$ near $n^{-1/3}$.

\subsection{Arithmetic on sets in second--order logic\label{sec:arithSO}}

We move away from permutations for the moment, and 
consider a finite set $S$ equipped with a total order $<$. 
Our aim will be to determine a second--order sentence $\an{Parity}$ such that
$S\models \an{Parity}$ if and only if $\log^{**}\log^{**} |S|$ is even.

Given that there is a total order $<$, we can identity $S$ with $[n]$ in the obvious way, where of course
$n := |S|$.
Note that we can express that $i = 1$ in a (first order) logical formula via $\neg(\exists x : x<i)$.
Similarly we can express that $i = n$, and that $i=k$ or $i=n-k$ for any fixed $k$.
In Section~\ref{subsec:TOOB_TOTO} we defined the formulas $\succ_1^{(k)}$ using the $<_1$ relation in $\TOTO$.
We will use $\an{succ}^{(j)}$ to denote analogue of the formulas $\an{succ}_1^{(j)}$ with $<_1$ replaced by $<$. 
For clarity, we point out that $\succ^{(j)}(x,y)$ expresses that $y=x+j$.

Our sentence $\an{Parity}$ will begin with the existential quantification over four binary relations $R_D,R_E,R_T$ and 
$R_W$. Here the subscript $D$ stands for ``double'', $E$ for (base two) ``exponential'', 
$T$ for ``tower'', and $W$ for ``wowser''.
We want to demand that these relations are such that:

\begin{equation}\label{eq:REetc}
\begin{array}{rl}
    R_D(i,j) &\iff j=2i, \\
    R_E(i,j) &\iff j=2^i,\\
    R_T(i,j) &\iff j=T(i),\\
    R_W(i,j) &\iff j=W(i),
\end{array}
\end{equation}

\noindent 
where $T(i)$ and $W(i)$ denote the tower and wowser functions defined in Section \ref{subsec:definitions}.
For $R_D$, the following two formulas axiomatize that $R_D$ is as required:

\begin{equation}\label{eq:def_RD}
\begin{array}{l}
\forall \,i,j  : 
	 \left(\neg\exists x :(x< i) \right) \leftrightarrow \left( R_D(i,j) 
        \leftrightarrow \succ(i,j) \right)   \\
        \forall \,i,j  : 
        (\exists x : x < i) \leftrightarrow \left(R_D(i,j)  \leftrightarrow \left( \exists {i',j'} : 
        \succ(i', i)\wedge \succ^{(2)}(j', j) \wedge  R_D(i',j') \right)\right).
    \end{array}
\end{equation}

\noindent
(Translated into more intuitive/informal terms, the first line says that if $i$ is the first element of $S$, then 
$R_D(i,j)$ if and only if $j=i+1$.
The second line says that if $i$ is not the first element of $S$ then
$R_D(i,j)$ if and only if $R_D(i-1,j-2)$.)

For $R_E$, we can write:

\begin{equation}\label{eq:def_RE}
\begin{array}{l}
\forall \,i,j  : 
	 \left(\neg\exists x :(x< i) \right) \leftrightarrow \left( R_E(i,j) 
        \leftrightarrow R_D(i,j) \right)   \\
        \forall \,i,j  : 
        (\exists x : x < i) \leftrightarrow \left(R_E(i,j)  \leftrightarrow \left( \exists {i',j'} : 
        \succ(i', i)\wedge R_D(j', j) \wedge  R_E(i',j') \right)\right).
    \end{array}
\end{equation}

\noindent
(The first line says that if $i$ is the first element of $S$ then 
$R_E(i,j)$ if and only if $j=2i$.
The second line says that if $i$ is not the first element then 
$R_E(i,j)$ if and only if there are $i',j'$ such that $i=i'+1$ and 
$j=2j'$ and $R_E(i',j')$.)

It is easily seen that, to express that the relation $R_T$ is as desired, we can substitute in~\eqref{eq:def_RE} 
$R_D$ and $R_E$ by $R_E$ and $R_T$, respectively. 
Similarly, to express that $R_W$ is as desired, we can substitute $R_D$ and $R_E$ by $R_T$ and $R_W$, respectively. 

The conjunction of all the axioms we've obtained 
gives a formula $\an{Arith}(R_D,R_E,R_T,R_W)$ that is satisfied
if and only if the relations $R_D,R_E,R_T,R_W$ correctly encode the target arithmetic functions.

Given such relations we now want to express that $\log^{**} \log^{**} n$ is even, where $n:=|S|$. 
Suppose that there is an element $x\in S$ for which there exist 
$y,z \in S$ such that 
\begin{equation}
    x=W(y)\quad\text{ and }\quad y = W(z). 
\end{equation}
Then $\log^{**}\log^{**} x = z$. Now let $x$ be the largest element for which such $y$ and $z$ exist, certainly $x\leq n$. 
Now, if $x=n$, which we can check using $<$, then $z = \log^{**}\log^{**}n$ and we can query whether or not this is even by 
$\exists w : R_D(w,z)$. If instead $x<n$, then $\log^{**}(\log^{**} n) = z+1$, and we can again query whether 
this is even by $\neg \exists w : R_D(w,z)$. All of this can be formalized in a second--order 
formula $\an{EvenSize}(R_D,R_E,R_T,R_W)$ in a straightforward manner. 
So a second--order sentence $\an{Parity}$ of the form 
\begin{equation}\label{eq:evenLL}
    \an{Parity} \be \exists R_D,R_E,R_T,R_W :  \an{Arith}(R_D,R_E,R_T,R_W) \wedge \an{EvenSize}(R_D,R_E,R_T,R_W)
\end{equation}
exists such that $S\models \an{Parity}$ if and only if $\log^{**}\log^{**} |S|$ is even. 
Here $\an{Arith}$ expresses that the relations define the correct arithmetic operations, and $\an{EvenSize}$ uses 
these relations to express that $\log^{**}\log^{**} |S|$ is even. 
%

In broad lines we will do the following in the next few sections: on 
a random permutation $\Pi_n \sim \Mallows(n,1)$ we will define directed graphs using $\TOTO$ formulas, 
whose vertices are intervals $\{i,i+1,\dots,i+j\} \subseteq [n]$. 
These directed graphs will be so abundant that, with probability $1-o_n(1)$, we can find four such graphs, 
all defined on the same set $S$ of vertices/intervals of cardinality roughly $|S| \approx {\log \log n}$, that 
encode the correct relations $R_D,R_E,R_T$ and $R_W$ on this set.
We will then be able to exploit the fact that $\log^{**}(\log^{**}(\log(\log n)))$ oscillates 
between being even and odd indefinitely, giving the desired non--convergence.

\subsection{Defining directed graphs on permutations in \texorpdfstring{$\TOTO$}{TOTO}\label{sec:graphs_toto}}

We now exhibit our definition of directed graphs on $\Pi_n$ in $\TOTO$. 
In what follows all directed graphs will be simple, i.e., contain no self--loops or multiple arcs. 
Recall that we write $i<_1 j$ if and only if $i<j$ and $i<_2 j$ if and only if $\Pi_n(i) < \Pi_n(j)$. 
Inclusion of $i$ in an interval $\{a,\ldots, b\}$ can be checked by $(a = i)\vee(b=i)\vee
\left((a<_1 i) \wedge (i<_1 b)\right)$.

Let $\Pi_n\sim \Mallows(n,1)$. That is, $\Pi_n$ is selected uniformly at random from $S_n$. 
For disjoint sets $I, J \subseteq [n]$ we define the (random) set
\begin{align}
    S(I, J) \ :=&\ \{ i  : i \in \Pi_n(I), i+1 \in \Pi_n(J) \}\\[1pt]
    =&\ \{ \Pi_n(x) \mid x\in I, \Pi_n(x) +1 \in \Pi_n(J)   \}.
\end{align}

\begin{figure}
\centering
\begin{tikzpicture}

\draw (-6,3) -- (6,3); 

\draw (-6,3.1) -- (-6,2.9) node[below] {$1$}; 
\draw (6,3.1) -- (6,2.9) node[below] {$n$};

\draw(-5,3.18) -- (-1,3.18) node[midway,above] {$I$};
\draw(-5,3) -- (-5,3.18);
\draw(-1,3) -- (-1,3.18);

\draw(0,3.18) -- (4,3.18) node[midway,above] {$J$};
\draw(0,3) -- (0,3.18);
\draw(4,3) -- (4,3.18);

\draw (-6,0) -- (6,0); 
\draw (-6,0.1) -- (-6,-0.1) node[below] {$1$}; 
\draw (6,0.1) -- (6,-0.1) node[below] {$n$};

\draw[dashed,-latex] (-2,3) -- (-3,0.1); 
\node at (-3,0) [circle,fill,inner sep=1.5pt,label=below:{$i_1$}]{};
\node at (-2,3) [circle,fill,inner sep=1.5pt]{};

\draw[dashed,-latex] (0.3,3) -- (1.43-4,0.1); 
\node at (1.4-4,0) [circle,fill,inner sep=1.5pt]{};
\node at (0.3,3) [circle,fill,inner sep=1.5pt]{};

\draw[dashed,-latex] (-4.6,3) -- (0.95,0.1); 
\node at (1,0) [circle,fill,inner sep=1.5pt,label=below:{$i_2$}]{};
\node at (-4.6,3) [circle,fill,inner sep=1.5pt,label=below:{$x(I,J)\  \ \  \  $}]{};

\draw[dashed,-latex] (3.4,3) -- (1.43,0.1); 
\node at (1.4,0) [circle,fill,inner sep=1.5pt]{};
\node at (3.4,3) [circle,fill,inner sep=1.5pt,label=below:{$\qquad \quad y(I,J)$}]{};

\draw[-latex] (7,2.5) -- (7,0.5) node[midway,right] {$\Pi_n$}; 

\end{tikzpicture}
\caption{An illustration of $S(I,J)$, $x(I,J)$ and $y(I,J)$ in the case that $I$ and $J$ are disjoint intervals. The elements $i_1,i_2$ are both contained in $S(I,J)$, as $\Pi_n(\{i_1,i_2\}) +1 \subseteq \Pi_n(J)$. The element $x(I,J)$ is the smallest element in $I$ such that $\Pi_n(x(I,J)) \in S(I,J)$, and $y(I,J)$ is such that $\Pi_n(y(I,J)) -1 = \Pi_n(x(I,J))$.\label{fig:SIJ}}
\end{figure}
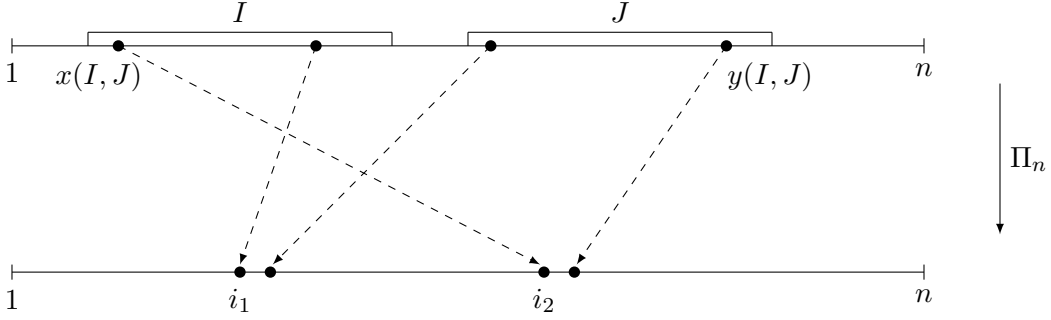

\noindent
For $I, J \subseteq [n]$ with $S(I, J) \neq \emptyset$, we define

$$ \begin{array}{rcl} 
x(I,J) & := & \min\{ x : \Pi_n(x) \in S(I,J) \} \\[1pt]
& = & \min\{ x \in I : \text{there is a $y \in J$ such that $\Pi_n(y) = \Pi_n(x)+1$} \}, \\[9pt]
y(I,J) & := & \Pi_n^{-1}(\Pi_n(x(I,J))+1).
\end{array} $$

\noindent
Thus $y(I,J)$ is the element of $J$ that gets mapped to $\Pi_n(x(I,J))+1$ by $\Pi_n$. If $S(I,J) = \emptyset$, then $x(I,J)$ and $y(I,J)$ are {\em undefined}. See Figure~\ref{fig:SIJ} for an illustration of these definitions. Membership of $x$ in $I$ or $J$ can be queried in $\TOTO$ if $I$ and $J$ are intervals using the $<_1$ ordering, and $\Pi_n(j) = \Pi_n(i) + 1$ implies simply that $j$ is the successor of $i$ in the $<_2$ ordering. So $x(I,J)$ and $y(I,J)$ can be determined in $\TOTO$ when $I$ and $J$ are intervals.

For a sequence $\Ical = (I_1, \dots, I_N)$  of subsets of $[n]$ and $J \subseteq [n]$ a single subset, we define the ordered pair $e(\Ical; J) \in \{I_1,\ldots, I_N\}\times \{I_1,\ldots, I_N\}$ as follows.

\begin{equation}
   e(\Ical; J) = (I_i , I_j) \quad \text{ if }\quad \begin{cases} 
     y(I_i, J) = \min\limits_{\ell=1,\dots, N} y(I_\ell, J), \text{ and } \\
     y(I_j, J) = \max\limits_{\ell=1,\dots, N} y(I_\ell, J). \end{cases} 
\end{equation}
(If $S(I_\ell, J) = \emptyset$ for some $1\leq \ell\leq N$ then $e(\Ical; J)$ is undefined.)

For two sequences $\Ical = (I_1, \dots, I_N), \Jcal = (J_1, \dots, J_M)$ of subsets of  $[n]$ we now define the directed graph $H(\Ical; \Jcal )$ by setting 
\begin{align}
    V\left( H(\Ical; \Jcal ) \right) & \ :=\  \{I_1,\dots, I_N\},  \\
    E\left(H(\Ical; \Jcal )\right) & \ :=\  \{ e(\Ical; J_\ell ) : \ell=1,\dots, M\}.
\end{align}
\noindent
(If $e(\Ical; J_\ell )$ is undefined for some $J_\ell$, then $J_\ell$ simply does not contribute to 
$H(\Ical; \Jcal)$. In particular $H(\Ical; \Jcal)$ is the empty graph on $N$ vertices if all $e(\Ical; J_\ell )$ are undefined.) Note that it is entirely possible that $e(\Ical; J_\ell )$ and $e(\Ical; J_{\ell'} )$ 
code the same
arc for some $\ell\neq\ell'$. This will not be an issue. 

For $A \subseteq [n]$ we define
\begin{align} 
    W_k(A) & := \{ i : \{\Pi_n(i), \Pi_n(i ) + 1, \dots, \Pi_n(i)+k-1\} \subseteq \Pi_n[A] \}, \\
    w_k(A) \label{eq:def_wk}  & := |W_k(A)|.
\end{align}
We can express that $i$ is an element of $W_k(J)$ for an interval $J$ in the following manner: For $i,j$ we can determine whether $\Pi_n(i)=\Pi_n(j)$ by the formula $ i=j$, as $\Pi_n$ is a permutation. For $i,j$ we can determine whether or not $\Pi_n(j) = \Pi_n(i)+1$ by saying that $j$ is the successor of $i$ under $<_2$. Then $i\in W_k(J)$ for an interval $J = \{a,\ldots, b\}$ if and only if $i\in J$ and there is no $c$ satisfying both $c\notin J$ and $\succ_2^{(j)}(i,c)$ 
for some $1\leq j\leq k-1$, where $\succ^{(j)}_2$ is as defined in Section \ref{subsec:TOOB_TOTO}. This may be expressed in $\TOTO$ as
\begin{equation}\label{eq:formula_wk}
    \neg \exists c \left(  \left((c <_1 a) \vee (b <_1 c)\right) \wedge \bigvee_{j=1}^{k-1} \succ_2^{(j)}(i, c)\right),
\end{equation}
Then we can also express for instance that $w_k(J) \neq 0$. 

For $J \subseteq [n]$ an interval with $W_k(J) = \{i_1,\dots, i_L \}$, where $i_1 < \dots < i_L$, the sequence
$\Ical_k(J)$ of the {\em minimal intervals} between points of $W_k(J)$ is defined as:

$$ \Ical_k(J) := ( \{i_1+1,\dots, i_2-1\}, \dots, \{i_{L-1}+1,\dots, i_L-1\} ). $$
(If $|W_k(J)| \leq 1$ then $\Ical_k(J)$ is the ``empty sequence''.) We will sometimes abuse notation and consider $\cI_k$ to be a set instead of a sequence.

The induced graph structures that we will consider will be of the form
\begin{equation}
    H(\cI_k(I);\cI_k(J))\qquad\text{ or } \qquad H(  (\cI_k(I_1),\cI_k(I_2));\cI_k(J)),
\end{equation}
where $I,I_1,I_2$ and $J$ are intervals and $(\cI_k(I_1),\cI_k(I_2))$ denotes the concatenation of $\cI_k(I_1)$ and $\cI_k(I_2)$. In Figure~\ref{fig:cherry} we give an example such that the induced graph structure is a directed cherry.

\begin{figure}
\centering
\includegraphics[scale=0.7]{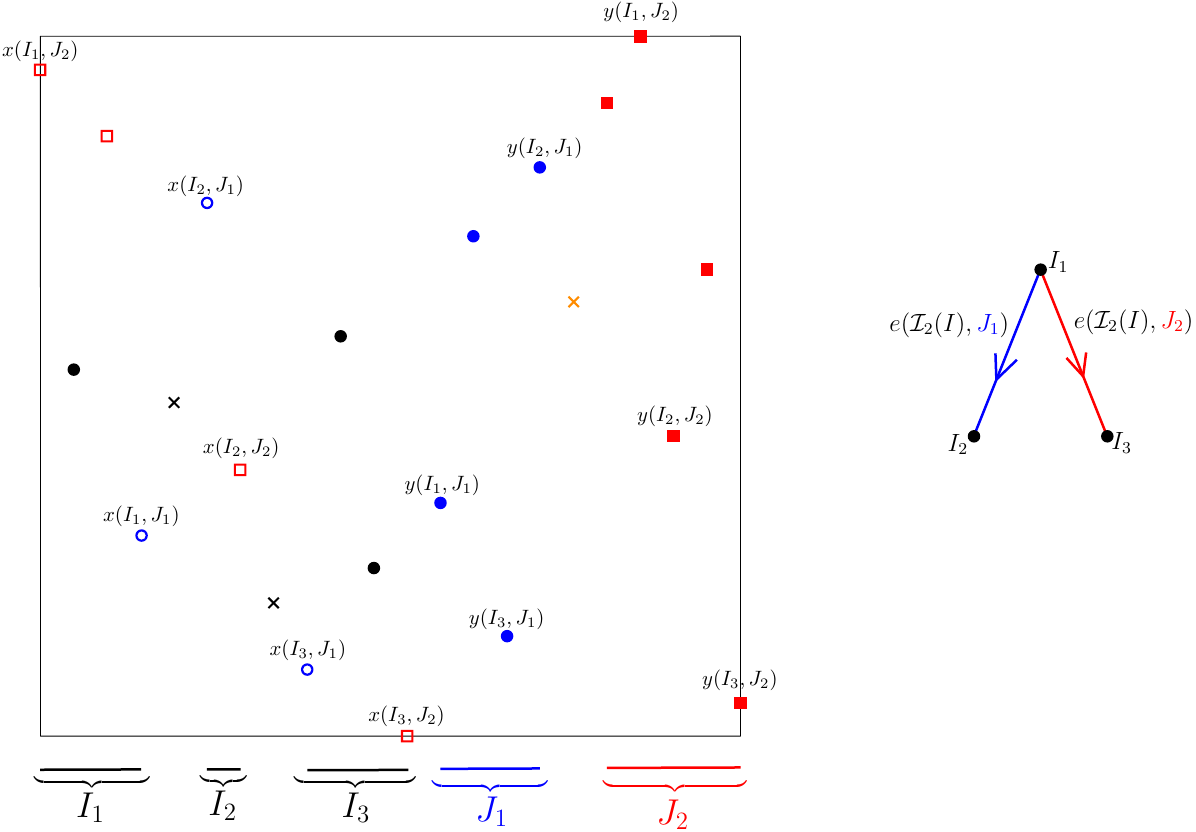}
\caption{In the figure we give an example of a permutation $\pi$ and intervals $I$ and $J$ such that $H(\cI_2(I);\cI_2(J))$ is a directed cherry. 
Take $n=22$, $I=\{1, \ldots, 12\}$, $J=\{13, \ldots, 22\}$ and permutation $\pi=21,12,19,7,11,17,9,5,3,13,6,1,8,16,4,18,14,20,22,10,15,2$. 
One may check that $\cI_2(I)=\{I_1, I_2, I_3\}$ and that $\cI_2(J)=\{J_1, J_2\}$ as shown. 
For each $i$ the point $(i, \pi(i))$ is plotted with a blue dot if $i \in J_1$, a blue circle if $i \in I$ and $\pi(i)+1 \in \pi(J_1)$, a filled red square if $i \in J_2$, an empty red square if $i \in I$ and $\pi(i)+1 \in \pi(J_2)$, a black cross if $i \in W_2(I)$ and an orange cross if $i \in W_2(J)$. 
}
\label{fig:cherry}
\end{figure}

As noted, the vertices of these graphs are elements of $\cI_k(I)$ and $(\cI_k(I_1),\cI_k(I_2))$, respectively. 
Now, these intervals are of the form $\{i_k+1,\ldots, i_{k+1}-1\}$ where $i_k$ and $i_{k+1}$ are both elements 
of $W_k(I')$ for $I'$ equal to one of $I,I_1$ or $I_2$. Membership in $W_k(I')$ can be determined via \eqref{eq:formula_wk}. 
So an interval $\{a,\ldots, b\}$ is an element 
of for instance $\cI_k(I)$ if and only if $a-1$ and $b+1$ are two consecutive elements of $W_k(I)$. Then if we want to check whether $I'\in \cI_k(I)$ and $I'' \in \cI_k(I)$ are connected by an arc, we simply need to check whether there is an interval $J'\in \cI_k(J)$ for which $y(I',J') = \min_{\ell=1,\ldots, |\cI_k(I)|}y(I_\ell,J')$ and $y(I'',J') = \max_{\ell=1,\ldots, |\cI_k(I)|}y(I_\ell,J')$. The elements $y(I_\ell,J')$ can be determined in $\TOTO$ as we have already seen, and we can compare $y(I_\ell,J')$ and $y(I_{\ell'},J')$ by the total ordering $<_1$.

We need one more ingredient in order to describe the final sentence $\phi$ whose probability does not converge 
in Proposition~\ref{prop:oscilate}. Given two disjoint sets $A,B$ we say that a directed graph $G$ with vertex set $A\uplus B$ is a \emph{directed matching from $A$ to $B$} if every arc of $G$ connects a vertex in $A$ to a vertex in $B$. 
Given disjoint intervals $I,I'$ we associate to them the intervals $\cI_k(I)=(I_1,\ldots,I_N)$, respectively $\cI_k(I')=(I'_1,\ldots, I'_{N'})$. We will need to determine if $N > N'$. We construct a formula $\an{Bigger}(I,I')$ satisfied with high probability by two intervals $I,I'$ whenever $N' < \log\log n <  N < 2\log\log n$. This formula may be written in $\TOTO$ as $\exists a,b$ such that, writing $J\be (a,b)$, the following holds in $H((\cI_k(I),\cI_k(I'));\cI_k(J))$:

\begin{itemize}
    \item All vertices in $\cI_k(I')$ have outdegree $1$;
    \item All vertices in $\cI_k(I)$ have indegree at most $1$;
    \item There is a vertex in $\cI_k(I)$ with indegree $0$;
    \item All arcs start at a vertex in $\cI_k(I')$ and end at a vertex in $\cI_k(I)$.
\end{itemize}

This directed graph is then necessarily a directed matching from $\cI_k(I')$ to $\cI_k(I)$ that saturates the vertices in $\cI_k(I')$ but not those in $\cI_k(I)$. Such a matching exists if and only if $|\cI_k(I)| > |\cI_k(I')|$.

The following two lemmas represent the ``meat'' of the proof of Proposition~\ref{prop:oscilate}. 
They will be proved in Section \ref{sec:prob_uniform}.

\begin{lem}\label{lem:graph_exists}
    For every fixed $k\geq 3$, with probability $1 - e^{-n^{\Omega(1)}}$ there is an interval $I$ with $N\be |\cI_k(I)|$ satisfying $\log \log n < N  < 2 \log\log n$ such that for every directed graph $G$ with vertex 
    set $\cI_k(I)$ there exists an interval $J \subseteq [n]$ such that 
    $$ H( \Ical_k(I); \Ical_k(J) ) = G,$$
    and moreover $w_{k+1}(I') = 0$ and $w_{k-1}(I') > 0$ for all $I'\in \Ical_k(I)$.
\end{lem}

\begin{lem}\label{lem:matching}
For $k$ sufficiently large but not depending on $n$, with probability $1-O(n^{-998})$, for every two disjoint intervals $I_1,I_2 \subseteq [n]$ satisfying 
\begin{itemize}
    \item $0<|\Ical_k(I_1)|  < \log\log n < |\Ical_k(I_2)| < 2 \log\log n$;
    \item $w_{k+1}(I_1)= w_{k+1}(I_2)=0$;
    \item $w_{k-1}(I') > 0$ for all $I\in \cI_k(I_1)\cup \cI_k(I_2)$,    
\end{itemize}
there exists an interval $J \subseteq [n]$ such that $H( (\Ical_k(I_1),\Ical_k(I_2));\FS{\Ical_2(} J) )$ is a directed matching between $\Ical_k(I_1)$ and $\Ical_k(I_2)$ that saturates $\Ical_k(I_1)$.
\end{lem}

\subsection{Proof of Proposition~\ref{prop:oscilate} assuming Lemmas~\ref{lem:graph_exists} and~\ref{lem:matching}}

Given an interval $I$ we have seen in the previous section how to define a 
directed graph on $\cI_k(I) = (I_1,\ldots, I_N)$ by selecting an interval $J$. Note that the $I_i$ are disjoint intervals, so 
we inherit a total ordering on them from the 
 relation $<_1$ on the domain of $\Pi_n$. This allows us  to distinguish for instance $I_1,I_2$ and $I_N$ in $\TOTO$.

We will apply the ideas developed in Section \ref{sec:arithSO} to sets of the form $\cI_k(I)$. 
Such sets are subsets of the collection of all intervals in $[n]$, and (given $I$) we can determine membership 
of an interval $\{a,\ldots, b\}$ in $\cI_k(I)$ by a $\TOTO$ formula as described in the previous section. 
Any four intervals $J_D,J_E,J_T,J_W$ induce four directed 
graphs on $\cI_k(I)$, and we may check whether these graphs satisfy $\an{Arith}$ as given in Section \ref{sec:arithSO}. If so, then $\an{EvenSize}(J_D,J_E,J_T,J_W)$ holds if and only if $\log^{**}\log^{**}|\cI_k(I)|$ is even. 
Lemma \ref{lem:graph_exists} will allow us to deduce that with high probability there are 
$J_D,J_E,J_T,J_W$ satisfying $\an{Arith}$ so that we can then check the parity of $\log^{**}\log^{**}|\cI_k(I)|$. 

We now describe the sentence $\phi$ appearing in Proposition~\ref{prop:oscilate}. It will be of the form
\begin{equation}\label{eq:non_con_phi}
    \phi =\exists x_1,y_1,\ldots,x_5,y_5 \left(\phi_0 \wedge   \forall x_6,y_6,\ldots,  x_{10},y_{10}  (\phi_1\to \phi_2)   \right).
\end{equation}
We use the ten pairs of elements to define intervals 
$I,J_D,J_E,J_T,J_W,I',J_D',J_E',J_T',J_W'$ by $I = (x_1,y_1),\ldots, J_W' = (x_{10},y_{10})$. We define $\phi_{0}(x_1,y_1,\ldots, x_5,y_5)$ as the $\TOTO$ formula such that $(\Pi_n,i_1,j_1,\ldots, i_5,j_5)\models \phi_0$ if and only if
\begin{itemize}
    \item $w_{k+1}(I) =0$ and $w_{k-1}(\tilde{I}) > 0$ for all $\tilde{I}\in \cI_k(I)$, and;
  \item $\an{Arith}$ is satisfied wrt.~the set $S := \Ical_k(I)$
  and the relations $R_D := \Hcal( \Ical_k(I); \Ical_k(J_D) ), R_E := \Hcal( \Ical_k(I); \Ical_k(J_E) ), 
  R_T := \Hcal( \Ical_k(I); \Ical_k(J_T) ), R_W := \Hcal( \Ical_k(I); \Ical_k(J_W) )$, and;
    \item $\an{EvenSize}$ is satisfied wrt.~the same set and relations.
\end{itemize}
We define the formula $\phi_{1}(x_6,y_6,\ldots, x_{10},y_{10})$ that holds if and only if 
the first two conditions above are satisfied wrt~$I',J_D',J_E',J_T',J_W'$, but 
$\an{EvenSize}$ is {\em not} satisfied wrt.~$I',J_D',J_E',J_T',J_W'$. 
Finally, we let $\phi_2 (x_1,y_1,x_6,y_6)$ express that $\an{Bigger}$ holds wrt.~$\Ical_k(I),\Ical_k(I')$.

To recap, in informal terms the formula $\varphi$ asks for an interval $I$ such that 
(some intervals/relations certify that) $\log^{**}\log^{**} |\Ical_k(I)|$ is even, and moreover there is no
interval $I'$ such that (some intervals/relations certify that) both $\log^{**}\log^{**} |\Ical_k(I')|$ is odd
and $|\Ical_k(I)| < |\Ical_k(I')|$.

\begin{proof}[Proof of Proposition~\ref{prop:oscilate}]
    We pick an $n$ satisfying $W^{(2)}(\log^{**}\log^{**}n-1)<{\log\log n}$. \\
    By $W^{(2)}(\log^{**}\log^{**}n-1)<{\log\log n}$ and Lemma~\ref{lem:ineq_w_log_star} we have
    \begin{equation}
         \log^{**}\log^{**}n - 1 < \log^{**}\log^{**} \ceil{\log \log n} \leq \log^{**}\log^{**}n.
    \end{equation}
    As the three quantities above are all integers, the second inequality is in fact an equality. 
    Thus, for any integer $m$ satisfying $\log\log n \leq m\leq n$ we have $\log^{**}\log^{**}m=\log^{**}\log^{**}n$. 

    We let $k$ be such that the conclusions of Lemmas~\ref{lem:graph_exists} and~\ref{lem:matching} hold.

    Consider first the case $\log^{**}\log^{**}n$ even. By Lemma~\ref{lem:graph_exists} and $\log^{**}\log^{**} m$ being even for all $\log\log n \leq m\leq n$, with probability $1-e ^{-n^{\Omega(1)}}$ there exist intervals $I,J_D,J_E,J_T$ and $J_W$ such that $\phi_{0}(I,J_D,J_E,J_T,J_W)$ holds and such that $\log\log n < |\cI_k(I)| < 2\log\log n$. Now let $I',J'_D,J'_E,J'_T$ and $J'_W$ be five intervals such that $\phi_{1}(I',J'_D,J'_E,J'_T,J'_W)$ holds (if such intervals do not exist then $\Pi_n\models \phi$ and we are done). Then $N' \be | \cI_k(I')|$ must be smaller than ${\log \log n}$, as otherwise $\log^{**}\log ^{**} N'$ would be even. Thus indeed $N' <  \log \log n < N < 2\log\log n$. By Lemma~\ref{lem:matching}, 
    with probability $1-O(n^{-998})$ there exists an interval $J$ such that $H((  \cI_k(I), \cI_k(I')); J)$  is a directed matching between $\cI_k(I)$ and $\cI_k(I')$ that saturates the vertices of $\cI_k(I')$ but not the vertices of $\cI_k(I)$. That is, $\an{Bigger}(I,I')$ holds with at least this probability. So if $n$ is even $\bP \left(   \Pi_n \models \phi \right) = 1- O (n^{-998})$.

    We now consider the case $\log^{**}\log^{**}n$ odd. Assume that there are intervals $I,J_D,J_E,J_T,J_W$ such that $\phi_{0}(I,J_D,J_E,J_T,J_W)$ holds, otherwise we immediately have $\Pi_n \not \models \phi$. We must then have $|\cI_k(I)| < {\log \log n}$. As before, with probability $1-e^{-n^{\Omega(1)}}$ there are $I'$, $J'_D$, $J'_E$, $J'_T$ and $J'_W$ such that $\phi_{1}(I',J'_D,J'_E,J'_T,J'_W)$ holds and $ \log\log n <  |\cI_k(I')|  < 2 \log \log n$. But then $|\cI_k(I')| > |\cI_k(I)|$, so $\an{Bigger}(I,I')$ cannot be satisfied. 
\end{proof}

\subsection{The probabilistic component\label{sec:prob_uniform}}

\noindent
For $B \subseteq [n]$ and $\tau : B \to [n]$ an injection, we will write
\begin{align}
    \bP_\tau( . ) &\be \bP( . \,|\, \Pi_n(i) = \tau(i) \text{ for all $i\in B$} ),\\
    \Ee_\tau( . ) &\be \Ee( . \,|\, \Pi_n(i) = \tau(i) \text{ for all $i\in B$} ). 
\end{align}
Recall that we defined 
\begin{align} 
    w_k(A) & := |\{ i : \{\Pi_n(i), \Pi_n(i ) + 1, \dots, \Pi_n(i)+k-1\} \subseteq \Pi_n(A) \}|.
\end{align}
We define 
\begin{equation}\label{eq:def_Y_j}
    Y_i(A) \be \ind{ \{ \{ i,\ldots, i+k-1 \} \,\subseteq\, \Pi_n(A) \} }.
\end{equation}
We note that 

$$w_k(A) = Y_1(A)+\ldots + Y_{n-k+1}(A).$$

\begin{lem}\label{lem:vkexp}
If $k$ is fixed, and $A, B \subseteq [n]$ are disjoint sets of cardinality $1 \ll |A| \ll n$ and $|B| \,\leq\, n / \log n$, and 
$\tau : B \to [n]$ is an arbitrary injection, then 

$$ 
  \begin{array}{l} 
   \Ee_\tau w_k(A) = (1+o(1)) n^{1-k} |A|^k, \\
   \\
   \Ee_\tau w_k(A)^2 \leq 
   \left(\Ee_\tau w_k(A) \right)^2
   + (2k-1) \cdot \Ee_\tau w_k(A).
   \end{array}
$$

\noindent
(where the $o(.)$ is uniform over all such $B$ and $\tau$.)
\end{lem}

\begin{proof}
Let us denote by $J$ the set of all $1\leq j\leq n-k+1$ for which $\{j, \dots, j+k-1\} \cap \text{Im}(\tau) = \emptyset$.
For notational convenience, we will write $Y_j$ instead of $Y_j(A)$.
We note that

$$ \Ee_\tau Y_j 
= \left\{\begin{array}{cl}
  \frac{|A|(|A|-1) \dots (|A|-k+1)}{(n-|B|)\dots(n-|B|-k+1)} & \text{ if $j\in J$,  } \\
   0 & \text{ otherwise.} 
  \end{array}\right..
$$ 

\noindent
Hence

$$ \Ee_\tau w_k(A) = |J| \cdot \frac{|A|(|A|-1) \dots (|A|-k+1)}{(n-|B|)\dots(n-|B|-k+1)}
= (1+o(1)) \cdot n^{1-k} |A|^k, $$

\noindent
using that $|B| \,\leq\, n / \log n$ and $(n-k(|B|+1)) \leq |J| \leq n$. The $o(1)$ term above is 
uniform for all $B$ by the uniform bound $B \,\leq\, n /\log n$. 

For the second moment, we remark that

\begin{equation}
 \Ee_\tau Y_iY_j  \label{eq:BEY}
= \left\{ \begin{array}{cl} 
           \frac{|A|(|A|-1) \dots (|A|-2k+1)}{(n-|B|)\dots(n-|B|-2k+1)} & \text{ if $i,j\in J$ and $|i-j|>k$, } \\
        0 & \text{ if }  i \not\in J \text{ or } j\not\in J , \\
         \leq \Ee_\tau Y_i & \text{ if } |i-j| \leq k.
         \end{array} \right.
\end{equation}

\noindent
and we also remark that

$$ \frac{|A|(|A|-1) \dots (|A|-2k+1)}{(n-|B|)\dots(n-|B|-2k+1)} \leq 
\left( \frac{|A|(|A|-1) \dots (|A|-k+1)}{(n-|B|)\dots(n-|B|-k+1)}  \right)^2, $$

\noindent
since $1 \ll |A|, |B| \ll n$ so that $(|A|-x)/(n-|B|-x) \leq (|A|-x-1)/(n-|B|-x-1)$ for all $x=1,\dots,2k-1$.

We see that

$$ \begin{array}{rcl} \Ee_\tau w_k(A)^2 
& = & \sum_i\sum_j \Ee_\tau Y_iY_j \\[4pt]
& \leq & |J|^2 \frac{|A|(|A|-1) \dots (|A|-2k+1)}{(n-|B|)\dots(n-|B|-2k+1)} + 
(2k-1) \sum_i \Ee_\tau Y_i \\[4pt]
& \leq & \left(\Ee_\tau w_k(A) \right)^2 + (2k-1)\cdot\Ee_\tau w_k(A).
\end{array} $$
\end{proof}

The following two lemmas establish that if $w_k(I)>0$ then $I$ is has to be relatively large and if $w_k(I)= 0$ then $I$ 
has to be relatively small. 

\begin{lem}\label{lem:zero_w}
    For each fixed $k\geq 2$, with probability $1-O(n^{-998})$, for every interval $I \subseteq [n]$ of  length at most $ n^{1-1000/k}$ we have $w_k(I) = 0$.
\end{lem}

\begin{proof} 
    It suffices to show the result for all intervals $I$ of length $\floor {n^{1-1000/k}}$ as $w_k(I) = 0$ implies 
    $w_k(I')=0$ for all $I'\subseteq I$. By Markov's inequality and Lemma~\ref{lem:vkexp} with $B=\emptyset$ we have 
    for any given interval $I\subseteq [n]$ with $|I| = \floor{  n^{1-1000/k}}$ that
    \begin{equation}
        \bP \left(  w_k(I) > 0 \right)  \leq \bE w_k(I) = (1+o(1)) n^{1-k}|I|^k \leq (1+o(1)) n^{1-k}n^{k -1000}=O \left( n^{-999} \right) .
    \end{equation}
    There are $O(n)$ intervals in $[n]$ of length $\floor{ n^{1-1000/k} }$, so the union bound gives the desired result.
\end{proof}

\begin{lem}\label{lem:positive_w}
    For each fixed $k$ and $0 < \eps < 1/k$, with probability at least $1 - e^{-\Omega(n^{1 - k \epsilon})}$, every 
    interval $I \subseteq [n]$ of length at least $ \ceil{n^{1-\eps}}$ satisfies $w_k(I) > 0$.
\end{lem}

\begin{proof}
    Since there are $O(n^2)$ possible intervals, it is enough to show that for any fixed interval $I$ of 
    length $\ceil{n^{1-\eps}}$ we have $\bP( w_k(I) = 0 ) \leq  e^{-\Omega(n^{1 - k \epsilon})}$.

    So let $I\subseteq [n]$ be an interval of length $|I|\geq \ceil{n^{1-\epsilon}}$. 
    Let $J\subseteq [n]$ be a random subset where we independently include $i$ in $J$ with probability
    \begin{equation}
        \bP \left(i \in J\right) = \frac{|I|}{n},\quad\text{for all }i\in [n].
    \end{equation}
    Denote by $E$ the event $\{|J| = |I|\}$. By Lemma~\ref{lem::binom_max} we have 
    $\bP \left( E \right) \geq 1/(n+1)$. Define 
    for $1\leq j \leq n-k + 1$ also the random variable
    \begin{align}
        X_j &\be \ind{\{ \{j,\ldots, j+k-1\} \, \subseteq\, J \}}.
    \end{align}
    If $E$ holds then $J\isd \Pi_n(I)$. That is, conditional on the event $E$ both $I$ and $J$ are uniformly 
    selected random subsets of $[n]$ of size exactly $\ceil{n^{1-\eps}}$. 
    We have 
    
    \begin{equation}\begin{array}{rcl}
        \bP \left( w_k(I) = 0 \right) & = & \bP \left( Y_j(I) = 0 \text{ for all } j \in [n-k+1]\right) =  
        \bP \left(  X_j = 0 \text{ for all } j\in [n-k+1] \Big|\, E\right) \\[2ex] & \leq & \displaystyle
        \frac{ \bP \left( X_j = 0 \text{ for all } j\in [n-k+1] \right)}{\bP \left( E \right) } 
         \leq  (n+1) \, \bP \left( X_j = 0 \text{ for all } j\in [n-k+1] \right).
       \end{array}
    \end{equation}
    
    Let $S = \{ ki\mid i \in \bN \} \cap [n-k+1]$. The $X_j$ are independent for $j\in S$. Therefore
    
    \begin{align}
         \bP \left( w_k(I) = 0 \right)  & \leq (n+1) \Pee( X_j = 0 \text{ for all } j \in S )
  = (n+1)  \left( 1 -  \frac{|I|^k}{n^k} \right) ^{|S|} \\ &\leq (n+1) \exp \left\{ -|S|\frac{|I|^k}{n^k} \right\} .
    \end{align}
    By $|S| =(1+o(1)) n/k$ and $|I|/n \geq n^{-\epsilon}$ we conclude that
    \begin{align}
        \bP \left( w_k(I) = 0 \right)  &\leq \exp \left\{ - (1+o(1))\frac{n^{1-k\epsilon}}{k} \right\}.
    \end{align}
\end{proof}

\begin{lem}\label{lem:vtexlem}
    For $k\geq 3$ and $\eps>0$, let $I \subseteq [n]$ be an interval of length 
    $\lfloor n^{1-1/k} \cdot \left(\log\log n\right)^{3/k} \rfloor$, and let 
    $B \subseteq [n]$ be disjoint from $I$ with $|B| \leq n / \log n$ and let $\tau : B \to [n]$ be an
    arbitrary injection. Let $E$ denote the event that 
    \begin{enumerate} 
        \item\label{itm:vtexlem1} $(1-\eps)\left(\log\log n\right)^{3} < w_k(I) < (1+\eps)\left(\log\log n\right)^{3}$, and;
        \item\label{itm:vtexlem2} $w_{k+1}(I) = 0$, and; 
        \item\label{itm:vtexlem3} $w_{k-1}(I') > 0$ for each $I' \in \Ical_k(I)$, and;
        \item\label{itm:vtexlem4} $|I'| > n^{ .51 }$ for all $I' \in \Ical_k(I)$.
    \end{enumerate}
    Then $\bP_\tau( E ) = 1 - o(1)$, where the $o(1)$ term is uniform over the choice of $B$ and $\tau$.
\end{lem}

\begin{proof} 
    By Chebyshev's inequality and Lemma~\ref{lem:vkexp} we have
    \begin{align}
        \bP_\tau \left( \left|  w_k(I) - \bE_\tau w_k(I)   \right| > \frac{\epsilon}{2} (\log\log n)^3 \right) &\leq \frac{ (2k-1) \bE_\tau w_k(I) }{ \left( \frac{\epsilon}{2} (\log\log n)^3 \right)^2   },
    \end{align}
    where
    \begin{equation}
        \bE _\tau w_k(I) = (1 + o(1)) n^{1-k}|I|^k = (1+o(1)) (\log\log n)^3.
    \end{equation}
    Thus
    \begin{equation}
        \bP_\tau \left( \left|  w_k(I) - (1+o(1)) (\log\log n)^3   \right| > \frac{\epsilon}{2} (\log\log n)^3 \right) < \frac{8k}{  \epsilon^2  (\log \log n)^3 } = o(1),
    \end{equation}
    showing Part~\ref{itm:vtexlem1}.

    A simple application of Markov's inequality together with Lemma~\ref{lem:vkexp} gives
    \begin{equation}
        \bP_\tau \left( w_{k+1} (I) > 0 \right) \leq \bE_\tau w_{k+1}(I) = (1+o(1)) n^{-k} \left( n^{1-\frac{1}{k}}   (\log \log n)^{\frac{3}{k}} \right)^{k+1} = o(1),
    \end{equation}
    giving Part~\ref{itm:vtexlem2}.
    
    We now look at Parts~\ref{itm:vtexlem3} and~\ref{itm:vtexlem4}. 
    We will show that for any fixed $\delta > 1/k$ we have 
    
    \begin{equation}\label{eq:clmprp}
        \bP_\tau \left(  |I'| \geq n^{1 - \delta} \text{ for all }I' \in \cI_k(I) \right) = 1 - o(1).
    \end{equation}
    
    \noindent
    This will imply Part~\ref{itm:vtexlem3} with an application Lemma~\ref{lem:positive_w} with 
    e.g.~$\delta = \frac12 ( \frac{1}{k} + \frac{1}{k-1})$ satisfying $\frac{1}{k} < \delta < \frac{1}{k-1}$. 
    Part~\ref{itm:vtexlem4} will follow by setting $\delta = .49$. 
    
    Note that $w_k(I)$ depends only on the image $\Pi_n[I]$. 
    Conditioning on $\Pi_n[I] = I'$ for some $I'\subseteq [n] \setminus \tau[B]$
    (in addition to conditioning on $\Pi_n(i) = \tau(i)$ for all $i \in B$), the restriction 
    $\Pi_n \upharpoonright I$ is distributed 
    uniformly over all bijections $I\to I'$. 
    From this it can be seen that, conditional on the value of $w_k(I)$, $W_k(I)$ is a subset of $I$ of size 
    $w_k(I)$ selected uniformly at random. 
    Let $F$ be the event $\{w_k(I) < 2 (\log \log n)^3 \}$. 
    It follows that, conditional on $F$, the probability that $|I'|< n^{1-\delta}$ for some $I'\in \cI_k(I)$ is 
    at most the probability that out of $\lfloor2(\log \log n)^3\rfloor$ points selected uniformly at random 
    from $I$, there is a pair $i,j$ with $|i-j|\leq n^{1-\delta}$. 
    There are at most $|I|\cdot n^{1-\delta}$ such pairs in $I$. The probability 
    that two points chosen uniformly at random from $I$ (without replacement) have distance $\leq n^{1-\delta}$ is thus at most
    
    $$ \frac{|I|\cdot n^{1-\delta}}{{|I|\choose 2}} = \frac{2 n^{1-\delta}}{|I|-1} = n^{1/k-\delta+o(1)}. $$ 
    
    Let 
    
    $$ X := \left|\left\{ \{i,j\} : i,j \in W_k(I), i\neq j, |i-j|\leq n^{1- \delta}\right\}\right|, $$
    
    \noindent
    denote the number of pairs of elements of $W_k(I)$ that are at distance $\leq n^{1-\delta}$.
    It follows that 
    $$ \begin{array}{rcl}
        \Pee_\tau( X > 0 | F ) & \leq & \Ee_\tau( X | F ) \\[2ex]
        & \leq &  {\lfloor2(\log \log n)^3\rfloor\choose 2} \cdot n^{1/k-\delta+o(1)} \\[2ex]
        & = & n^{1/k-\delta+o(1)} \\
        & = & o(1), 
\end{array} $$
    
    \noindent
    where the last inequality holds by choice of $\delta > 1/k$.
    By Part~\ref{itm:vtexlem1}, $\bP_\tau \left(F \right) = 1-o(1)$, which gives that 
    
    $$ \Pee_\tau( X > 0 ) = 
    \Pee_\tau( X > 0 | F )\cdot\Pee_\tau(F) + \Pee_\tau( X > 0 | \overline{F} )\cdot\Pee_\tau( \overline{F} )
    \leq \Pee_\tau( X > 0 | F ) + \Pee_\tau( \overline{F} )
    = o(1), $$
    
    \noindent
    proving~\eqref{eq:clmprp} and hence Parts~\ref{itm:vtexlem3} and~\ref{itm:vtexlem4}.
\end{proof}

\begin{cor}\label{cor:at_least_half}
    Let $k\geq 3$ be arbitrary, $0 < \eps < 1/k$ and let $I_1,\dots,I_N$ be disjoint intervals of 
    length $\lfloor n^{1-1/k} \cdot \left(\log\log n\right)^{3/k} \rfloor$ where $N = \ceil{n^\eps}$. 
    With probability $1-e^{-\Omega(n^{\epsilon})}$, at least half of the $I_i$ satisfy 
    Properties~\ref{itm:vtexlem1}-\ref{itm:vtexlem4} of Lemma~\ref{lem:vtexlem}.
\end{cor}

\begin{proof}
    Let $E_i$ denote the event that $I_i$ satisfies 
    properties~\ref{itm:vtexlem1}-\ref{itm:vtexlem4} of Lemma~\ref{lem:vtexlem}. 
    Let $S=\{s_1,\dots,s_K\}\subseteq [N]$ be a subset of cardinality $K>N/2$.
    We can assume without loss of generality that $n$ is sufficiently large for the $o(1)$ term (uniform over
    the choice of $B, \tau$) in 
    Lemma~\ref{lem:vtexlem} to be $<1/100$.
    
    We note that for $i=1,\dots,K$:
    
    $$ \Pee( \overline{E_{s_i}} | \overline{E_{s_1}}\cap\dots\cap\overline{E_{s_{i-1}}} )
    < 1/100. $$
    
    This follows from Lemma~\ref{lem:vtexlem}, noting that 
    whether or not the event $\overline{E_{s_1}}\cap\dots\cap\overline{E_{s_{i-1}}}$ holds
    is determined completely by the restriction of $\Pi_n$ to $I_{s_1}\cup\dots\cup I_{s_{i-1}}$, a set 
    of cardinality at most 
    
    $$|I_1|+\dots+|I_N| = N \cdot \lfloor n^{1-1/k} \cdot \left(\log\log n\right)^{3/k} \rfloor
    = n^{1+\eps-1/k+o(1)} \leq \frac{n}{\log n}, $$
    
    \noindent
    where the last inequality holds for sufficiently large $n$ (and by choice of $\eps$).
    It follows that 
    
    $$ \Pee( \overline{E_{s_1}}\cap\dots\cap\overline{E_{s_{K}}} ) 
    = \Pee(  \overline{E_{s_1}} ) \cdot \Pee(  \overline{E_{s_2}} |  \overline{E_{s_1}} ) \cdot \dots \cdot 
    \Pee(  \overline{E_{s_K}} |  \overline{E_{s_1}}\cap\dots\cap\overline{E_{s_{K-1}}})
    \leq (1/3)^{N/2}, $$
    
    \noindent
    for each set $S \subseteq [N]$ of cardinality $>N/2$.
    Hence, if $F$ denotes the event that at least $N/2$ of the events $E_1,\dots,E_N$ hold then
    
    $$ \Pee( \overline{F} ) \leq 2^N \cdot (1/100)^{N/2} = (1/5)^N = \exp[ -\Omega(n^\eps) ]. $$
    
    \noindent
    (Here $2^N$ is a crude upper bound on the number of subsets of $[N]$ of cardinality $>N/2$.)
\end{proof}

\begin{lem}\label{lem:shrink_CkI}
    Let $I$ be an interval satisfying Conditions~\ref{itm:vtexlem2}-\ref{itm:vtexlem4} of 
    Lemma~\ref{lem:vtexlem}. 
    Let $J$ be obtained from $I$ by deleting an arbitrary element. Then $J$ satisfies Conditions~\ref{itm:vtexlem2}-\ref{itm:vtexlem4} of Lemma~\ref{lem:vtexlem} and 
    \begin{equation}\label{eq:gjenhhnnntny}
        |\cI_k(I)| - k \leq |\cI_k(J)| \leq |\cI_k(I)|.
    \end{equation}
\end{lem}

\begin{proof}
    As $J \subseteq I$ we have $W_t(J) \subseteq W_t(I)$ for all $t\geq 1$, by definition of $W_t(.)$. 
    IN particular $w_{k+1}(I) = 0$ implies $w_{k+1}(J) = 0$. In other words, Part~\ref{itm:vtexlem4} holds
    for $J$ if it holds for $I$. 
    Similarly, any interval of $J' \in \Ical_k(J)$ must contain
    some interval $I' \in \Ical_k(I)$, showing that Part~\ref{itm:vtexlem4} holds for $J$ if it holds for $I$.
    
    To see that~\eqref{eq:gjenhhnnntny}, suppose $J := I \setminus \{j\}$ for some $j$. 
    Then there are at most $k$ elements $i\in I$ such that $\Pi_n(j)\in \{\Pi_n(i),\ldots, \Pi_n(i)+k-1\}$. 
    Thus, by definition of $W_k(.)$, we have $|W_k(I)\setminus W_k(J)| \leq k$. 
\end{proof}

Recall that we defined
\begin{equation}
    S(I, J) = \{ i : i \in \Pi_n(I), i+1 \in \Pi_n(J) \}.
\end{equation}

\begin{lem}\label{lem:S}
Let $\eps,\delta > 0$ be arbitrary (but fixed). 
With probability at least $1-e^{-n^{\Omega(1)}}$, for every two disjoint intervals $I_1, I_2 \subseteq [n]$ of 
length at least $n^{\frac12+\eps}$, we have 

$$ (1-\delta)\cdot\frac{|I_1|\cdot|I_2|}{n} \leq |S(I_1,I_2)| \leq (1+\delta)\cdot\frac{|I_1|\cdot|I_2|}{n}. $$ 

\end{lem}

\begin{proof}
Fix disjoint intervals $I_1, I_2$ of the stated sizes. 
Note $J_1 := \Pi_n(I_1), J_2 := \Pi_n(I_2)$ are disjoint, random sets of cardinalities 
exactly $k_1 := |I_1|$, resp.~$k_2 := |I_2|$, every pair of disjoint sets of the correct 
cardinalities being equally likely.

We approximate these by sets with a random number of elements, generated as follows. 
Write $p_i := k_i/n$ for $i=1,2$. Let $Z_1, \dots, Z_n$ be i.i.d.~with common distribution given by 

$$ \bP( Z_j = x ) = \left\{\begin{array}{cl} 1-(p_1+p_2) & \text{ if $x=0$, } \\
              p_1 & \text{ if $x=1$, } \\
                      p_2 & \text{ if $x=2$. } \\
                     \end{array}\right.
$$

\noindent 
We now set $J_1' := \{ j : Z_j = 1 \}, J_2' := \{ j : Z_j = 2 \}$. Conditional on the event 

$$E := \{|J_1'| = k_1, |J_2'|=k_2\}, $$

\noindent
the pair $(J_1', J_2')$ has the same distribution as $(J_1, J_2)$.
Namely, the uniform distribution over all pairs of disjoint sets of the correct cardinalities $k_1, k_2$. 
Lemma~\ref{lem::binom_max} gives the crude bound

\begin{equation}\label{eq:Eineq}
\begin{array}{rcl}
     \bP( E ) & = & \bP \left( |J'_1|=k_1 \right) \bP \left( |J_2'|=k_2 \Big| |J'_1|=k_2  \right)  \\[2ex]
     & = & \displaystyle \Pee( \Bi(n,p_1)=k_1 ) \cdot \Pee\left( \Bi\left( n-k_1, p_2/(1-p_1) \right) = k_2 \right) \\[2ex]
     & = & \displaystyle \Pee\left( \Bi\left(n,k_1/n\right)=k_1 \right) \cdot 
     \Pee\left( \Bi\left( n-k_1, k_2/(n-k_1) \right) = k_2 \right) \\[2ex]
     & \geq & \displaystyle \frac{1}{(n+1)^2}.
\end{array} 
\end{equation}

Next, we consider the size of 

$$S' \be \{ j : j \in J_1', j+1 \in J_2' \}. $$

\noindent
Write $X_j := \ind{\{j\in S'\}}$, so that $|S'| = X_1 + \dots + X_{n-1}$. We have $\Ee X_j = p_1p_2 =: p$ and 

\begin{equation}\label{eq:SaccIneq} \Ee |S'| = (n-1)p = \frac{|I_1|\cdot|I_2|}{n} - p. \end{equation}

\noindent
In particular $\bE |S'| = (1+o(1)) n^{2\eps}$ by our assumption on the sizes $|I_1|, |I_2| \geq n^{1/2+\eps}$. Now $|S'|$ is not a binomial random variable as the $X_j$ are (mildly) dependent -- it is for instance not possible to have two consecutive integers in $S'$. Note that $X_j$ is however (mutually) independent from the collection of all $X_i$ with $|i-j| \geq 2$. So, writing 

$$ X_{\text{even}} := \sum_{j \text{ even}} X_j, \quad X_{\text{odd}} := \sum_{j \text{ odd}} X_j, $$

\noindent
we see that $|S'| = X_{\text{even}} + X_{\text{odd}}$ and $X_{\text{even}} \isd \Bi( \lfloor (n-1)/2\rfloor, p ), X_{\text{odd}} \isd \Bi( \lceil (n-1)/2\rceil, p )$. 
By the Chernoff bound we have

$$ \bP \left( |X_{\text{even}}-\Ee X_{\text{even}}| \geq \frac{\delta}{100} \cdot \Ee X_{\text{even}}\right) \leq 
e^{-\Omega( \Ee X_{\text{even}} )} =  e^{-\Omega( np )} = e^{-n^{\Omega(1)}}, $$

\noindent
and analogously $\bP( |X_{\text{odd}}-\Ee X_{\text{odd}}| \geq \frac{\delta}{100} \cdot 
\Ee X_{\text{odd}} ) = e^{-n^{\Omega(1)}}$. It follows that 

\begin{equation}\label{eq:Sdashdev} 
\bP\left( \left| |S'| - \Ee|S'|\right| \geq \frac{\delta}{100} \Ee |S'| \right) \leq e^{-n^{\Omega(1)}}. 
\end{equation}

We see that 

\begin{equation}\label{eq:apa} 
\begin{array}{rcl} 
\bP\left( \left|S -\frac{|I_1| \cdot |I_2|}{n}\right| > \delta \frac{|I_1| \cdot |I_2|}{n} \right)
& \leq &
\bP\left( |S -\Ee S'| > \frac{\delta}{100} \Ee |S'| \right) \\[2pt]
& = &
\bP\left( |S' -\Ee S'| > \frac{\delta}{100} \Ee |S'| \Big| E \right) \\[2pt]
& \leq & 
\bP\left( |S' -\Ee S'| > \frac{\delta}{100} \Ee |S'| \right) / \bP(E)  \\[2pt]
& \leq & (n+1)^2 \cdot e^{-n^{\Omega(1)}} \\[4pt]
& = & e^{-n^{\Omega(1)}}. 
\end{array} \end{equation}

\noindent
(Here the first inequality holds for $n$ sufficiently large and uses~\eqref{eq:SaccIneq}. 
In the second line we use the earlier observation that $(S'|E) \isd S$ and in the fifth line we 
use~\eqref{eq:Eineq} and~\eqref{eq:Sdashdev}.)

It is easily seen that the bound~\eqref{eq:apa} holds uniformly over the choice of 
all pairs of disjoint intervals $I_1, I_2 \subseteq [n]$ of sizes $|I_1|,|I_2|\geq n^{1/2+\eps}$.
(Looking more carefully at the Chernoff bound as stated in Theorem~\ref{thm:chernoff}, the error bound 
$e^{-n^{\Omega(1)}}$ is 
bounded above by an expression of the form $\exp[ - c \cdot n^{2\eps} ]$, where
$c=c(\delta)$ depends only on $\delta$.)
Since there are $O( n^4 )$ choices of such pairs of intervals, the lemma follows by the union bound.
\end{proof}
    
\begin{lem}\label{lem:exists_J}
For every fixed $\eps>0$ and $k\geq 3$, with probability $1-e^{-n^{\Omega(1)}}$, for every sequence $\Ical = (I_1,\dots, I_N)$ of at most $10 \log\log n$-many disjoint intervals satisfying 
$$ n^{\frac12+\eps} \leq |I_1|,\dots,|I_N| \leq n^{1-\eps}, $$

\noindent
and every directed graph $G$ with $V(G) = \{ I_1,\ldots, I_N \}$, there exists an interval $J \subseteq [n]$ such that 

$$ H( \Ical ; \Ical_k(J) ) = G. $$

\end{lem}

\begin{proof}
    There are no more than $10 \log\log n \cdot n^{20 \log\log n} = e^{ O\left( \log n \cdot \log\log n\right) }$ many ways to pick $I_1, \dots, I_N$, and at most $2^{ O((\log\log n)^2) }$ corresponding directed graphs $G$ for each such a sequence. We will show that for each particular choice of $I_1,\dots, I_N$ and $G$ with $v(G)=N$, the probability that no $J$ exists as desired is at most $e^{-n^{\Omega(1)}}$. This will prove the lemma.

    Let us thus fix $I_1,\dots, I_N$, meeting the conditions of the lemma but otherwise arbitrary, and an arbitrary directed graph $G$ on $N$ vertices. Note that if $G$ is the ``empty directed graph'' ($e(G) = 0$) then we can simply take $J=\{1\}$. Thus, we can and do assume from now on that $G$ has at least one arc.

    Inside the set $[n] \setminus (I_1\cup\dots\cup I_N)$ we can find $M = n^{\Omega(1)}$ disjoint intervals
    $J_1, \dots, J_M$ each of size $s := \lceil n^{1-1/k} \cdot \left(\log\log n\right)^{3/k}\rceil$.

    For $\tau : I_1 \cup \dots \cup I_N \to [n]$ an injection and $L_1,\dots,L_M \geq 0$ and
    $A_{11},\dots,A_{1L_1} \subseteq J_1$, $\dots$, $A_{M1}, \dots, A_{ML_M} \subseteq J_M$ disjoint intervals, and $B_{11}, \dots, B_{1L_1}, \dots, B_{ML_M} \subseteq [n] \setminus \tau[ I_1\cup\dots\cup I_N ]$ disjoint sets with $|A_{ij}| = |B_{ij}|$ for all $i,j$, we define the event

    $$ F_{\tau,\underline{A},\underline{B}}
    := \left\{ \begin{array}{l} 
    \Ical_k(J_\ell) = (A_{\ell 1},\dots,A_{\ell L_\ell}) \text{ for all $\ell=1,\dots,M$, } \\
    \Pi_n[A_{\ell \ell'}] = B_{\ell \ell'} \text{ for all $\ell=1,\dots,M$ and $\ell'=1,\dots,L_{\ell}$, } \\
    \Pi_n(x) = \tau(x) \text{ for all $x \in I_1\cup\dots\cup I_n$.}
               \end{array} \right\}
    $$

    We observe that, conditional on $F_{\tau,\underline{A},\underline{B}}$, every bijection $A_{ij} \to B_{ij}$ is equally likely. That is, for every choice of bijections $\tau_{11} : A_{11} \to B_{11}, \dots, \tau_{ML_M} : A_{ML_M} \to B_{ML_M}$ we have

    $$ \bP\left( \text{$\Pi_n\upharpoonright A_{ij} \equiv \tau_{ij}$ for all $i,j$} {\big|} F_{\tau,\underline{A},\underline{B}} \right)
    = \frac{1}{|A_{11}|!\cdot\dots\cdot|A_{ML_M}|!}. $$

    Let us say an index $1\leq \ell \leq M$ is {\em good} if $(1-\eps)(\log\log n)^3 \leq L_\ell \leq (1+\eps)(\log\log n)^3$ and  
    $|S(I_i, J')| > 0$ for all $1\leq i\leq N$ and $J' \in \Ical_k(J_\ell)$. (So in particular $H( \Ical, \Ical_k(J_\ell) )$ is defined when $\ell$ is good.) Note also that whether $\ell$ is good or not is determined completely by $F_{\tau,\underline{A},\underline{B}}$, because the sets $S(I_i, A_{\ell j})$ depend only on $B_{\ell j}$, not on the choice of the specific bijection $A_{\ell j} \to B_{\ell j}$.

    Note that if $\ell$ is good and we randomly choose bijections $\tau_{i\ell} : A_{i\ell} \to B_{i\ell}$ then the points $y(I_1,A_{i\ell}), \dots, y(I_N,A_{i\ell})$ are chosen uniformly without replacement from $A_{i\ell}$. In particular, defining $P \be \{(I_i,I_j) \mid  i\neq j\}$, for each ordered pair $uv \in P$ the probability that $e( \Ical, A_{i \ell } ) = uv$ is $\frac{1}{|P|} = \frac{1}{N(N-1)}$.

    We fix a sequence $e_1,e_2,\dots,e_{\lceil(1+\eps)(\log\log n)^3\rceil} \in  P$ with the property that 

    $$ E(G) = \{ e_1,\dots,e_{\lceil(1-\eps)(\log\log n)^3\rceil}\} = \{e_1,\dots,e_{\lceil(1+\eps)(\log\log n)^3\rceil}\}. $$

    \noindent
    (Since $(1-\eps)(\log\log n)^3 > N^2$, there are necessarily some duplicate arcs. 
    This will not be a problem.) 
    We note that $H( \Ical, (A_{\ell 1}, \dots, A_{\ell L_{\ell}} ) ) = G$ is implied by the 
    event $\{ e( \Ical, A_{\ell i} ) = e_i \text{ for all $1\leq i \leq L_i$} \}$.

    For notational convenience, let us write 

    $$ E := \{ H(\Ical,\Ical_k(J_\ell)) = G \text{ for some $1\leq \ell \leq M$ } \}. $$

    It follows from the previous observations that if $Z$ denotes the number of good indices, then the value of $Z$ is completely determined by the event $F_{\tau,\underline{A},\underline{B}}$, and 

    $$ \begin{array}{rcl} 
        \bP\left( E | F_{\tau,\underline{A},\underline{B}} \right) 
        & \geq & \displaystyle 
        1 - \left( 1 - \left(\frac{1}{N^2}\right)^{(1+\eps)(\log\log n)^3} \right)^Z \\[2ex]
        & \geq & \displaystyle 
        1 - \exp\left[ - Z \cdot \left(\frac{1}{N^2}\right)^{(1+\eps)(\log\log n)^3} \right] \\[3ex]
        & = & \displaystyle 
        1 - \exp\left[ - Z\cdot n^{o(1)} \right]. 
    \end{array}$$

        \noindent
        By Corollary~\ref{cor:at_least_half} and Lemma~\ref{lem:S} we have that 

        $$ \bP( Z < M/2 ) = e^{-n^{\Omega(1)}}. $$
        It follows that 

    $$ \begin{array}{rcl} \bP( E )
        & \geq & \displaystyle  
        \sum_{\tau,\underline{A},\underline{B}} 
        \bP\left(E | F_{\tau,\underline{A},\underline{B}} \right)  \cdot \bP( F_{\tau,\underline{A},\underline{B}} ) \\
        & \geq & \displaystyle 
        1 - \exp[ - (M/2)\cdot n^{o(1)} ] - \bP( Z < M/2 ) \\
        & = & 1 - e^{-n^{\Omega(1)}}.
    \end{array} $$
    This proves the statement of the lemma.
\end{proof}

\begin{proof}[Proof of Lemma~\ref{lem:graph_exists}]
    Let $I_1,\ldots, I_N$ be $\ceil{n^{1/2k}}$ intervals of length $\lfloor  n^{1-1/k} \cdot \left ( \log \log n \right)^{3/k} \rfloor$. By Corollary~\ref{cor:at_least_half}, with probability $1-e^{-\Omega(n^{1/2k})}$ at least $1$ of them satisfies Conditions~\ref{itm:vtexlem1}-\ref{itm:vtexlem4} of Lemma~\ref{lem:vtexlem}, by relabeling if necessary, we can assume that $I_1$ satisfies these conditions. We have $\cI_k(I_1) \gg \log \log n$ by Condition~\ref{itm:vtexlem1}. By Lemma \ref{lem:shrink_CkI} we may 
    shrink $I_1$ by iteratively removing an element until $\log\log n < |\cI_k(I_1)| < 2 \log \log n$, ensuring 
    that $I_1$ still satisfies Conditions~\ref{itm:vtexlem2}-\ref{itm:vtexlem4} of 
    Lemma~\ref{lem:vtexlem}. 
    In particular $|I'| > n^{.51}$ for all $I' \in \cI_k(I_1)$, and by $|I_1|\leq n^{1-1/2k}$ we also have $I' \leq n^{1-1/2k}$ for all $I  ' \in \cI_k(I_1)$. 

    By Lemma~\ref{lem:exists_J}, for all directed graphs $G$ with $V(G)=\cI_k(I_1)$ there is an interval $J$ satisfying $H( \cI_k(I_1),\cI_k(J)) = G$, with probability $1-e^{-n^{\Omega(1)}}$.
\end{proof}

\begin{proof}[Proof of Lemma~\ref{lem:matching}]
    Let $I_1,I_2$ be any two disjoint intervals such that $0 < |\cI_k(I_1)| <   \log \log n < |\cI_k(I_2)| < 2\log\log n$, and such that $w_{k+1}(I_1)=w_{k+1}(I_2) = 0$ and $w_{k-1}(I') >0 $ for all $I\in \cI_k(I_1)\cup \cI_k(I_2)$. Each of these $I' \in \cI_k(I_1)\cup \cI_k(I_2)$ has length at most $ \max\{ |I_1|,|I_2|\}$. By Lemma~\ref{lem:positive_w} with $\epsilon = \frac{1}{2(k+1)}$
    \begin{equation}
        \bP \left( w_{k+1} (I') > 0 \text{ for all intervals }I' \text{ s.t. }|I'| > n^{1- \frac{1}{2(k+1)}} \right) \geq 1 - e^{-\Omega(n^{1/2})}.
    \end{equation}
    So with at least this probability, all $I_1$ and $I_2$ as described satisfy $|I_1|,|I_2| \leq  n^{1 - \frac{1}{2(k+1)}}$. By $w_{k-1}(I') >0 $ for all $I'\in \cI_k(I_1)\cup \cI_k(I_2)$ and Lemma~\ref{lem:zero_w} we have $|I'| \geq n^{1-1000/{(k-1)}}$ for all such $I_1,I_2$ and $I'\in \cI_k(I_1)\cup \cI_k(I_2)$ with probability $1- O(n^{-998})$. So for sufficiently large $k$ we have 
    \begin{equation}
         n^{\frac12 + \frac{1}{2(k+1)} }<n^{1-1000/(k-1)} \leq |I'| < n^{1 - \frac{1}{2(k+1)}}.
    \end{equation}
    The result then follows by $|\cI_k(I_1)\cup \cI_k(I_2)| \leq 3 \log\log n$ and Lemma~\ref{lem:exists_J}.
\end{proof}

\section{Proof of Theorem~\ref{thm:toto}~Part~\ref{part:toto_non_convergence}\label{sec:totoallq}}

Here we extend the proof of the case when $q=1$ to $|1-q|  < 1/\log^* n$. 
We do so step by step, first extending to $|1-q| \leq n^{-4}$ then extending to
$|1-q|=O(1/n)$ and finally to the entire range.

\subsection{Extending the $q=1$ case to $q=1\pm n^{-4}$}

Here we derive the following from Proposition~\ref{prop:oscilate}.

\begin{prop}\label{prop:n4}
There exists a $\phi\in {\TOTO}$ such that for all  sequences
$q=q(n)$ satisfying $1-n^{-4} < q < 1+n^{-4}$ and
all sufficiently large $n$ satisfying $ W^{(2)}(\log^{**}\log^{**}n-1)<\log\log n$ we have
\begin{equation}
    \bP( \Pi_n \models \phi) = 
            \begin{cases}
            1 - O( n^{-2} ) & \text{if } \log^{**}\log^{**}n \text{ is even}, \\
            O( n^{-2} ) &  \text{if } \log^{**}\log^{**}n \text{ is odd}.
        \end{cases},  
\end{equation}
where $O(n^{-2})$ terms can be taken uniform over all sequences $q(n)$ considered. 
\end{prop}

\noindent
(Again Lemma~\ref{lem:composed_large_functions} ensures that $W^{(2)}(\log^{**}\log^{**}n-1)<\log\log n$ is 
for instance satisfied for all 
$n$ of the form $n=W^{(2)}(m)$ with $m>3$.)

Proposition~\ref{prop:n4} will follow by showing that for $1-n^{-4} < q < 1$, a $\Mallows(n,q)$ distribution 
is in some sense very close to the uniform distribution over $S_n$.

\begin{lem}\label{lem:trunc_unif_close}
     Let $q=q(n)$ satisfy $1-n^{-4} < q < 1$. 
    For $1\leq i \leq n $ let $Z_i \sim \TGeo(n-i+1, 1-q)$ and $X_i\sim \Unif([n-i+1])$. Then 
    \begin{equation}
        \dtv(Z_i,X_i) \leq 4 \cdot n^{-3},
    \end{equation}
    for $n$ sufficiently large and all $1\leq i \leq n$. 
\end{lem}

\begin{proof}
    Let us write $m \be n-i+1$ and $p_j := \Pee( Z_i=j )$. By \eqref{eq:dtvalt} we may write
    \begin{align}\label{eq:rgkjerng}
        \dtv(Z_i,X_i) = 
        \sum_{ j : p_j > 1/m } \left(p_j - \frac{1}{m}\right).
    \end{align}
    Now note that 
    
    $$ \frac{\max_j p_j}{\min_j p_j} = q^{-(m-1)} < \left(\frac{1}{1-n^{-4}}\right)^{m-1}
    < \left(\frac{1}{1-n^{-4}}\right)^n \leq \left(1+2n^{-4}\right)^n
    \leq e^{2n^{-3}} \leq 1+ 4n^{-3}, $$
    
    \noindent
    where the third inequality holds for $n$ sufficiently large; the fourth inequality uses 
    $1+x \leq e^x$, and; the last inequality holds for $n$ sufficiently large, and uses the Taylor expansion for $e^x$.
    It follows that 
    
    $$ p_j \leq (1+4n^{-3}) \cdot \frac{1}{m}, $$
    
    \noindent
    for all $j$. Filling this in to~\eqref{eq:rgkjerng} gives $\dtv(Z_i,X_i) \leq 4n^{-3}$ as desired. 
\end{proof}

\begin{lem}\label{lem:n4}
    Let $q=q(n)$ satisfy $|1 - q | < n^{-4}$, and let $\Pi_n\sim  \Mallows(n,q)$ and $ \Pi^*_n \sim \Mallows(n,1) $. 
    There is a coupling of $\Pi_n$ and $\Pi^*_n$ such that 
    $\bP \left( \Pi_n \neq \Pi_n^* \right) \leq 4n^{-2}$, for all sufficiently large $n$.
\end{lem}

\begin{proof}
    If $q=1$, then the result is trivial. 
    We now assume that $1-n^{-4} < q < 1$. 

    Let $\Pi_n$ be generated by a sequence $Z_1,\ldots, Z_n$ of independent random variables 
    with $Z_i \sim \TGeo(n-i+1, 1-q)$ for $i\in [n]$ (see Section \ref{subsec:constr_trunc}). 
    Let $\Pi_n^*$ be generated in the same manner by a sequence $X_1,\ldots, X_n$ of independent 
    random variables where $X_i \sim \Unif ([n-i+1])$ for $i\in [n]$. 
    By Lemma~\ref{lem:trunc_unif_close} and~\ref{lem:coupling} we may couple $Z_i$ and $X_i$ for each $i$ such that 
    \begin{equation}
        \bP \left( Z_i \neq X_i \right) \leq 4n^{-3},\quad\text{as }n\to\infty. 
    \end{equation}
    This gives a coupling of $\Pi_n$ and $\Pi^*_n$. Now, $\Pi_n = \Pi^*_n$ if and only if $Z_i = X_i$ for all $i\in [n]$, giving
    \begin{equation}
        \bP \left( \Pi_n \neq \Pi^*_n \right) \leq 4n^{-2}. 
    \end{equation}
    
   Finally we consider the case when $1<q<1+n^{-4}$.
   Recall that if $r_n \in S_n$ is the ``reverse'' map given by $r_n(i) := n+1-i$, then 
   $r_n\circ \Pi_n \isd \Mallows(n,1/q)$.
   Note that $1 \geq 1/q  \geq 1 - n^{-4}$, so by the previous case there
   exists a coupling of $\Pi_n$ and $\Pi_n^{**} \sim \Mallows(n,1)$ such that 
   
   $$ \Pee( r_n\circ\Pi_n \neq \Pi_n^{**} ) \leq 4 n^{-2}. $$
   
   Setting $\Pi_n^* = r_n \circ \Pi_n^{**}$, we have 
   
   $$ \Pee( \Pi_n \neq \Pi_n^* )
   = \Pee( r_n\circ\Pi_n \neq r_n\circ \Pi_n^{*} ) 
   = \Pee( r_n\circ\Pi_n \neq \Pi_n^{**} ) \leq 4n^{-2}, $$
   
   \noindent
   where we use that $r_n\circ r_n$ is the identity.
   Finally we note that if $\Pi_n^{**}$ is uniformly distributed on $S_n$ then so is 
   $\Pi_n^* = r_n\circ\Pi_n^{**}$.
\end{proof}

\begin{proof}[Proof of Proposition~\ref{prop:n4}]
    Let $q$ satisfy $|1 - q| \leq n^{-4}$ and  let $\Pi_n\sim \Mallows(n,q)$ and $\Pi_n^* \sim \Mallows(n,1)$ be coupled 
    such that
    \begin{equation}\label{eq:rgkhebgrheg}
        \bP \left( \Pi_n \neq \Pi_n^* \right) = O(n^{-2}).
    \end{equation}
    This coupling exists by Lemma~\ref{lem:n4}. Let $\phi$ be the $\TOTO$ sentence from 
    Proposition~\ref{prop:oscilate}. Then $\Pi_n$ and $\Pi_n^*$ agree on $\phi$ with probability 
    $1 - O(n^{-2})$, giving the desired result.
\end{proof}

\subsection{Extending the $q=1\pm n^{-4}$ case to $q=1\pm O(1/n)$\label{sec:extend_to_1n}}

In this subsection we prove the following intermediate result.

\begin{prop}\label{prop:q1/n}
    There is a sentence $\psi \in \TOTO$ such that the following holds.
    Suppose that $q=q(n)$ satisfies $1-c/n<q<1+c/n$ for some constant $c>0$. 
    For all $n$ satisfying $W^{(2)}(\log^{**}\log^{**}n-1)<\log\log\log n$ we have
    \begin{equation}
         \bP( \Pi_n \models \psi) = 
        \begin{cases}
           1 - O(n^{-.001})& \text{if } \log^{**}\log^{**}n \text{ is even}, \\
           O(n^{-.001}) &  \text{if } \log^{**}\log^{**}n \text{ is odd},
        \end{cases}
    \end{equation}
    where the $O(n^{-.001})$ terms can be taken uniform over all 
    sequences considered (but may depend on $c$).
\end{prop}

\noindent
(Again, the condition $W^{(2)}(\log^{**}\log^{**}n-1)<\log\log\log n$ is satisfied if 
$n$ is of the form $n = W^{(2)}(m)$ with $m>3$, by 
Lemma~\ref{lem:composed_large_functions}.)

For $\pi \in S_n$ a permutation, we define

$$ \Psi(\pi) := \min\{ 1\leq j \leq n: \text{ there exist $1\leq i_1, i_2 \leq j$ such that 
$\pi(i_2)=\pi(i_1)+1$ }\}. $$

Letting $\Pi_n \sim \Mallows(n,q)$ as usual, we define:

$$ X_n := \Psi(\Pi_n). $$

\noindent
Before studying the behaviour of $X_n$, we first present a preparatory lemma.

\begin{lem}\label{lem:P_trunc_lt_j}
    Suppose that $q=q(n)$ satisfies $1-c/n  < q < 1$. There exist $c_U=c_U(c)>0$ and $c_L=c_L(c) >0$ such that 
    for $n$ sufficiently large and all $1\leq m\leq n$ and $A\subseteq [m]$ we have  
    \begin{equation}
         c_L \frac{|A|}{m} \leq  \bP \left( \TGeo (m, 1-q) \in A\right) \leq c_U \frac{|A|}{m}.
    \end{equation}
\end{lem}

\begin{proof}
We fix $n$ and some $m\in[n]$.
For convenience we again set $p_j := \Pee( Z_i = j )$.
We again have

$$ \frac{\max_j p_i}{\min_j p_j} = q^{-(m-1)} \leq \left(\frac{1}{1-cn^{-1}}\right)^n \leq 
(1+2cn^{-1})^n \leq e^{2c}, $$

\noindent
where the penultimate inequality holds for $n$ sufficiently large. Similarly

$$ \frac{\min_j p_j}{\max_j p_j} = q^{m-1} \geq (1-cn^{-1})^n \geq \frac12 e^{-c}, $$

\noindent
where the last inequality holds for $n$ sufficiently large.
We can conclude that 

$$ c_L \cdot \frac{1}{m} \leq p_j \leq c_U \cdot \frac{1}{m}, $$

\noindent
for all $j=1,\dots,m$, with $c_L := \frac12 e^{-c}, c_U := e^{2c}$.
The result follows.
\end{proof}

\begin{lem}\label{lem:bound_J1}
    Let $q=q(n)$ satisfy $1-cn^{-1} < q < 1+cn^{-1}$ for some $c>0$.
    For every $\epsilon > 0$ we have
    \begin{equation}\label{eq:asdskajdnsakj}
        \bP \left(  n^{1/2-\eps} \leq  X_n \leq n^{1/2+\eps}  \right) > 1 - n^{-\eps},
    \end{equation}
    provided that $n\geq n_0$ for some $n_0=n_0(c,\eps)$ depending on $\eps$ and $c$. 
\end{lem}

\begin{proof}
    We first assume $q\leq 1$.
    For $i\geq 2$, let $E_i$ denote the event that 
    $\Pi_n(i) \not\in\{ \Pi_n(1)-1,\Pi_n(1)+1,\dots,\Pi_n(i-1)-1,\Pi_n(i-1)+1\} =: A_i$.
    We point out that if $E_1,\dots,E_{i-1}$ hold, then $i-1\leq |A_i| \leq 2i$.
    Recalling the sampling procedure~\ref{eq:MallowsIter} for the Mallows model described in 
    Section~\ref{subsec:constr_trunc} and using the
    previous lemma, we find that for $i\geq 2$:
    
    $$ 1-c_U\frac{2i}{n-i+1} \leq \Pee( E_i | E_2 \cap \dots \cap E_{i-1} ) 
    \leq 1 - c_L \frac{i-1}{n-i+1}. $$
    
    \noindent
    Writing $x := \floor{n^{1/2+\eps}}$, we find that 
    
    $$ \begin{array}{rcl} 
    \Pee( X > x ) 
    & = & \Pee( E_2 \cap \dots \cap E_x ) \\
    & = & \Pee( E_2 ) \cdot \Pee( E_3 | E_2) \cdot \dots \cdot \Pee( E_x | E_2 \cap\dots\cap E_{x-1} ) \\
    & \leq & \prod_{i=2}^x \left(1 - c_L \frac{i-1}{n-i+1}\right) \\
    & \leq & \exp\left[ - c_L \sum_{i=2}^x \frac{i-1}{n-i+1} \right] \\
    & \leq & \exp\left[ - c_L \cdot (2/n) \cdot {x\choose 2} \right] 
    = \exp\left[ - \Omega( n^{2\eps} ) \right],  
       \end{array}    $$
 
 \noindent
 where the third line holds assuming $n$ is sufficiently large.
 
 Writing $y := \ceil{n^{1/2-\eps}}$, we find that 
 
 $$ \begin{array}{rcl} 
    \Pee( X < y ) 
    & = & \Pee( \overline{E_2} ) + \Pee(\overline{E_3}|E_2) + \dots + \Pee( \overline{E_y} | E_2\cap\dots\cap E_{y-1} ) \\
    & \leq & \sum_{i=2}^y c_U \frac{2i}{n-i+1} \leq \frac{c_U}{n} \sum_{i=1}^y i \\
    & =  & \frac{c_U}{n} \cdot {y+1\choose 2} = O( n^{-2\eps} ).  
       \end{array}    $$
       
Combining these bounds completes the proof for the case when $q\leq 1$.
For the case when $1 \leq q \leq 1+c/n$, we recall that 
$r_n \circ \Pi_n \isd \Mallows(n,1/q)$. 
We also note that $1 \geq 1/q \geq 1 - c n^{-1}$ and that 
$\Psi(\pi) = \Psi(r_n\circ\pi)$ for all $\pi \in S_n$, by definition of $\Psi$.
So the result for $1 \leq q \leq 1+c/n$ follows immediately from the result for $1-cn^{-1}\leq q \leq 1$.
\end{proof}

We next define

$$ Y_n := \Psi\left( \rk( \Pi_n(1),\dots,\Pi_n(X_n) ) \right), $$ 

$$ Z_n := \Psi\left( \rk( \Pi_n(1),\dots,\Pi_n(Y_n) ) \right). $$ 

The next observation will allow us to make use of Lemma~\ref{lem:bound_J1} also for $Y_n$ and $Z_n$.

\begin{lem}\label{lem:J2_monotone}
    For every $n\geq 3$, $2 \leq k\leq n-1$ and every $\pi \in S_n$ we have
    \begin{equation}\label{eq:J2_ineq}
        \Psi(\pi_k) \leq \Psi(\pi_{k+1}),
    \end{equation}
    where $\pi_j :=\rk(\pi(1),\ldots, \pi(j))$. 
\end{lem}

\begin{proof}
    This follows readily from the observation that, for all distinct $i_1,i_2 \leq k$,
    $\pi_{k+1}(i_2)=\pi_{k+1}(i_1)+1$ implies that 
    $\pi_k(i_2)=\pi_k(i_1)+1$.
\end{proof}

\begin{lem}\label{lem:bound_K1}
    Suppose that $q = q(n)$ satisfies $1-c/n < q < 1+c/n$ for some constant $c>0$. 
    Then
    \begin{equation}\label{eq:asdskajdnsakj}
        \bP \left( n^{.1} \leq Z_n \leq n^{.2}  \right) > 1 - n^{-.001},
    \end{equation}
    provided that $n\geq n_0$ for some $n_0=n_0(c)$ depending on  $c$. 
\end{lem}

\begin{proof}
    We fix $\delta>0$, to be determined during the course of the proof.
    For notational convenience, let us write 
    $x := \ceil{n^{1/2-\delta}}, y := \ceil{n^{1/4-\delta}}, z := \ceil{n^{1/8-\delta}}$.
    Recall that by Lemma~\ref{lem:prefix_mallows}, for each fixed $j$ we have 
    $\Pi_j^* := \rk(\Pi_n(1),\dots,\Pi_n(j)) \isd \Mallows(j,q)$.
    Using also the previous two lemmas and the definitions of $X_n,Y_n,Z_n$, we find:
    
    $$ \begin{array}{rcl} \Pee( Z < z ) 
    & \leq &  \Pee( X_n < x )+\Pee( \Psi(\Pi_x^*) < y ) + \Pee( \Psi(\Pi_y^*) < z ) \\
    & = &  \Pee( X_n < x )+\Pee( \Psi(\Pi_x) < y ) + \Pee( \Psi(\Pi_y) < z ) \\
    & = & \Pee( X_n < x ) + \Pee( X_x < y ) + \Pee( X_y < z ) \\
    & \leq & n^{-\delta} + x^{-\delta} + y^{-\delta} \\
    & \leq & n^{-\delta/8},       
    \end{array} $$

    \noindent
    where the fourth line holds for $n$ sufficiently large and we 
    use 
    
    $$\begin{array}{c} 
    x^{1/2-\delta} = (1+o(1)) n^{(1/2-\delta)^2} >  y = (1+o(1))n^{1/4-\delta},  \\
    y^{1/2-\delta} = (1+o(1)) n^{(1/2-\delta)^3} > z = (1+o(1))n^{1/8-\delta}, 
    \end{array} $$
    
    \noindent
    and the last line holds for $n$ sufficiently large.

    Analogously we have, 
    writing $x' := \floor{n^{1/2+\delta}}, y' := \floor{n^{1/4+2\delta}}, z := \floor{n^{1/8+4\delta}}$:
    
    $$ \begin{array}{rcl} \Pee( Z > z' ) 
    & \leq & \Pee( X_n > x' ) + \Pee( X_{x'} > y' ) + \Pee( X_{y'} > z' ) \\
    & \leq & n^{-\delta} + {x'}^{-\delta} + {y'}^{-\delta} \\
    & \leq & n^{-\delta/8},       
    \end{array} $$
    
    \noindent
    for $n$ sufficiently large.
    
    Setting $\delta = .025$ and combining the bounds gives the result, with some room to spare.
\end{proof}

We are almost ready to prove Proposition~\ref{prop:q1/n}. 
But, we first point out that it is possible to express that $i \leq Z_n$ in $\TOTO$.

\begin{lem}\label{lem:varphiZ}
There is a formula $\varphi_Z(x)$ in $\TOTO$ with one free variable such that 
for every $n\in\eN$, every permutation $\pi \in S_n$ and every $1\leq i \leq n$:

$$ (\pi,i) \models \varphi_Z(x) \text{ if and only if } i \leq Z_n(\pi). $$

\end{lem}

\begin{proof} We first show there is a formula $\varphi_X(x)$ that expresses that $i \leq \Psi(\pi)$
if $\pi \in S_n, i \in [n]$ for some $n\in\eN$.
This formula is given by 

$$ \forall y,z : \left((y <_1 x ) \wedge (z <_1 x )\right) 
\implies \left(\neg\succ_2(y,z) \wedge \neg\succ_2(z,y) \right). $$
 
\noindent
(``For every pair $y,z < x$ we have $|\pi(y)-\pi(z)|\neq 1$''.)
Now let $\varphi_Y := \varphi_X^{\varphi_X}$, the relativization of
$\varphi_X$ to $\varphi_X$, as provided by Lemma~\ref{lem:relativization}.
Then, invoking that lemma, $(\pi,i) \models \varphi_Y(x)$ if and only if $i \leq \Psi(\pi)$ and
$(\sigma, i) \models \varphi_X(x)$ where $\sigma := \rk(\pi(1),\dots,\pi(\Psi(\pi)) )$.
In other words, $(\pi,i) \models \varphi_Y(x)$ if and only if $i \leq Y_n(\pi)$.

Similarly, if we now set $\varphi_Z := \varphi_X^{\varphi_Y}$ then
$(\pi,i) \models \varphi_Z(x)$ if and only if $i \leq Z_n(\pi)$.
\end{proof}

\begin{proofof}{Proposition~\ref{prop:q1/n}}
    Let $\phi$ be the sentence provided by Proposition~\ref{prop:n4} and let $\varphi_Z$ be the 
    formula provided by Lemma~\ref{lem:varphiZ}. 
    Then the relativization $\psi := \varphi^{\varphi_Z}$ (as provided by Lemma~\ref{lem:relativization})
    expresses that $\phi$ is satisfied by $\rk(\Pi_n(1),\dots,\Pi_n(Z_n))$.
    
    Now let $n$ be such that $W^{(2)}(\log^{**}\log^{**}n-1)<\log\log\log n$. 
    First assume that $\log^{**}\log^{**}n$ is even.
    We have, writing $\Pi_j^* := \rk(\Pi_n(1),\dots,\Pi_n(j))$: 
    
    $$ \begin{array}{rcl} 
    \Pee( \Pi_n \models \psi ) 
    & \geq & \Pee( \Pi_m^* \models \varphi \text{ for all 
    $n^{.1}\leq m\leq n^{.2}$}) - \Pee( Z_n \not\in [n^{.1},n^{.2}] ) \\[2ex]
    & \geq & \displaystyle 1 - \sum_{n^{.1}\leq m \leq n^{.2}} \Pee( \Pi_m^*\not\models \varphi ) 
    - \Pee( Z_n \not\in [n^{.1},n^{.2}] ) \\[2ex]
    & = & \displaystyle 1 - \sum_{n^{.1}\leq m \leq n^{.2}} O( m^{-100} ) - O( n^{-.001} ) \\[3ex]
    & = & 1 - O( n^{-.001} ), 
    \end{array} $$

    \noindent 
    where we use Lemma~\ref{lem:prefix_mallows} and Proposition~\ref{prop:n4} together with the fact 
    that, for $n$ sufficiently large and 
    all $n^{.1}\leq m\leq n^{.2}$,
    we have 
    
    $$\log\log m > \log\log\log n > W^{(2)}(\log^{**}\log^{**}n-1) \geq W^{(2)}(\log^{**}\log^{**}m-1). $$
    
    Now assume that $\log^{**}\log^{**}n$ is odd. In this case we have
    
    $$ \begin{array}{rcl} 
    \Pee( \Pi_n \models \psi ) 
    & \leq & \displaystyle \sum_{n^{.1}\leq m\leq n^{.2}} \Pee( \Pi_m^* \models \varphi ) + 
    \Pee( Z_n \not\in [n^{.1},n^{.2}] ) \\[3ex]
    & \leq & \displaystyle \sum_{n^{.1}\leq m \leq n^{.2}} O( m^{-100} ) + O( n^{-.001} ) \\[3ex]
    & = & O( n^{-.001} ), 
    \end{array} $$
    
    \noindent
    with the same justifications as previously.
\end{proofof}

\subsection{Lifting to $q =  1 \pm 1/\log^*n$}

In this section we finalize the proof of Theorem~\ref{thm:toto}~Part~\ref{part:toto_non_convergence}. 
Specifically, we will prove the following more detailed version of Theorem~\ref{thm:toto}~Part~\ref{part:toto_non_convergence}.

\begin{prop}\label{prop:final}
There exists a sentence $\xi \in\TOTO$ such that, for all sequences $q=q(n)$ satisfying 
$1-1/\log^* n < q < 1+1/\log^* n$ and all $n$ satisfying $W^{(2)}(\log^{**}\log^{**}n-1) < \log^*\log^*\log^* n$
we have

$$ \Pee( \Pi_n \models \xi ) = \begin{cases}
                                   1 - o(1) & \text{ if $\log^{**}\log^{**}n$ is even, } \\
                                   o(1) & \text{ if $\log^{**}\log^{**}n$ is odd, }
                                  \end{cases}, $$
 \noindent 
 where the error terms $o(1)$ can be taken uniform over all sequences $q(n)$ considered.
\end{prop}

\noindent
(Again, we note the condition $W^{(2)}(\log^{**}\log^{**}n-1) < \log^*\log^*\log^* n$ is for instance 
satisfied for all $n$ of the form 
$n=W^{(2)}(m)$ with $m>3$, by Lemma~\ref{lem:composed_large_functions}.)

The quantity $\Pi_n^{-1}(1)$ will play an important role in the definition of the $\TOTO$ property $\xi$.
We start with some preparatory lemmas on its behaviour.

\begin{lem}\label{lem:large_alpha}
For every $c, \eps > 0$ there exist constants $c_1=c_1(c,\eps),c_2=c_2(c,\eps)>0$ and 
$n_0=n_0(c,\eps)$ such that for $n\geq n_0$ and any sequence $q=q(n)$ satisfying $ 1 - 1/ \log^* n  < q  < 1 - c/n$, 
the random permutation $\Pi_n\sim \Mallows(n,q)$ satisfies
    \begin{align}
        \bP \left( \frac{c_1}{1-q} \leq  \Pi_n^{-1}(1) \leq \frac{c_2}{1-q}\right) \label{eq:rgkhherghr1} > 1-\eps.
    \end{align}
\end{lem}

\begin{proof}
Recall from Section~\ref{subsec:useful_lit} that if $\Pi_n \sim \Mallows(n,q)$ then also $\Pi_n^{-1} \isd \Pi_n$.
In particular $\Pi_n^{-1}(1)\isd \Pi_n(1) \isd \TGeo(n,1-q)$,
where the last distributional equality follows from the sampling algorithm presented in 
Section~\ref{subsec:constr_trunc}.
Writing $x_1 := \floor{c_1/(1-q)}$, we first note that 
provided the constant $c_1>0$ is chosen sufficiently small, we have
$x_1 < (c_1/c) \cdot n \leq n$. 
We can thus write:

$$ \Pee\left( \Pi_n^{-1}(1) \leq \frac{c_1}{1-q}\right)
= \sum_{i=1}^{x_1} \frac{q^{i-1}(1-q)}{1-q^n}
= \frac{1-q^{x_1}}{1-q^n} 
< \frac{1-q^{x_1}}{1-(1-c/n)^n}
< \frac{1-q^{x_1}}{1-e^{-c}}. $$

\noindent
Now we note that 

$$q^{x_1} = \left(1-(1-q)\right)^{x_1} \geq 1 - x_1\cdot(1-q)  \geq 1-c_1. $$

\noindent
It follows that 

$$ \Pee\left( \Pi_n^{-1}(1) \leq \frac{c_1}{1-q}\right) \leq \frac{c_1}{1-e^{-c}} < \eps/2, $$

\noindent 
provided we chose the constant $c_1>0$ sufficiently small.

Similarly, writing $x_2 := c_2/(1-q)$ we have 

$$ \Pee( \Pi_n^{-1}(1) > x_2 ) < \frac{1}{1-\left(1-c/n\right)^{n}} \cdot q^{x_2}
\leq \frac{1}{1-e^{-c}} \cdot \left(1-(1-q)\right)^{c_2/(1-q)} \leq \frac{e^{-c_2}}{1-e^{-c}} < \eps/2, $$

\noindent
where the last inequality holds
provided we chose $c_2$ sufficiently large. The result follows combining the two bounds. 
\end{proof}

\begin{lem}\label{lem:Pi1bdd}
 For every $c,\eps>0$ there exist $\delta=\delta(c,\eps),n_0=n_0(c,\eps)$ such that for all sequences $q=q(n)$ with 
 $1-c/n\leq q \leq 1+c/n$ and all $n\geq n_0$:
 
 $$ \Pee\left( \delta n \leq \Pi_n^{-1}(1) \leq (1-\delta) n \right) > 1-\eps. $$
 
\end{lem}

\begin{proof}
First note that in the special case when $q=1$, we have that $\Pi_n^{-1}(1)$ is 
chosen uniformly at random from $[n]$. So any $\delta < \eps/2$ would work in this case.
It therefore suffices to consider $q\neq 1$ so long as we ensure the value of $\delta$ we 
end up selecting in the end is $<\eps/2$. (Which is easily achieved by adjusting the choice of 
$\delta$ in the end.)

Next, we consider the case when $1-c/n\leq q < 1$.
Again we will use that $\Pi_n^{-1}(1) \sim \TGeo(n,1-q)$.
With $c_U=c_U(c)$ as provided by Lemma~\ref{lem:P_trunc_lt_j}, we can write:

$$ \Pee\left( \delta n \leq \Pi_n^{-1}(1) \leq (1-\delta) n \right) 
\geq 1 - c_U \cdot 2\delta > 1 - \eps, $$

\noindent
provided we chose $\delta$ small enough.

It remains to consider $1<q<1+c/n$. 
We recall that
$\Pi_n^{-1} \isd \Pi_n \isd \Mallows(n,q)$ and hence $r_n \circ \Pi_n^{-1} \isd \Mallows(n,1/q)$.
But then $1/(1+c/n) > 1-c/n$, together with the previous case gives

$$ \begin{array}{rcl} \Pee\left( \delta n \leq \Pi_n^{-1}(1) \leq (1-\delta) n \right) 
& = & \Pee\left( n+1-(1-\delta) n \leq (r_n \circ \Pi_n^{-1})(1) \leq n+1-\delta n \right) \\
& \geq & \Pee\left( 2\delta \leq (r_n\circ\Pi_n^{-1})(1) \leq (1-2\delta)n \right)  \\
& > & 1 - \eps, \end{array} $$

\noindent
where the first inequality holds for $n$ sufficiently large, and the last inequality
by the previous case and adjusting the value of $\delta$.
\end{proof}

For the remainder of the section, we let $\Sigma_n$ denote the 
permutation on $\{1,\dots,\Pi_n^{-1}(1)-1\}$ induced by $\Pi_n$. (If $\Pi_n(1) = 1$,
then $\Sigma_n$ is undefined.)
In other words, we set 

$$ \Sigma_n := \rk( \Pi_n(1),\dots,\Pi_n(\Pi_n^{-1}(1)-1) ). $$

 \begin{lem}\label{lem:sigmacond}
 For $0<q<1$, $n\geq 2$ and $2\leq k \leq n$ we have 
 
 $$ \left( \Sigma_n \Big| \Pi_n^{-1}(1) = k \right) \isd \Mallows(k-1,q). $$
 
 \end{lem}
 
 \begin{proof}
 Fix an arbitrary $0<q<1,n\geq 2$ and $2\leq k \leq n$.
 Let $Z_1,\dots,Z_n$ be the random variables used in the sampling algorithm from Section~\ref{subsec:constr_trunc}.
 In particular, they're independent with $Z_i \sim\TGeo(n+1-i,1-q)$.
 We first observe that
 
 \begin{equation}\label{eq:star} 
 \{ \Pi_n^{-1}(1)=k \} = \{ \Pi_n(k)=1 \} 
 = \{ Z_1 > 1, \dots, Z_{k-1} > 1, Z_k=1 \}, \end{equation}
 
 \noindent
 We set $W_i := Z_i - 1$ for $i<k$ and $W_i=Z_{i+1}$ for $k \leq i \leq n-1$.
 Next, we claim that conditional on the event $\Pi_n^{-1}(1)=k$ the random variables $W_1,\dots,W_{n-1}$ are independent
 with $W_i \sim \TGeo(n-i,1-q)$. That they are independent (conditional on $\Pi_n^{-1}(1)=k$) should be obvious
 from~\eqref{eq:star} and the independence of $Z_1,\dots,Z_n$.
 Their distribution is also obvious for $i\geq k$. For $i<k$ the claim follows from 
 
 $$ \begin{array}{rcl} \Pee( W_i = w |  \Pi_n^{-1}(1)=k ) & = & \Pee( Z_i = w+1 | Z_i > 1 ) \\[2ex]
 & = & \displaystyle \left(\frac{q^{w}(1-q)}{1-q^{n+1-i}}\right) / \left(\frac{q-q^{n+1-i}}{1-q^{n+1-i}}\right) \\[2ex]
 & = & \displaystyle \frac{q^{w-1}(1-q)}{1-q^{n-i}}. 
 \end{array} $$

 \noindent
Conditional on $\Pi_n^{-1}(1)=k$, we can define $\Sigma_n^* \in S_{n-1}$ by 
 
 $$ \Sigma_n^*(i) := \begin{cases} \Pi_n(i) - 1 & \text{ if $1\leq i\leq k-1$, } \\
                 \Pi_n(i+1)-1 & \text{ if $k \leq i\leq n-1$. }
                \end{cases}
 $$
 
 \noindent
 Now observe that, still conditional on $\Pi_n^{-1}(1)=k$, we have 
 
 $$ \Sigma_n^*(i) = \begin{cases} W_1 & \text{ if $i=1$, } \\
                \text{ the $W_i$-th element of $[n-1] \setminus \{W_1,\dots,W_{i-1}\}$} & \text{ otherwise.}
               \end{cases},
 $$
 
 \noindent
 It follows that 
 
 $$\left(\Sigma_n^* \Big|  \Pi_n^{-1}(1)=k \right) \isd \Mallows(n-1,q). $$
 
 This gives
 
 $$ \begin{array}{rcl}
    \left(\Sigma_n{\Big|} \Pi_n^{-1}(1)=k\right) 
     & = & \left(\rk(\Pi_{n}(1),\dots,\Pi_{n}(k-1) ) {\Big|}  \Pi_n^{-1}(1)=k \right) \\[2ex]
     & = & \left(\rk(\Sigma_n^*(1),\dots,\Sigma_n^*(k-1) ) {\Big|}  \Pi_n^{-1}(1)=k \right) \\[2ex]
     & \isd & \Mallows(k-1,q),
    \end{array}
$$ 

\noindent
applying Lemma~\ref{lem:prefix_mallows} in the last line. 
 \end{proof}

 We next establish that $i < \Pi_n^{-1}(1)$ is in fact expressible in $\TOTO$.
 This is a seemingly innocuous statement.
 As a side remark, let us mention that it is not possible to express
 that $i < \Pi_n(1)$ in $\TOTO$
 (since it is not relevant for the proofs in this paper, we do not offer a formal proof).
 This is the reason we use $\Pi_n^{-1}(1)$ and not $\Pi_n(1)$ to define $\Sigma_n$ above.
 
 \begin{lem}\label{lem:chi}
 There is a $\TOTO$ formula $\chi$ with one free variable, such that 
 for ever $n\in\eN$, every permutation $\pi\in S_n$ and every $1\leq i \leq n$ we have 
 
 $$ (\pi,i) \models \chi(x) \text{ if and only if } i < \pi^{-1}(1). $$
 
 \end{lem}

 \begin{proof} The formula $\chi$ can be written as:
 
 $$ \exists y : (x <_1 y ) \wedge \neg \left(\exists z : z <_2 y \right). $$
 
 \noindent
 So, the formula $\chi$ asks that there exists $y$ with $x<y$ such that 
 there is no $z$ with $\pi(z) < \pi(y)$. This clearly does the trick. 
 \end{proof}

 We now present a final preparatory $\TOTO$ sentence whose probability does not convergence 
 on part of the target range for $q=q(n)$.
 
 \begin{cor}\label{cor:bijnafinal}
  There exists a sentence $\zeta \in \TOTO$ such that for every fixed $c>0$, all 
  sequences $q=q(n)$ satisfying $1-1/\log^* n < q < 1+ c/n$ 
  and all $n$ satisfying $W^{(2)}(\log^{**}\log^{**}n-1) < \log^*\log^*\log^* n$ we have 
  
  $$ \Pee( \Pi_n \models \zeta ) = \begin{cases}
                                    1-o(1) & \text{ if $\log^{**}\log^{**} n$ is even, }\\
                                    o(1) & \text{ if $\log^{**}\log^{**} n$ is odd }
                                   \end{cases},
  $$
  
  \noindent
  where the error terms $o(1)$ can be taken uniform over all sequences considered.
 \end{cor}

 \begin{proof}
 Let $\psi$ is as provided by Proposition~\ref{prop:q1/n} and let $\chi$ be as provided by Lemma~\ref{lem:chi}.
 We let $\zeta := \psi^\chi$ be the relativization as provided by Lemma~\ref{lem:relativization}.
 Then $\zeta$ formalizes that $\Sigma_n$ satisfies $\psi$.
 
  Now let $\eps>0$ be arbitrary (but fixed) and let $n$ be such that 
  $W^{(2)}(\log^{**}\log^{**}n-1) < \log^*\log^*\log^* n$. 
 Let $\delta=\delta(c,\eps)$ be as provided by Lemma~\ref{lem:Pi1bdd} and let $c_1=c_1(c,\eps),c_2=c_2(c,\eps)$ 
 be as provided by Lemma~\ref{lem:large_alpha}. We define
 
 $$ k_1 := \begin{cases} \delta n & \text{ if $q > 1-c/n$, }\\
            \frac{c_1}{1-q} & \text{otherwise.}
           \end{cases}, \quad 
           k_2 := \begin{cases} n & \text{ if $q > 1-c/n$, }\\
            \frac{c_2}{1-q} & \text{otherwise.}
           \end{cases}. $$
           
\noindent
A  crucial step in the current proof is the observation that, provided $n$ is sufficiently large, 
for all $k_1\leq k \leq k_2$ we have 

\begin{equation}\label{eq:qnew} 1 - \frac{C}{k-1} < q < 1 + \frac{C}{k-1}, \end{equation}

\noindent
where we can take $C := 2c/\delta$ in the case when $q > 1-c/n$ and 
$C := 2/c_1$ otherwise.

If $\log^{**}\log^{**} n$ is even then, using Lemma~\ref{lem:sigmacond}, we can write

$$ \begin{array}{rcl} 
\Pee( \Pi_n \models \zeta )
& = & \Pee( \Sigma_n \models \psi ) \\[2ex]
& \geq & \displaystyle \sum_{k_1 \leq k \leq k_1} \Pee( \Sigma_n \models \psi | \Pi_n^{-1}(1) = k )\cdot
\Pee( \Pi_n^{-1}(1)=k) \\[2ex]
& = & \displaystyle \sum_{k_1 \leq k \leq k_2} \Pee( \Mallows(k-1,q) \models \psi  )\cdot\Pee( \Pi_n^{-1}(1) = k ) \\[2ex]
& \geq & \displaystyle 
\min_{k_1\leq k \leq k_2} \Pee( \Mallows(k-1,q) \models \psi  ) - \Pee( \Pi_n^{-1}(1) < k_1 \text{ or } 
\Pi_n(1) > k_2 ) \\[2ex]
& > & 1 - O\left( \left(\log^* n\right)^{-.001} \right) - \eps, 
\end{array} $$

\noindent
using in the last line Proposition~\ref{prop:q1/n} together with~\eqref{eq:qnew}
and that the fact that for all $k \geq k_1$ we have 

$$ \begin{array}{rcl} \log\log\log k  & \geq & \log\log\log( c_1\cdot\log^* n ) \geq \log^*\log^*\log^* n  \\
& > & W^{(2)}(\log^{**}\log^{**}n-1) \geq W^{(2)}(\log^{**}\log^{**}k-1), \end{array} $$

\noindent
for $n$ sufficiently large.
If $\log^{**}\log^{**} n$ is odd then, analogously:

$$ \begin{array}{rcl} 
\Pee( \Pi_n \models \zeta ) & \leq & \displaystyle 
\max_{k_1\leq k \leq k_2} \Pee( \Mallows(k-1,q) \models \psi  ) + \Pee( \Pi_n^{-1}(1) 
< k_1 \text{ or } 
\Pi_n^{-1}(1) > k_2 ) \\[2ex]
& < & \displaystyle 
O\left( \left(\log^* n\right)^{-.001} \right) + \eps.  
\end{array} $$

\noindent
Sending $\eps\searrow 0$ finishes the proof.
 \end{proof}

 The sentence $\zeta$ provided in the previous corollary works in an asymmetric range of $q$.
 In order to remedy this, we will use a construction that allows us to distinguish (in $\TOTO$, with appropriately 
 high probability) between the case when $q < 1-c/n$ and $q > 1+c/n$.
 The following observation is the key to that construction.

\begin{lem}\label{lem:Pi1Pin}
 For every $\eps>0$ there are $c=c(\eps), n_0=n_0(\eps)$ such the following holds 
 for all $n\geq n_0$;
 \begin{enumerate}
  \item\label{itm:Pi1Pin1} If $0<q<1-c/n$ then $\Pee( \Pi_n(1) < \Pi_n(n) ) > 1-\eps$, and;
  \item\label{itm:Pi1Pin2} If $q>1+c/n$ then $\Pee( \Pi_n(1) < \Pi_n(n) ) < \eps$.
 \end{enumerate}
\end{lem}

\begin{proof}
 We first consider the case $q<1-c/n$.
 By Theorem~\ref{thm:bhat_peled} and Markov's inequality:
 
 $$ \Pee( \Pi_n(1) > n/10 ) \leq \Pee( |\Pi_n(1)-1| > n/100 ) 
 \leq \frac{100 \cdot \Ee|\Pi_n(1)-1|}{n}
 \leq \frac{200 q}{n(1-q)} \leq \frac{12}{c} < \eps/2, $$
 
 \noindent
 where the first inequality holds for $n$ sufficiently large and 
 the last inequality holds provided the constant $c$ is chosen sufficiently large.
 Analogously
 
 $$ \Pee( \Pi_n(n) < 9n/10 ) \leq \frac{100 \cdot \Ee|\Pi_n(n)-n|}{n} < \eps/2. $$
 
 \noindent
 Part~\ref{itm:Pi1Pin1} follows.
 
 Assume now $q > 1+c/n$ for some large constant $c$ to be determined.
 Recall that $r_n \circ \Pi_n \isd \Mallows(n,1/q)$.
 We have $1/q < 1/(1+c/n) < 1 - (c/2)/n$, where the last inequality holds for $n$ sufficiently large.
 Using the previous part:
 
 $$ \Pee( \Pi_n(1) < \Pi_n(n) ) = \Pee( r_n \circ \Pi_n(1) > r_n \circ\Pi_n(n) ) < \eps, $$
 
 \noindent
 for $n$ sufficiently large (and provided we chose $c$ sufficiently large).
\end{proof}

\begin{proofof}{Proposition~\ref{prop:final}}
 Let $\zeta$ be as provided by the previous corollary, let $\zeta^{\text{reverse}}$ be as provided by 
Lemma~\ref{lem:reverse_formula}, and let

$$\theta :=  \exists x\exists y : (x <_2 y) \wedge \neg\left(\exists z : z <_1 x \right) \wedge \neg\left(
\exists z : y <_1 z \right), $$

\noindent
denote a $\TOTO$ sentence that formalizes that $\Pi_n(1) < \Pi_n(n)$.
The sentence $\xi$ is now given by

$$ \xi := \left( \theta \wedge \zeta \right) \vee \left( \neg\theta \wedge \zeta^{\text{reverse}} \right). $$

That is, $\xi$ demands that one of {\bf a)} $\Pi_n(1) < \Pi_n(n)$ and $\zeta$ holds, or {\bf b)} 
$\Pi_n(1)>\Pi_n(n)$ and $\zeta$ holds for $r_n\circ\Pi_n$.

Let $\eps>0$ be arbitrary (but fixed) and let $c>0$ be as provided by Lemma~\ref{lem:Pi1Pin}.
We consider an arbitrary sequence $q=q(n)$ with $1-1/\log^* n < q < 1+1/\log^* n$.
Let $n$ be such that $W^{(2)}\left(\log^{**}\log^{**}n - 1\right) < \log^*\log^*\log^* n$.
First we assume that $\log^{**}\log^{**}n$ is even and $1-c/n<q<1+c/n$. In this case 

$$ \begin{array}{rcl} 
\Pee( \Pi_n \models \psi ) 
& \geq & 
\Pee( \Pi_n \models \zeta \text{ and } r_n\circ \Pi_n \models \zeta ) \\
& \geq & 1 - \Pee( \Pi_n \not\models \zeta ) - \Pee( r_n \circ \Pi_n \not\models \zeta ) \\
& = & 1 - o(1), 
\end{array} $$
 
 If  $\log^{**}\log^{**}n$ is odd and $1-c/n<q<1+c/n$ then similarly

$$ 
\Pee( \Pi_n \models \psi ) \leq \Pee( \Pi_n \models \zeta ) + \Pee( r_n \circ \Pi_n \models \zeta ) = o(1).
 $$
 
 If $\log^{**}\log^{**}n$ is even and $1-1/\log^* n < q \leq 1-c/n$ then we have
 
 $$ \Pee( \Pi_n \models \psi ) \geq \Pee( \Pi_n \models \zeta ) - \Pee( \Pi_n(1) > \Pi_n(n) )
 > 1-o(1)-\eps. $$
 
 If $\log^{**}\log^{**}n$ is odd and $1-1/\log^* n < q \leq 1-c/n$ then 
 
 $$ \Pee( \Pi_n \models \psi ) \leq \Pee( \Pi_n \models \zeta ) + \Pee( \Pi_n(1) > \Pi_n(n) )
 < o(1)+\eps. $$
 
 Now  suppose that $\log^{**}\log^{**}n$ is even and $1+c/n \leq q < 1+1/\log^* n$.
 Observe that $1-1/\log^* n < 1/q < 1-(c/2)/n$ for $n$ sufficiently large.
 Adjusting the value of $c$ if needed, we can again apply Corollary~\ref{cor:bijnafinal} to obtain:
 
 $$ \Pee( \Pi_n \models \psi ) 
 \geq \Pee( r_n \circ \Pi_n \models \zeta ) - \Pee( \Pi_n(1) < \Pi_n(n) )
 > 1-o(1) - \eps. $$
 
 Similarly, if $\log^{**}\log^{**}n$ is odd and $1+c/n \leq q < 1+1/\log^* n$ then 
 
 $$ \Pee( \Pi_n \models \psi ) 
 \leq \Pee( r_n \circ \Pi_n \models \zeta ) + \Pee( \Pi_n(1) < \Pi_n(n) )
 < o(1) + \eps. $$
 
 \noindent
 The result follows by sending $\eps\searrow 0$.
\end{proofof}

\section{Discussion and suggestions for further work\label{sec:discuss}}

Theorem~\ref{thm:logical_limits} established that for $q<1$ the zero-one law holds wrt.~$\TOOB$, while for 
$q>1$ there is some $\TOOB$ sentence whose probability of holding does not converge with $n$.
(In the case when $q=1$ Compton~\cite{Compton89II} has shown decades ago that the convergence law holds, but 
not the zero-one law.)
The property exhibiting the non-convergence simply asks for at least a certain number of fixed points, 
exploiting a result from~\cite{MullerVerstraaten} (repeated as Theorem~\ref{thm:odd_cycles} in the present paper). 
Having another look at the result on the limiting distributions 
for the cycle counts provided by that result, it should be relatively straightforward to show that 
when $q>1$ and we consider either the 
sequence $\Mallows(2n,q)$ of Mallows permutations restricted to 
even domain sizes or the sequence $\Mallows(2n+1,q)$ restricted to odd domain sizes, then 
in fact the convergence law will hold.

Theorem~\ref{thm:logical_limits} only considers the case of $q$ fixed. 
The case when $q=q(n)$ depends on $n$ is certainly worth investigating, in particular the 
``phase transition'' regime when $q \to 1$.
A natural line of inquiry would be to try and determine a 
``critical window'' inside which the behaviour is essentially the same as $q=1$ (convergence law but not 
zero-one law) while below we get the zero-one law and above the convergence law fails.
The behaviour of $\TOOB$ sentences is essentially determined by the 
limiting distribution of (short) cycles, so one might expect this to be relatively straightforward.
The behaviour of cycles in the regime when $q\to 1$ is however far from well understood at the moment.
We point the reader to~\cite{MullerVerstraaten} for some more discussion of the relevant open questions
on cycles in the Mallows model.

In Theorem~\ref{thm:toto}, Part~\ref{part:toto_non_convergence} we did in fact consider a range 
of $q$ with $q\to 1$.
We constructed a sequence $\varphi \in \TOTO$ such that $\Pee( \Mallows(n,q) \models \varphi )$ 
does not converge for any sequence $q=q(n)$ with $1-1/\log^* n < q < 1+1/\log^* n$.
For the $\varphi$ we've constructed, we believe it should be possible 
to determine a sequence $q=q(n)$ that converges to 1 (very slowly) while $\Pee( \Mallows(n,q) \models \varphi )$
converges.

On the other hand, as mentioned earlier, the range $(1-1/\log^* n,1+1/\log^* n)$ is certainly not best possible, as 
should evident from the proof.
It should be possible to achieve a range for $q$ that approaches $1$ much more slowly, with a minor variation 
on our proof. The only constraint on the range we can obtain seems to be what we can axiomatize using the 
arithmetization construction in Section~\ref{sec:arithSO}.
That is, if we can axiomatize some increasing function $F(n)$ using binary relations 
as in Section~\ref{sec:arithSO} and $f(n)$ is its 
discrete inverse (i.e.~$f(n) := \max\{m\leq n : F(m) \leq n \}$), then one should be able to construct 
a $\varphi$ such that $\Pee( \Mallows(n,q) \models \varphi )$ does not converge for all $1-1/f(n) < q < 1+1/f(n)$.
It is not quite clear to the authors at present, most likely just for want of a proper background in 
the relevant subfields of Mathematics, exactly how wide of a range can be obtained in this manner.
But, perhaps a different proof technique can yield even wider ranges.
In particular, perhaps one can widen the range of $q$ where there is non-convergence even more 
if one does not insist on finding a single sentence that works in the entire range?

\subsection*{Acknowledgements}

We thank Kevin Compton, Valentin F\'{e}ray, Tomasz {\L}uczak, Marc Noy and Joel Spencer for helpful 
discussions and pointers to the literature.
We thank Pieter Trapman and Rineke Verbrugge for helpful and very detailed comments on the chapter of 
TWV's PhD thesis on which this 
paper is based

\bibliographystyle{plain} 
\bibliography{bibfile}

\end{document}